\documentclass[reqno,11pt]{amsart}
\usepackage{amsmath, latexsym, amsfonts, amssymb, amsthm, amscd}
\usepackage{graphics,epsf,psfrag}
\usepackage{graphics,epsf,psfrag}

\usepackage{color}

\numberwithin{equation}{section}

\newtheorem{theorem}{Theorem}[section]
\newtheorem{lemma}[theorem]{Lemma}
\newtheorem{proposition}[theorem]{Proposition}
\newtheorem{cor}[theorem]{Corollary}
\newtheorem{rem}[theorem]{Remark}

\DeclareMathOperator{\sign}{\mathrm{sign}}

\newcommand{\dd}{\mathrm{d}}
\newcommand{\ind}{\mathbf{1}}

\renewcommand{\tilde}{\widetilde}
\renewcommand{\hat}{\widehat}

\newcommand{\cA}{{\ensuremath{\mathcal A}} }

\newcommand{\cJ}{{\ensuremath{\mathcal J}} }

\newcommand{\cE}{{\ensuremath{\mathcal E}} }

\newcommand{\cN}{{\ensuremath{\mathcal N}} }

\newcommand{\cT}{{\ensuremath{\mathcal T}} }

\newcommand{\cZ}{{\ensuremath{\mathcal Z}} }

\newcommand{\bP}{{\ensuremath{\mathbf P}} }
\newcommand{\bE}{{\ensuremath{\mathbf E}} }

\DeclareMathSymbol{\leqslant}{\mathalpha}{AMSa}{"36} 
\DeclareMathSymbol{\geqslant}{\mathalpha}{AMSa}{"3E} 
\DeclareMathSymbol{\eset}{\mathalpha}{AMSb}{"3F}     

\newcommand{\Var}{\mathrm{Var}}        
\newcommand{\maxtwo}[2]{\max_{\substack{#1 \\ #2}}} 
\newcommand{\sumtwo}[2]{\sum_{\substack{#1 \\ #2}}} 
\newcommand{\prodtwo}[2]{\prod_{\substack{#1 \\ #2}}}     

\newcommand{\bbB}{{\ensuremath{\mathbb B}} }
\newcommand{\bbC}{{\ensuremath{\mathbb C}} }

\newcommand{\bbE}{{\ensuremath{\mathbb E}} }

\newcommand{\bbH}{{\ensuremath{\mathbb H}} }

\newcommand{\bbN}{{\ensuremath{\mathbb N}} }

\newcommand{\bbP}{{\ensuremath{\mathbb P}} }

\newcommand{\bbR}{{\ensuremath{\mathbb R}} }

\newcommand{\bbZ}{{\ensuremath{\mathbb Z}} }

\newcommand{\ga}{\alpha}
\newcommand{\gb}{\beta}
\newcommand{\gd}{\delta}

\newcommand{\gep}{\varepsilon}       
\newcommand{\gp}{\varphi}
\newcommand{\gr}{\rho}
\newcommand{\gvr}{\varrho}

\newcommand{\gG}{\Gamma}

\newcommand{\gD}{\Delta}

\newcommand{\go}{\omega}

\newcommand{\gl}{\lambda}
\newcommand{\gL}{\Lambda}
\newcommand{\gs}{\sigma}

\makeatletter
\def\captionfont@{\footnotesize}
\def\captionheadfont@{\scshape}

\long\def\@makecaption#1#2{%
  \vspace{2mm}
  \setbox\@tempboxa\vbox{\color@setgroup
    \advance\hsize-6pc\noindent
    \captionfont@\captionheadfont@#1\@xp\@ifnotempty\@xp
        {\@cdr#2\@nil}{.\captionfont@\upshape\enspace#2}%
    \unskip\kern-6pc\par
    \global\setbox\@ne\lastbox\color@endgroup}%
  \ifhbox\@ne 
    \setbox\@ne\hbox{\unhbox\@ne\unskip\unskip\unpenalty\unkern}%
  \fi
  \ifdim\wd\@tempboxa=\z@ 
    \setbox\@ne\hbox to\columnwidth{\hss\kern-6pc\box\@ne\hss}%
  \else 
    \setbox\@ne\vbox{\unvbox\@tempboxa\parskip\z@skip
        \noindent\unhbox\@ne\advance\hsize-6pc\par}%
\fi
  \ifnum\@tempcnta<64 
    \addvspace\abovecaptionskip
    \moveright 3pc\box\@ne
  \else 
    \moveright 3pc\box\@ne
    \nobreak
    \vskip\belowcaptionskip
  \fi
\relax
}
\makeatother
\def\writefig#1 #2 #3 {\rlap{\kern #1 truecm
\raise #2 truecm \hbox{#3}}}

\newcommand{\tf}{\textsc{f}}

\newcommand{\ubgb}{\overline{\gb}}

\title[Disorder relevance for LGFF  pinning model]{Pinning and disorder relevance \\
for the lattice Gaussian free field}

\begin{document}

\author{Giambattista Giacomin}
\address{
  Universit\'e Paris Diderot, Sorbonne Paris Cit\'e,  Laboratoire de Probabilit{\'e}s et Mod\`eles Al\'eatoires, UMR 7599,
            F- 75205 Paris, France
}

\author{Hubert Lacoin}
\address{
  IMPA, Institudo de Matem\`atica Pura e Aplicada, Estrada Dona Castorina 110
Rio de Janeiro, CEP-22460-320, Brasil. 
}

\begin{abstract}
  This paper provides a rigorous study of the localization transition for a  Gaussian free field on $\bbZ^d$
  interacting with a quenched disordered substrate that acts on the interface when its height is close to zero. 
  The substrate has the tendency to localize or repel the interface at different sites and one can show that a localization-delocalization 
  transition takes place when varying the average pinning potential $h$: the free energy density is zero in the delocalized regime, that is for $h$ smaller than a threshold $h_c$,  and it is positive for $h>h_c$. 
  For $d\ge 3$  we compute $h_c$ 
  and we show that the transition happens at the same value as for the annealed model. 
  However we can show that the critical behavior of the quenched model differs from the one of the annealed one. 
   While the phase transition of the annealed model is of first order,
  we show that the quenched free energy is bounded above by  $ ((h-h_c)_+)^2$ times a positive constant and that,
  for Gaussian disorder, the quadratic behavior is sharp. Therefore this  provides an example in which a {\sl relevant disorder critical exponent}
  can be made explicit: in theoretical physics disorder is  said to be {\sl relevant} when the disorder changes the critical behavior of a system and, while there are cases in which it is known that disorder is relevant, the exact critical behavior is typically unknown.
  For $d=2$  we are not able to decide whether the quenched and annealed critical points coincide, 
  but we provide an upper bound for the difference between them.
   \\[10pt]
  \\[10pt]
\end{abstract}

\maketitle

\tableofcontents

\section{Introduction}

A central question in  statistical mechanics is 
understanding the effect of disorder on phase transitions and critical phenomena. This issue has been raised soon after 
Lars Onsager's celebrated solution of the two dimensional Ising model with zero external field, see  \cite[Section~5.3]{cf:G}
for a historical overview and references. The model solved by Onsager 
has constant couplings 
-- Onsager's solution actually allows couplings that are different in the horizontal and vertical directions, but not more than that -- 
 and the question of whether this result 
withstands, and to which extent, the introduction  of impurities emerged as a compelling {\sl stability} issue. 
Modeling systems with impurities naturally led to considering 
systems in which the interaction terms, for example the potentials between nearest neighbor spins, are chosen by sampling a random
field -- which we call {\sl disorder} -- with good ergodic properties, often even a field of independent identically distributed random variables.
One then tries to understand
the  properties of the arising statistical mechanics system which is no longer translation invariant, even if it retains some translation invariance  in a statistical sense. Some basic results like existence of the thermodynamic limit and the fact that observables are self-averaging (i.e., independent on the sample of the disorder) can be established \cite{cf:Bovier}. When 
the variance of the disorder tends to zero the system approaches the non disordered, or {\sl pure},  case but transferring 
a result proven for the pure system to the disordered case, even if the disorder is very weak, is far from being straightforward.

As a matter of fact, the first  arguments set forth pointed toward predicting that even a very low amount of disorder would wash out completely the phase transition \cite[Section~5.3]{cf:G}. Only  later on a substantially richer picture emerged. Since the Ising model has to a certain extent driven the progress, it 
is  worthwhile recalling that a disorder in the form of an external random field makes the Ising transition disappear in two dimensions \cite{cf:AW}, while the transition persists in $d\ge 3$ \cite{cf:BK}. On the other hand, it is not difficult to realize
that
 introducing a disorder in the
coupling potentials, for example by introducing a dilution,  may in general modify the precise value of the critical point,
but preserves the existence of a transition:
the nature of the transition (for example, the critical exponents) is however still an open question (at least in low dimensions) \cite[Section~5.3]{cf:G}. Giving more details on this beautiful issue is beyond our scope,  but what interests us the most is 
that A.~B.~Harris \cite{cf:Hcrit} introduced an intriguing way of looking at the problem and proposed a surprisingly easy criterion to predict whether  a small amount of disorder modifies the critical 
behavior with respect to that of the pure system (assuming that the transition 
persists).   Essentially, Harris criterion says that if the transition for the pure system is sufficiently smooth, a small or a moderate amount of disorder does 
not modify the transition: this is the case of {\sl irrelevant disorder}. When the Harris criterion fails, one expects to be
in a {\sl relevant disorder} case, except possibly at the boundary between these two situations where the analysis is trickier ({\sl marginal disorder}). 
These notions of relevant, irrelevant and marginal disorder are  connected to the framework in which the theory has been
developed, that is {\sl renormalization} \cite{cf:Bovier} and the idea behind Harris' approach is that disorder may be downsized or enhanced by the renormalization transformation, leading, after repeated application of the transformation,  in the first case to the pure system fixed point and, in the second case, to a different one or to no fixed point at all \cite{cf:DR,cf:G,cf:Hcrit}.

One of the substantial obstacles to the mathematical exploration of the Harris criterion is that  a good understanding of  the critical
properties of  pure systems is limited to very special cases. 
But in the 
last twenty years this question has been addressed, first by theoretical physicists (e.g. \cite{cf:DHV} and references therein) and then by mathematicians, 
for a simple model
of one dimensional 
interface interacting with a substrate: the random walk (RW) disordered pinning model (see \cite{cf:GB, cf:G}).
For this model the interface is given by the graph of a random walk which takes random energy 
rewards when it touches a defect line. The random walk can  be very general and the full class of RW pinning models is better apprehended if viewed in terms of {\sl renewal  pinning}:  we refer  the interested reader to the introductions of \cite{cf:G,cf:GB}. The pure system has the remarkable quality of being what physicists call 
{\sl solvable}, meaning that there exists an explicit expression for the  free energy \cite{cf:Fisher}. 
All the results which have been obtained confirm the validity of the Harris criterion and its interpretation for the RW pinning
model \cite{cf:Ken,cf:KZ, cf:QH2, cf:CSZpin,cf:CDenH,cf:DGLT,cf:G,cf:GLT2,cf:GT_cmp,cf:L,cf:Trep}.

A natural generalization of the RW pinning  model is obtained by replacing the graph of the random walk 
by a random surface, and one of the first natural choices is  the Lattice Gaussian Free Field (LGFF) -- recently  called also  Discrete Gaussian Free Field  -- on a subset, for example an (hyper)-cube, of 
$\bbZ^d$, $d\ge 2$. 
 While the pure  model is not exactly solvable in that case, it has been  studied
 and the nature of the phase transition is well known \cite{cf:BV,cf:Vel}.
However, until now very few attempts have been made to understand the quenched behavior of the system (see \cite{cf:CM1,cf:CM2}).

Our model has two parameters: the noise intensity $\gb\ge 0$ and the average pinning strength $h \in \bbR$. 
In this paper, we describe completely the characteristics of the phase transition in the case $d\ge 3$ and the results can be summed up as follows:

\begin{enumerate}
\item
We identify the disordered critical point $h_c=h_c(\gb)$. More precisely, with the  choice of the parameters we make, which is 
the same as the one adopted for the RW pinning in the mathematical literature,
the  critical point of the disordered (i.e. quenched) model and the one of the annealed model  coincide.  
However, the critical behaviors do not and this contrasts sharply with what happens for RW pinning 
where, except for the marginal disorder case for which the question is open, coincidence of critical points happens if and only if  the critical behaviors coincide. 
We stress also that, with our choice, the annealed model coincides with the one in which we simply switch off the disorder by setting its intensity $\gb$ to zero 
 and this is what we call {\sl pure} model. 
\item 
The free energy density, or just free energy for conciseness, is zero for $h\le h_c(\gb)$ and it is positive for $h> h_c(\gb)$.
We prove in full generality (in the choice of the disorder) that the free energy  is
$O\left( (h-h_c(\gb))^2\right)$ as 
$h \searrow h_c(\gb)$, which implies that the first derivative of the free energy is continuous at
$h_c(\gb)$: this is what is usually called a {\sl second order} transition.
The transition  for the pure system instead  is of first order, i.e. the first derivative of the free energy is discontinuous (it has a jump) at $h_c(\gb)$,  hence, in the Harris' sense, disorder is relevant. 
\item 
When disorder is Gaussian we show that the behavior of the free energy at criticality is precisely quadratic and the critical exponent associated to the free energy is therefore precisely identified.  
Harris' theory yields no prediction of how the critical properties are modified when disorder is relevant. As a matter 
of fact capturing the critical exponent of transitions in disorder relevant cases appears  to be a major challenge
and the authors do not know of any rigorous results in this direction when the disorder is weakly correlated
(for strongly correlated environment see \cite{cf:QB,cf:QH}). Even in the RW pinning models several contrasting conjectures have been set forth, but a certain consensus appears to emerge in favor of an infinite order transition, i.e.  $C^\infty$ regularity of the the free energy  at
the critical point  (see the review of the literature in \cite[Section~5.3]{cf:G}  to which one should add the recent contribution \cite{cf:DR}). 
\end{enumerate}

\medskip

We present also results for $d=2$, but we are unable to show that disordered and pure critical points coincide and,
as a consequence, we are unable to establish results on the critical behavior. Finally, we have also a quick look at 
the higher dimensional analog of the problem of a {\sl copolymer near an interface between selective solvents}.

\section{Model and results} \label{modelresult}

\subsection{The disordered model}
Given $\Lambda$ be a finite subset of $\bbZ^d$, we
 let $\partial \gL$ denote the internal boundary of $\Lambda$, 
$\mathring{\gL}$ the set of interior points of $\gL$ and 
 and $\partial^- \gL$ 
 the set of interior points  that are in contact with the boundary.
\begin{equation}\label{boundary}\begin{split}
\partial \gL&:=\{x \in \gL : \,  \exists y\notin \gL, \ x\sim y  \},\\
\mathring{\gL}&:=\gL \setminus \partial \gL,\\
\partial^{-} \gL&:=\{x \in \mathring{\gL}  :\,   \exists  y\in \partial \gL , \ x\sim y  \}.
\end{split}\end{equation}
In general some of these sets could be empty, but throughout this work $\gL$ is going to be a large 
(hyper-)cube. 

Given $(\hat \phi_x)_{x\in \bbZ^d}$ a real valued field, one defines 
$\bP_{\gL}^{\hat \phi}$ to be the law of the lattice Gaussian free field on $\Lambda$ 
(denoted by $\phi=(\phi_x)_{x\in \Lambda}$) with boundary conditions $\hat \phi$ on $\partial \gL$. Formally we set 
\begin{equation}
 \phi_x\,:=\, \hat \phi_x \quad  \text{ for every } x \in  \partial \gL_N,
\end{equation}
and consider $\bP_\gL^{\hat \phi}$ as a measure on $\bbR^{\mathring{\gL}}$ whose density is given by
\begin{equation}
\label{density}
\bP_\gL^{\hat \phi}(\dd \phi)=\frac{1}{\mathcal Z_{\gL}^{\hat \phi}}
  \exp\left(-\frac 1 2 \sumtwo{(x,y)\in (\gL)^2 \setminus (\partial \gL)^2 }{x\sim y}\frac{ (\phi_x-\phi_y)^2 }{2} \right)
\prod_{x\in \mathring{\gL}} \dd \phi_x \, ,
\end{equation} 
where $\prod_{x\in \mathring{\gL}} \dd \phi_x$ denotes the Lebesgue measure on $\bbR^{\mathring{\gL}}$ and
\begin{equation}\label{eq:defcalz}
\mathcal Z^{\hat \phi}_{\gL}:= \int_{\bbR^{\mathring{\gL}}} 
\exp\left(-
 \frac 1 2 \sumtwo{(x,y)\in (\gL)^2 \setminus (\partial \gL)^2 }{x\sim y}
\frac{ (\phi_x-\phi_y)^2 }{2} \right) \prod_{x\in\mathring{\gL}} \dd \phi_x \, .
\end{equation}
For the particular case $\hat \phi\equiv u$ we write $\bP^u_\gL$, and $\bP_\gL$ when $u=0$. One of the factors $1/2$ in the exponential is present to 
compensate for the fact that each edge is counted twice.

\medskip

In what follows we consider mostly the case
$\Lambda=\Lambda_N:=\{0,\dots,N\}^d$, for some $N\in \bbN$.
Note that in this case we have  $\mathring{\gL}_N:= \{1,\dots, N-1\}^d$.
We also introduce the notation $\tilde \gL_N:= \{1,\dots, N\}^d$.
We simply write $\cZ_N$ and $\bP_N$ for $\cZ_{\gL_N}$ and $\bP_{\gL_N}$.

\medskip

Given $\go=\{\go_x\}_{x \in \bbZ^d}$ a family of  IID square integrable centered random variables (of law $\bbP$), we set
\begin{equation}
\label{eq:defgl}
 \gl(\gb)\, :=\, \log \bbE[e^{\gb \go_x}]\, ,
\end{equation}  
and assume that there exists $\ubgb\in (0, \infty]$ such that 
\begin{equation}
\label{eq:assume-gl}
 \max(\gl(2\gb), \gl(-\gb))\, < \,  \infty\ \text{ for  every } \gb \in  (0, \ubgb)\, .
 \end{equation}
Many of the arguments rely only on $\gl(\gb)< \infty$:   $\gl(2\gb)< \infty$ is related to
two replica arguments  (lower bounds) and $\gl(-\gb)< \infty$ is exploited when fractional moments
estimates are performed (upper bounds) and a look at the proof of Proposition~\ref{rounding} 
suffices to see that this second  requirement can be relaxed. Moreover, a part of the results
are given for Gaussian $\go$ and in that case $\ubgb=\infty$.
Note that \eqref{eq:assume-gl} implies smoothness of $\gl(\cdot)$ for $\gb\in (-\ubgb, 2\ubgb)$ 
and around zero
\begin{equation}
\gl(\gb)\, =\, \frac{\gb^2}2+O(\gb^3)\, .
\end{equation}
For $x\in \gL_N$  set  $\delta_x:= \ind_{[-1,1]}(\phi(x))$.
For $\gb>0$ and $h\in \bbR$, we define  a modified measure $\bP_{N, h}^{ \go, \gb , \hat \phi}$ via
\begin{equation}
\label{eq:modmeas}
\frac{\dd \bP^{\go,\gb,\hat \phi}_{N,h}}{\dd \bP^{\hat \phi}_N}=\frac{1}{Z^{\gb,\go,\hat \phi}_{N,h}}\exp\left( \sum_{x\in  \tilde \gL_N} (\gb \go_x-\gl(\gb)+h)\delta_x\right)\, ,
\end{equation}
where
\begin{equation}
\label{eq:modZ}
Z^{\gb,\go,\hat \phi}_{N,h}:=\bE^{\hat \phi}_N\left[ \exp\left( \sum_{x\in  \tilde \gL_N} (\gb \go_x-\gl(\gb)+h)\delta_x\right)\right].
\end{equation}
Note that in the definition of $\bP^{\go,\gb,\hat \phi}_{N,h}$, the sum $\left(\sum_{x\in  \tilde \gL_N}\right)$ can be replaced by either 
$\left(\sum_{x\in  \gL_N}\right)$ or $\left(\sum_{x\in  \mathring{\gL}_N}\right)$
as these changes affect only the partition function. We have chosen to sum over $\tilde \gL_N$ for super-additivity reasons (see Proposition \ref{superadd}).
The superscript $\hat\phi$ is dropped when $0$ boundary conditions are considered and replaced by $u$ when $\hat \phi \equiv u$.

\subsection{The pure model}
\label{sec:puremodel}
The natural homogeneous model associated to the disordered model  $\bP_{N, h}^{ \go, \gb , \hat \phi}$ can be obtained by switching off the
disorder: the pure model is therefore precisely $\bP_{N, h}^{ \go, 0 , \hat \phi}$ but the notation is heavy and a bit misleading because 
the measure does not depend on $\go$. Moreover our choice of the parametrization is such that the pure model coincides with the annealed model, 
that is with the model with partition function $ \bbE \left[Z^{\gb,\go,\hat \phi}_{N,h}\right]$.  For the pure model we use the notation $\bP_{N,h}$ and 
we limit ourselves to the case
$\hat \phi\equiv 0$:
\begin{equation}
\frac{\dd \bP_{N,h}}{\dd \bP_N}\, =\, \frac{1}{Z_{N,h}}\exp\left( h \sum_{x\in  \mathring{ \gL}_N} \delta_x\right)\, .
\end{equation}
It is  very easy to see -- the proof is detailed just below -- that this model has a transition at $h=0$, in the sense that the free
energy density 
\begin{equation}\label{fdeh}
 \tf(h)\, = \, \lim_{N\to \infty} \frac 1  {N^{d}} \log Z_{N,h}\,,
\end{equation}
 satisfies 
\begin{equation}
\label{eq:2.11}
\tf(h)  \begin{cases} =0 &\text{ for } h\le 0\, ,\\
>0 & \text{ for } h>0\, ,
\end{cases}
\end{equation}
and therefore  it is not analytic at $h=0$.
 Moreover, by standard convexity arguments $\tf(h)$ is differentiable everywhere 
 except, possibly, at countably many values of $h$. 
When it exists the derivative of $\tf(h)$ is equal to
 the  {\sl asymptotic contact fraction} defined by
\begin{equation}\label{eq:contactfrac}
\lim_{N \to \infty} \frac 1  {N^{d}} \bE_{N,h}\left[\sum_{x \in \tilde\gL_N} \gd_x \right] \, .
\end{equation}
It is obvious from \eqref{eq:2.11} that the asymptotic contact fraction is $0$ for $h<0$. Moreover,  since  $\tf(\cdot)$ is convex, the asymptotic contact fraction is non decreasing and, again because of  \eqref{eq:2.11}, 
it is  positive for every $h>0$. 
\medskip

The existence of the limit \eqref{fdeh} is standard: 
the argument can be recovered from the proof in Section \ref{statbound} (it is a particular easy case). 
It is non negative because
 \begin{equation}
 \label{eq:forLM}
 Z_{N,h}\, \ge\,   \bP_N(\gp_x >1 \text{ for every } x \in \mathring \gL _N)\, ,
 \end{equation} 
 and 
 it is not difficult to show that the logarithm of 
 the latter expression is
 $o(N^d)$: this is a (rough) entropic repulsion type estimate 
 and it is an easy consequence of the continuum symmetry of the interaction
 that is broken only at the boundary  \cite{cf:LM}.  
 On the other hand  $Z_{N,h}\le 1$ for $h \le 0$, and hence $\tf(h)=0$ for $h\le 0$.
 
 \medskip

The fact that $\tf(h)>0$ for every $h>0$ can be established in a number of elementary ways 
(see Section~\ref{rem1} for  $d\ge 3$ and Remark~\ref{rem:whendis2} for $d=2$), but here we  mention 
the more refined estimate (see \cite[Fact 2.4]{cf:CM1}): for every $d=2,3, \ldots$ there exists $c_d>0$ such that
\begin{equation}
\label{eq:fact2.4}
\tf(h)  \stackrel{h \searrow 0} \sim \begin{cases} c_d \, h & \text{ for } d\ge 3\, ,
\\
c_2 \frac{h} {\sqrt{\vert \log h \vert  }}  & \text{ for } d=2\, .
\end{cases}
\end{equation}
Therefore the transition is of first order for $d\ge 3$ and the contact fraction has a jump discontinuity. Note that 
the transition is of second order  for $d=2$: the contact fraction is continuous at the transition, even if the continuity modulus 
vanishes (hence matching the behavior in the delocalized phase) only logarithmically. 

\medskip

\subsection{Some more details about the phase transition for $d\ge 3$}
\label{rem1}
The result in $d\ge 3$ is going to be particularly relevant for us and we want to stress that a rougher version of
\eqref{eq:fact2.4} is trivially established and that even the sharp statement isn't much harder. Note that
\begin{equation}
\frac1{N^d}\partial_h \log Z_{N,h} \vert _{h=0}\, =\,  \frac1{N^d}\bE_N \sum_x \gd_x\, .
\end{equation}
Now notice that $\bP_N$ is a centered Gaussian measure and the variance of $\phi_x$ under $\bP_N$
is bounded uniformly by 
 the variance of the infinite volume free field  which we denote by $\sigma_d^2$ (see Section \ref{afffaff}). 
Hence 
\begin{equation}
\label{eq:C_d-def}
\frac1{N^d}
\partial_h \log Z_{N,h} \vert _{h=0} \, \ge\, P( \gs_d \, \cN \in [-1,1])\, =:\,  C_d\, ,
\end{equation}
 where $\cN$ is a standard normal variable.

On the other hand   the same derivative is bounded above by one and therefore 
(using convexity and the fact that $Z_{N,0}=1$) for all $N$
\begin{equation}
  C_dh\, \le\,  \frac1{N^d} \log Z_{N,h}\, \le\, h\, ,
\end{equation}
and therefore for every $h\ge 0$ 
\begin{equation}
C_d h\, \le\, \tf(h)\, \le\,  h\,.
\end{equation}
 This establishes a rougher version  of \eqref{eq:fact2.4} for $d\ge 3$ (which is however a statement 
only for $h$ small). 

In fact, we have $c_d=C_d$. For this observe that if we go back to the
partition function in \eqref{eq:modZ}, but setting $\beta=0$, that is $Z^{\hat \phi}_{N, h}:= Z_{N,h}^{0, \go, \hat\phi}$,
we  readily check that for every $N$
\begin{equation}
\frac 1{(2N)^d}\log \sup_{\hat \phi} Z^{\hat \phi}_{2N, h} \, \le\,  \frac 1{N^d}\log \sup_{\hat \phi} Z^{\hat \phi}_{N, h}\, ,
\end{equation} 
from which one infers that $\tf(h) \le  \frac 1{N^d}\log \sup_{\hat \phi} Z^{\hat \phi}_{N, h}$ for every $N$. 
Now remark that for every $h>0$, we have
\begin{equation}
\partial_h \log Z^{\hat \phi}_{N, h}\, =\, \bE_{N,h}^{\hat \phi} \left[ \sum_{x\in \tilde \gL _N} \gd_x \right]
\, \le \, e^{N^d h}\bE_{N,0}^{\hat \phi} \left[ \sum_{x\in \tilde \gL _N} \gd_x \right]\le 
 e^{N^d h}\bE_{N} \left[ \sum_{x\in \tilde \gL _N} \gd_x \right].
\end{equation}
In the last step we have used that that $\sup_{m}P(\gs \cN +m\in[-1,1]) = P(\gs \cN \in[-1,1])$.
Integrating the above inequalty on the interval $[0,h]$ we obtain 
\begin{equation}
\frac{\tf(h)}{h}\le  \frac{1}{N^d} e^{N^d h}\bE_{N} \left[ \sum_{x\in \tilde \gL _N} \gd_x \right].
\end{equation}
and hence that for every $N$,
\begin{equation}
\limsup_{h\searrow 0} \frac{\tf(h)}{h}\le   \frac{1}{N^d} \bE_{N} \left[ \sum_{x\in \tilde \gL _N} \gd_x \right]\, .
\end{equation}
Now using the fact that the variance of $\phi_x$ is close to $\sigma^2_d$ when the distance of $x$ to the boundary 
is large (see Section~\ref{afffaff}), it is standard to check that 
\begin{equation}
\lim_{N\to \infty} \frac{1}{N^d}\bE_{N} \left[ \sum_{x\in \tilde \gL _N} \gd_x \right]=C_d,
\end{equation}
which is sufficient to conclude.

The proof of \eqref{eq:fact2.4} for $d=2$ is substantially more involved and it is less related to our results because in any case
for $d=2$ we are unable to address the issue of the order of the transition when disorder is present. 
However, the reader can check that
the above method gives, for $d=2$, an upper-bound on $\tf(h)$ of the right order of magnitude ( that is $h(\log h)^{-1/2}$). 
See also Remark~\ref{rem:whendis2} for a proof of a lower-bound of the same order  (which also implies that the transition is at $h=0$ for $d=2$ too).

\medskip

Before moving toward the disordered case it is worth recalling that the phase transition we have just described is a localization 
transition and the localized LGFF is profoundly different from the LGFF since the continuum invariance of the latter is broken by the
localizing potential. In particular, correlations decay exponentially with the distance for the localized measure \cite{cf:BB,cf:IV,cf:BV,cf:Vel}, 
while the decay of correlations
for the LGFF is power law (see Section~\ref{afffaff}). Moreover a directly related issue for an akin model is the one of wetting
\cite{cf:CV,cf:BDZ,cf:Vel}: in this case, added to the pinning potential, the LGFF is constrained not to enter the lower half plane. 
This constraint  generates a repulsion, but the transition is still at $h=0$.

\subsection{Free energy and transition for the disordered model}

The existence of quenched free energy for the disordered model has been proved in {\cite[Theorem 2.1]{cf:CM1}}. We recall the result here:

\medskip

\begin{proposition}\label{freen}
The free energy 
\begin{equation}
\label{eq:freen}
\tf(\gb,h):=\lim_{N\to \infty} \frac{1}{N^d}\bbE \left[\log Z^{\gb,\go}_{N,h}\right]
\stackrel{\bbP(\dd \go)-a.s.}{=}  \lim_{N\to \infty} \frac{1}{N^d}\log Z^{\gb,\go}_{N,h} \, ,
\end{equation}
exists (and is self-averaging).
\end{proposition}

\medskip

Note that $\tf(0, h)=\tf(h)$. Moreover
it is easy to observe that $\tf(\gb,h)$ is non-decreasing and convex in $h$ and 
  we have (cf.\ \eqref{eq:contactfrac})
\begin{equation}\label{eq:discontact}
\partial_h \tf(\gb,h)\, =\, \lim_{N\to \infty} \frac{1}{N^d}  \bE^{\gb,\go}_{N,h}  \left[\sum_{x\in \tilde \gL_N}\delta_x\right]\, ,
\end{equation}
as soon as the left-hand side is defined.

Furthermore, from Jensen's inequality and convexity (we refer to the proof of \cite[Proposition 5.1]{cf:GB} for more details) we have
\begin{equation}
\label{freeineq}
 \tf(0,h-\gl(\gb)) \le \tf(\gb,h) \le \tf(0,h)\, ,
\end{equation}
which  implies  that  $\tf(\gb, h)\ge0$ for every $h\in \bbR$. This elementary but important lower bound  can  be established in a direct fashion precisely in the same way as for the 
non disordered case (cf. \eqref{eq:forLM}). But \eqref{freeineq} guarantees also that $F(\gb, h)=0$ for $h\le 0$ and that 
$F(\gb, h)>0$ if $h > \gl(\gb)$. 
Hence we have established the existence of  a localization transition and the critical value 
\begin{equation}
h_c(\gb)\, :=\, \inf\left\{ h :\,  \ \tf(\gb,h)>0\right\}\, ,
\end{equation}
 satisfies 
\begin{equation}
\label{hineq}
0\, \le\,  h_c(\gb)\, \le\,  \gl(\gb)\, .
\end{equation}

\subsection{The main results}
The aim of this paper is to investigate if the inequalities \eqref{freeineq} and \eqref{hineq} are sharp, and 
to compare the behavior of the model with respect to the pure, i.e. annealed, one.

The result we obtain are the following.

\medskip

\begin{theorem}
\label{th:main}
When $d\ge 3$, we have
\begin{itemize}
 \item [(i)] For all $\gb\in (0, \ubgb)$ there exists a constant $C$ (depending on $\gb$, $d$ and the law of $\go$) 
 such that for  $h \in (0,1)$
  \begin{equation}
  \label{nonGres}
   h^{66d}\,  \le\,  \tf(\gb,h)\,\le\, Ch^2 \, .
  \end{equation}
  \item [(ii)] When $\go$ is Gaussian,  for every $\gb>0$ there exists a constant $c(\gb,d)$ such that for $h \in (0,1)$
   \begin{equation}
   \label{Gres}
  c(\gb,d) h^2 \, \le \, \tf(\gb,h)\, \le\, \frac{h^2}{\gb^2}\, .
  \end{equation}
Moreover one can find a constant $C(d)$ such that $c(\gb,d)\ge C(d)/\gb^2$ for every $\gb\in (0,1]$. 
\end{itemize} 
\end{theorem}

\medskip

A trivial consequence of the Theorem is that $h_c(\gb)=0$ for all $\gb>0$. 

\medskip

For $d=2$ we are yet unable to decide whether there is a critical-point shift. However, in the Gaussian case, 
we are able to get a much better upper bound on $h_c(\gb)$ than the annealed one in \eqref{hineq}.

\medskip

\begin{theorem}
\label{th:d=2}
 When $d=2$ and $\go$ is Gaussian, every $\gep>0$
 there exists $c_\gep>0$ such that for $\gb \in (0,1)$
 \begin{equation}
 0\le h_c(\gb)\, \le\,  c_\gep \gb^{3-\gep}\, .
 \end{equation} 
\end{theorem}

\medskip
\subsection{Behavior of the field under $\bP^{\gb,\go}_{N,h}$}

The main focus of this paper is the free energy, but let  us briefly discuss  the properties of the trajectories in the case $d\ge 3$. The basic remark is that the behavior of the free energy directly yields that the (asymptotic) contact fraction  is zero
for $h< h_c(\gb)$ and it is positive and increasing for $h>h_c(\gb)$: strictly speaking the existence of the contact fraction 
is guaranteed by convexity only out of a countable subset of $\{h:\, h>h_c(\gb)\}$ but one can extend the definition by taking
limits (for example) from the right.   For what concerns 
$h=h_c(\gb)$
the  smoothing of the phase transition (due to the disorder) directly yields  that the  contact fraction is zero at the critical point
($h=0$), and this is in strong contrast with what happens in the pure case.

These are all issues that are directly related to convexity and free energy estimates, but 
a number of sharper questions are very natural,
notably the precise nature of the delocalized phase, that is when $\tf(\gb, h) =0$:
is it true that the total number of contacts is $O(N^{d-1})$ and they are all close to the boundary? The analogous question even  in the one dimensional set-up is not trivial even if by now rather sharp results are available
\cite{cf:AZnew}. Precise path description in the localized phase raises a number of issues too, in particular there are all
the issues that have been treated, not always with complete success, in the one dimensional set-up (see \cite[Ch. 8]{cf:G} and references therein), but the situation in the higher dimensional set-up may be richer and harder to tackle.  

Nevertheless we want to observe that the results that we prove suggest the following  
 typical behavior of $\phi$ for $h>0$ small, so in the localized phase but close to criticality:
$\phi$ typically stands at a large but finite (depending on $h$) distance of the interaction zone
(the proof seems to indicate that $|\phi_x|$ should be of order $u(h)\sim \sigma_d \sqrt{2 \log (1/h)}$)
since otherwise it should be difficult to avoid having a larger density of contact.
The contacts with the interaction zone are typically produced by atypical peaks off the typical height (since we are talking of peaks of finite height,
there is a positive but small density of them).

Here is a statement that goes in the direction of this conjecture:
\medskip

\begin{proposition}
\label{th:conjection}
For every $\gep>0$ there exists $h_0=h_{0}(\gep)$ such that 
for all $h\in (0,h_0)$
\begin{equation}
\lim_{N\to \infty} \bbE \left[ \bP^{\gb,\go}_{N,h} \left( \sum_{x\in \tilde \gL_N} \ind_{\{|\phi_x| \, \le\,  \sqrt{ \frac{1}{4d} \log (1/h)} \} }\ge \gep N^d \right)\right]\, =\, 0\,.
\end{equation}
\end{proposition}
\medskip

One is then tempted to conjecture that the interface chooses one side where to lie entirely, close to criticality, 
but we make no claim about this. Proposition~\ref{th:conjection} is proven in Appendix~\ref{sec:app2}.

\subsection{Co-membranes and selective solvents}
\label{sec:co-m}
It is worthwhile stating the generalization of the results to a model in which the localization mechanism is somewhat
different, but for which the technics can be adapted in a straightforward way. It is the analog of the model of a
copolymer in the proximity of the interface between selective solvents, see \cite{cf:Bcoprev,cf:coprev} and references therein. 
The model  is defined by
\begin{equation}
\label{eq:modmeascop}
\frac{\dd \check\bP^{\go,\gvr}_{N,h}}{\dd \bP_N}\, \propto\, \exp\left( \gvr \sum_{x\in  \tilde \gL_N} ( \go_x+h)\sign \left(\phi_x \right)\right)\, ,
\end{equation}
where without loss of generality we can assume both $h$ and $\gvr$ non negative and $\sign(0)=+1$.
There is a rather natural way of understanding the model: imagine that the free field models a membrane made up by
portions, say the unit box around $x$, that have an affinity for solvent A (if $\go_x+h>0$) or for solvent B (if $\go_x+h<0$).
And that the solvent A fills in the upper half plane, and in the lower one there is solvent B. When $h$ is positive there is an
overall preference, since $\go_x$ is centered, for solvent A, and the membrane in a average sense prefers to fluctuate in the upper
half plane.  However, there are membrane trajectories that, staying close to the A-B interface, can collect more
energetic rewards  and  the localization transition is between a regime in which the membrane trajectories stay close to the 
A-B interface and a regime in which the membranes prefer to stay in the A solvent ($h\ge 0$, so if there is a globally preferred solvent, it has to be A). 

A direct link with the pinning measure 
\eqref{eq:modmeas} can be made by observing that we can write
\begin{equation}
\label{eq:modcop}
\frac{\dd \check\bP^{\go,\gvr}_{N,h}}{\dd \bP_N}\, =\, \frac1{\check Z^{\go,\gvr}_{N,h}} \exp\left( -2\gvr \sum_{x\in  \tilde \gL_N} ( \go_x+h)\gD_x
\right)\, ,
\end{equation}
where $\gD_x:= (1- \sign(\phi_x))/2$, that is $\gD_x$ is the indicator function that $\phi_x$ is in the lower half plane.
It is probably worth stressing that from \eqref{eq:modmeascop} to \eqref{eq:modcop} there is a non-trivial (but rather simple) change in energy
(and free energy), but this change does not affect the measure, hence the model\footnote{It is however straightforward to see 
that the annealed models associated to \eqref{eq:modmeascop} and \eqref{eq:modcop} are substantially different \cite{cf:coprev,cf:GB}}. 
 And in the form \eqref{eq:modcop}
the analogy with the pinning case is evident. In particular, the strict analog of Proposition~\ref{freen} holds -- the free 
energy in this case is denoted by $\check\tf(\gvr, h)$ -- and, precisely like for the pinning case, one sees that
$\check\tf(\gvr, h)\ge 0$. We then set $\check h_c(\gvr):= \inf\{h>0:\, \check\tf(\gvr, h)=0\}$.

\medskip

\begin{theorem}
\label{th:cop}
For $d\ge 3$ and under the most general assumptions on the IID field $\go$ (i.e. bounded exponential moments, centered and unit variance)    we have that for every $\gvr\ge0$
\begin{equation}
\label{eq:cop}
\check h_c(\gvr)\, =\, \frac1{2\gvr}  \gl (-2\gvr)\, .
\end{equation}
Moreover \eqref{nonGres}, with $\tf(\gb,h)$ replaced by 
$\check\tf (\gvr, h_c(\gvr)-h)$, holds true and, if $\go$ is Gaussian, then  also  \eqref{Gres} holds once the same replacement is made.

For $d=2$ and assuming $\go$ to be Gaussian we have that
$\lim_{\gvr \searrow 0} \check h_c(\gvr)/ \gvr=1$.
\end{theorem}

\medskip

We have preferred to put the emphasis on the critical curve and on \eqref{eq:cop} because that is the same formula that
appears for the copolymer, but as a strict upper bound, except for the very particular case of inter-arrival laws of the form $L(n)/n$, $L(\cdot)$ slowly varying,  in which the upper bound  \eqref{eq:cop} is achieved  \cite{cf:Bcoprev,cf:coprev,cf:GB}. Moreover, a substantial emphasis for the copolymer has been put on the slope at the origin of $h_c(\gvr)$: 
in this case the slope is simply one.

Theorem~\ref{th:cop} also provides a smoothing result for $d\ge 3$ and, which is most interesting, 
when the disorder is Gaussian we have again a model in which the disorder is relevant -- in fact also for the co-membrane the pure model has a first order transition -- and we can compute the critical exponent of the free energy. 

We will not give a detailed proof of Theorem~\ref{th:cop}, because the arguments are really close to the ones for the pinning model and we limit ourselves to Remarks~\ref{rem:cop3}   and \ref{rem:cop2}.

\subsection{Discussion of the results, sketch of proofs  and structure of the paper}

\subsubsection{On the upper bound (and smoothing)}
The upper-bound   in \eqref{nonGres} and \eqref{Gres} is quite easy to prove and is  valid in any dimension.
Its proof can be read independently of the rest of the paper: it relies on the {\sl disorder tilt and   fractional moment bound} 
introduced in \cite{cf:GLT,cf:DGLT}. However, here the implementation of the idea  is remarkably straightforward: 
no coarse graining procedure is needed (see \cite[Section~6]{cf:G} for a review of various coarse graining procedures). 
The reason why things here are simpler is that the method is not used to show that
the free energy density is zero, like in the papers we have just mentioned, but simply to have a positive upper bound on it. 

Note that, on its own, the inequality $\tf(\gb,h)\le Ch^2$ does not imply a rounding or smoothing of the free energy function. It does only if one can prove that $h_c(\gb)=0$, and this is precisely what we prove for the disordered LGFF pinning.  Nevertheless
 such a bound recalls the smoothing inequality in \cite{cf:GT_cmp}, proven for RW pinning models. 
As a matter of  fact the upper-bound   in \eqref{nonGres}, that is Proposition~\ref{rounding}, applies to 
RW
pinning models too, but in this case $h_c(\gb)=0$ only if disorder is irrelevant  and, even if  
 the smoothing 
inequality in \cite{cf:GT_cmp}
and  Proposition~\ref{rounding} are essentially the same result in this case,
both of them end up having
little importance because a direct application of Jensen inequality (annealed
bound) and explicit computations lead to a better result (the exponent is larger than $2$! \cite{cf:GB}). Of course the smoothing inequality for RW pinning holds with respect to
the correct critical point also when disorder is relevant and $h_c(\gb) \neq 0$.  
Generalizing the rare stretch approach in \cite{cf:GT_cmp} to LGFF in order to establish a quadratic bound on the critical behavior
does not appear to be straightforward and such a result would be in any case sensibly weaker than what we prove here.

\subsubsection{On the lower bound ($d\ge 3$)}
But how can we match the upper bound? That is, how can we show that $h_c(\gb)=0$
and  find a lower bound on the free energy of quadratic type? We try to sketch here an answer to this question in a few steps:

\medskip

\begin{enumerate}
\item We show in Section \ref{preparatory} that one can {\sl raise} the boundary conditions from $0$ to an arbitrary $u$
(that, conventionally, we choose positive). 
The reason why this is true is the continuum symmetry enjoyed by the LGFF: the 
term in the exponent in \eqref{density} is formally invariant when $\phi_x$ is mapped to $\phi_x+u$ for all $x$, if one chooses to neglect the effect of the boundary, 
which  is irrelevant for free energy computations (the reason for this is that the volume of the boundary is negligible with respect to that of the whole box).


We then choose $h>0$ close to zero and  a box of linear size $N$ ($N$ will be chosen as a function of $h$ and it will be somewhat large, see below).  
We then choose $u=u(h)$ such that the probability
that $\phi_x\sim \cN(u, \gs_d^2)$ is in $[-1,1]$ (the {\sl contact probability}) is $ah$, where $a$ is a positive constant to be chosen.  
We have in particular $\lim_{h \searrow 0} u(h)=\infty$.
\item We make now a bold proposal: we ask the reader to think of  the variables $\phi_x$'s as independent. 
Of course they are not, but it is well known (see \cite{cf:CCH} for a quantitative result) that extrema and large excursions  of the LGFF in $d\ge 3$ 
are close to what we would get forgetting the correlations and we are now rather far from the region where the pinning acts
($u=u(h)$ is large!).  We stress that in the previous steps we have invoked the continuum symmetry of the model, that
leads to power law correlations, so this step is a delicate one. If we accept this bold replacement we are now dealing with a model which is 
exactly solvable:
\begin{equation}
\label{eq:indep1.0}
\tilde \tf _N( \gb, h) \, :=\, 
\frac 1{\gL_N}\bbE \log \bE\left[ \exp\left( \sum_{x\in   \gL_N} (\gb \go_x-\gl(\gb)+h)\tilde \delta_x\right) \right] \, ,
\end{equation}
where $\tilde \delta_x= \ind_{[-1,1]} (\tilde \phi_x)$ and the  $\tilde \phi_x$'s are IID $\cN(u, \gs_d^2)$ random variables.
Recall that we have chosen $u$ so that $\bE[\tilde \delta_x]=\bP (\tilde \delta _x=1)=ah$
so that it is straightforward to see that 
\begin{equation}
\label{eq:indep1.1}
\tilde \tf _N( \gb, h)\, =\, 
\bbE \log \bE\left[ \exp\left( (\gb \go-\gl(\gb)+h)\tilde \delta_x\right) \right] \, =\, 
\bbE \log \left(1 + ah \xi \right)
\,,
\end{equation}
where $x$ is arbitrary (the variables are IID) and $\xi:= \exp\left( \gb \go-\gl(\gb)+h\right)-1>-1$.
If we assume that $\bbE[\exp(3\gb \go)]< \infty$ (with some more effort one can generalize the argument to $\gb < \ubgb$),
for  $h \searrow 0$ we have $\bbE[ \xi]= e^h-1= h +O(h^2)$ 
and $\bbE [\xi^2] = c_\gb + O(h)$, with $c_\gb:=e^{\gl(2\gb)-2\gl(\gb)}-1>0$ and 
$\bbE  [\xi _+^3]$ is bounded. 
By putting all this together with the elementary bound
\begin{equation}
  x^3 \ind_{[-1/2,0]}(x) \stackrel{x\ge -1/2}\le    \log(1+x)-x+\frac 12 x^2 \stackrel{x>-1}\le  \frac 13 x^3 \ind_{[0, \infty)}(x) \, , 
\end{equation}
one  sees that
\begin{equation}
\label{eq:indep1.2}
\tilde \tf _N( \gb, h) \, =\, a h^2 - c_\gb\frac  {a^2 h^2} 2 +O(h^3)\, ,
\end{equation}
and setting $a=1/c_\gb$   yields the quadratic behavior  in $h$ we were looking for.
Note that this gives a justification \textit{a posteriori} for choosing $\bE[\tilde \delta_x]$ proportional to $h$:
any other choice would give a smaller, if not negative, lower bound on the free energy.

\item It appears that $N$ can be chosen arbitrarily up to now (and this is  quite troublesome!).
However a closer look suggest that $N$ has to be chosen large -- at least like a power of $\frac 1 h$ -- because boundary effects have to be taken care of.
In fact, in order to deal with a super-additive model we do not choose boundary conditions equal to $u$, but boundary conditions
that are a sampled from an infinite volume free field of mean $u$. Therefore the value of the field at the boundary (hence also close to it)
can occasionally be also rather different from $u$ and that the contact probability is $ah$ -- that we have used above -- can be rather far from the truth. We need therefore to be able to neglect a fairly large portion of sites close to the boundary 
in order to be sufficiently far so that an averaging effect -- the mean on a LGFF is the solution of a Dirichlet problem for
the discrete Laplacian -- takes place. It is not too difficult to get convinced that one needs to take $N$ to grow  like a 
power of $\frac 1 h$:
even if we imagine that we are able to make sufficiently sharp estimates for sites that are at a finite (large) distance from the boundary, hence gaining  in the bulk  a contribution in any case not larger (annealed bound!) than   $h \bP^u(\phi_0\in[-1,1]) N^d=
a h^2N^d$ (we are  assuming that  $\bP^u(\phi_x\in[-1,1])$ essentially does not feel the boundary),
when one is on the boundary it is not evident how to argue that one does not get a negative contribution.
Actually  
in \eqref{badevent}, but this is taken up in a more informal fashion in Section~\ref{nonoptimality} and notably in \eqref{boundaryeffectseparated}, it is  argued that the boundary gives a contribution smaller, i.e. more negative,  than  a $\gb$-dependent negative constant 
times $\bP^u(\phi_0\in[-1,1])N^{d-1}$, which is hence of the order of $h N^{d-1}$ and  we have therefore to choose $N \gg h^{-1}$
to have a chance that the bulk prevails on the boundary term.
\item At this point we get back with a last consideration on the {\sl bold replacement} at step (2). The structure of the result 
we got using this replacement, that is \eqref{eq:indep1.2}, is quite clear: we have an energetic gain (the first term in the right-hand side)
that is what we would get by Jensen's inequality (annealed bound) even without independence assumption. We have then a quadratic 
loss, that is the second herm in the right-hand side. So one needs to implement an efficient second moment method and to do this 
we resort to Gaussian interpolation techniques \cite{cf:GuTo,cf:Trep}, which limits our result to Gaussian disorder. 
Still, even exploiting the interpolation formula, the result is not straightforward because the {\sl quadratic coupling term} grows too fast.
So what we do is to apply the interpolation after having restricted the model to trajectories of the LGFF  that have only a bounded 
number of contacts on suitably chosen intermediate scale boxes (for example, if the box has volume smaller than $\frac 1{ah}$ then in average there will be less than one contact). We do not explain this procedure in detail here, but we just remark that
the event that the number of contacts is suitably limited becomes improbable if the region in which this requirement is made is too large,  
but boxes of edge-length  that is a power of $\frac 1 h$ turn out to be fine.  
\item  All of this targets the quadratic behavior. We can be much rougher if we just target $h$ to some (large) positive power,
see the lower bound in \eqref{nonGres}. In this case, once $N$ is chosen to grow like a power of $\frac 1h$, we can 
choose $u(h)$ growing so that the contact probability is $h$ to some power larger than one and we choose a power so large that the probability of having a contact in the whole box vanishes with $h$. Of course this way we will not get  close to the quadratic behavior, but
the boundary control, since the field at the boundary is very high, is easier and the second moment procedure is much less delicate because there are
so little contacts in the underlying measure. The whole argument then goes through using less sophisticated techniques which are however 
helpful in understanding the argument leading to the quadratic lower bound.
\end{enumerate}

\subsubsection{Structure of the paper}
The rest of the paper is organized as follows:
\begin{itemize}
\item
We conclude Section \ref{modelresult} by mentioning classical results for the 
lattice free field which we will us throughout the paper. 
\item
In Section \ref{fracmom}, 
we use a very simple fractional moment method to show that 
 $\tf(\gb,h)\le Ch^2$ (in any dimension). 
 \item
 In Section \ref{preparatory} we show that the free energy is not sensitive to mild modifications of the 
 boundary conditions, and use this information to get a lower bound on $\tf(\gb,h)$ which is the free energy of a system of finite size 
 (see \eqref{superad} this is what we call a finite volume criterion). This criterion is used in all the next sections.
 \item
 Sections \ref{lowbno} and \ref{lowbo} are dedicated to the lower-bound 
 on the free energy for $d\ge 3$: 
 In Section \ref{lowbno}, we establish the non-optimal lower bound in the non Gaussian case \ref{nonGres}.
In Section \ref{lowbo}, we establish the sharp bound in the Gaussian case, which is the most technical result of the paper.
We advise the reader to go through Section \ref{lowbno} before reading \ref{lowbo}.
\item
Finally Section \ref{twodim} is dedicated to the case $d=2$ and the proof of Theorem \ref{th:d=2}: this last section  adapts and uses tools of Section 
\ref{preparatory}.
\end{itemize}

\subsection{A few fun facts about the free field}
\label{afffaff} 

Let $(X_t)_{t\ge 0}$ 
denote the continuous time simple random walk on $\bbZ^d$ (let $P^x$ denote its law starting from $x$) 
whose transition rates are one along $\bbZ^d$-edges (see \cite{cf:LL} for a complete reference on the subject).
We let $\gD$ denote the generator of $X$
\begin{equation}
 \gD f (x):= \sum_{y\sim x} (f(y)-f(x))\, .
\end{equation}
Let us stress that the simple random walk in  \cite{cf:LL} is generated by $\gD/2d$ --
the walk jumps at rate one and chooses one of the $2d$ neighborhoods at random -- but 
our choice \eqref{density} requires speeding up the walk by a factor $2d$ to have that the 
covariance of the $\phi$ field is the  random walk Green function, see \eqref{eq:GreenL} and \eqref{eq:Green}.
For a set $B\subset \bbZ$ let $\tau_{B}$ be the first hitting time of the set $B$ by $X$.
Note that the Gaussian free field is a Gaussian process. Its covariance under measure $\bE^{\hat \phi}_\gL$ 
does not depend on the boundary conditions and is given by
\begin{equation}
\label{eq:GreenL}
G_{\gL}(x,y):= E^x\left[\int_0^{\tau_{\partial \gL}} \ind_{\{X_t=y\}} \dd t \right].
\end{equation}
Note that for $d\ge 3$, $G_{\gL}(x,y)$ is uniformly bounded (in $\gL$).
This is the reason for which in this case, there exists a (unique) centered infinite volume version of the field  whose 
covariance function is given by 
\begin{equation}
\label{eq:Green}
G(x,y):= E^x\left[\int_0^{\infty} \ind_{\{X_t=y\}} \dd t\right].
\end{equation}
In particular we set $\sigma_d^2$ (or $\sigma^2$ when no confusion is possible) to be the variance $G(x,x)$ of 
the infinite volume field.
We recall the translation invariance $G(x,y)=G(0,y-x)$ and the standard bound
\begin{equation}
\label{eq:Green-bound}
G(0,x)\, \le\, \frac{\mathsf{c}_d}{(1+\vert x\vert^{d-2})}\, ,
\end{equation}
where $\mathsf{c}_d$ is a constant that can be made explicit if one desires and $|\cdot |$ denotes the Euclidean norm. We call $\bP$ resp. $\bP^u$ the law of the field with covariance 
given by \eqref{eq:Green} with mean $0$ resp.\ $u$. We use the notation $\hat \bP^u$ when the field is denoted by $\hat \phi$ instead of $\phi$.

\medskip

In  $d=2$ the infinite volume field does not exist and we will make use of the following
estimate \cite[Prop.~6.3.2]{cf:LL}
\begin{equation}\label{vsnd}
G_{\{x\in \bbZ^2:\, \vert x\vert \le N\}}(0,0)\, =\, \frac 1 {2\pi} \log N + O(1)\, .
\end{equation}
One can easily extract a number of results from \eqref{vsnd} by the mean of comparison arguments 
(use $G_\gL(x,y)\ge G_{\gL'}(x,y)$ if $\gL' \subset \gL$), notably that
for every $\gep>0$ we can find $d_\gep>0$ such that if $N>2d_\gep$ and if $x\in \gL_N$ is such that
$\text{dist}(x, \partial \gL_N)>d_\gep$ then 
\begin{equation}
\label{eq:G2est}
G_{\gL_N}(x,x)\, \ge \, (1-\gep) \frac 1 {2\pi} \log\text{dist}(x, \partial \gL_N)\, . 
\end{equation}
where 
\begin{equation}
\text{dist}(x, A)\, :=\, \min_{y\in A} |y-x|\, .
\end{equation}

For $m>0$, the massive-free field with mass $m$ is defined by adding an harmonic confinement for each $x$:
 \begin{equation}\label{massivedensity}
\bP_\gL^{\hat \phi,m}(\dd \phi)=\frac{1}{\mathcal Z_{\gL}^{\hat \phi,m}}
  \exp\left(-\frac 1 2 \sumtwo{(x,y)\in (\gL)^2 \setminus (\partial \gL)^2 }{x\sim y}\frac{ (\phi_x-\phi_y)^2 }{2} 
\right)\prod_{x\in \mathring{\gL}} \exp\left(-\frac{m^2}{2}\phi_x^2\right)\dd \phi_x \, .
\end{equation} 
Its covariance fonction is given by the Green function of the operator $\gD-m^2$, or 
\begin{equation}
\label{eq:GreenLmassic}
G^m_{\gL}(x,y):= E^x\left[\int_0^{\max\left(\tau_{\partial \gL},m^{-2}\cT\right)} \ind_{\{X_t=y\}} \dd t \right].
\end{equation}
where $X$ is a simple random walk and $\cT$ is an exponential variable, of parameter one, independent of $X$.
The infinite volume massive free field exists in any dimension $d\ge 1$ and the covariance is given by 
\begin{equation}
\label{eq:Greenmassic}
G^m(x,y):= E^x\left[\int_0^{m^{-2}\cT} \ind_{\{X_t=y\}} \dd t\right].
\end{equation}

\medskip

It follows from the expression \eqref{density} that the Gaussian Free Field (and the massive one) satisfies the spatial Markov
property. If $\gG \subset \gL$ (or $\subset \bbZ^d$ for the infinite volume case)
the law of $\phi_{|\gG}$ knowing  $\phi$ ouside of $\mathring \gG$ is given by
$\bP_{\gG}^{\phi_{|\partial \gG}}$ ($\bP_{\gG}^{\phi_{|\partial \gG},m}$ in the massive case).

Moreover for $m\ge 0$ the mean of $\phi$ under $\bP_{\gL}^{\hat \phi,m}$ is given by $H=H^{\hat \phi,m}_\gL$
the solution of 
\begin{equation}
\label{harmonic}
 \begin{cases}
  (\gD -m^2) H(x)=  0  & \text{ if } x \in \mathring{\gL},\\
 H (x)=\hat \phi_x&  \text{ if } x \in \partial\gL\, .
 \end{cases}
 \end{equation}
 We will exploit the random walk (or Poisson kernel) representation of this solution
\begin{equation}
\label{poisson}
H^{\hat \phi,m}_{\gL} (x)\, =\, E^x\left[ \hat \phi_{X_{\tau_{\partial \gL}}}; \, \tau_{\partial \gL} < m^{-2}\cT\right]\, ,  
\end{equation} 
with $\tau_A= \inf\{t:\, X_t\in A\}$. If $m=0$, we just drop it from the notation.

\section{Fractional moment: upper-bound on the free energy} 
\label{fracmom}

\begin{proposition}
\label{rounding} Choose $\gb< \ubgb$ (cf. \eqref{eq:assume-gl}).
For every $c>1$ there exists $h_0>0$ such that
for $h \in (0, h_0]$
\begin{equation}
\tf(\gb,h)\, \le\,  c\frac{h^2}{\gl'(\gb)^2}\, ,
\end{equation}
where $\gl'(\cdot)$ is the derivative of $\gl(\cdot)$ defined in \eqref{eq:defgl}. 
In the Gaussian case we can choose $c=1$ and the result is valid for all $h$.
\end{proposition}

\medskip

\begin{proof}
Let us fist observe that by Jensen's inequality
\begin{equation}
 \bbE\left[ \log Z^{\gb,\go}_{N,h} \right]=  2 \bbE\left[ \log \sqrt{Z^{\gb,\go}_{N,h}} \right]\le 2 \log  \bbE\left[\sqrt{Z^{\gb,\go}_{N,h}} \right].
\end{equation}
This implies that 
\begin{equation}\label{eq:dafree}
\tf(\gb,h)=  \limsup_{N\to \infty} \frac{2}{N^d}\log  \bbE\left[\sqrt{Z^{\gb,\go}_{N,h}} \right].
\end{equation}
We are going to estimate  $\bbE\left[\sqrt{Z^{\gb,\go}_{N,h}} \right]$ by making a change of measure on the environment.
Let us start by making the preliminary observation that for every $\gb>0$ and 
$h\in (0,\gl(\gb)+\gl(-\gb))$ there exists a unique solution $\alpha(\gb,h)\in(0,\gb)$ to
\begin{equation}
\label{defalpha}
\gl(\gb-\alpha)-\gl(-\alpha)-\gl(\gb)+h\, =\, 0\, , 
\end{equation}
which follows by observing that the left-hand side is positive for $\ga=0$,
negative for $\ga=\gb$ and decreasing in $\ga$ in the interval $(0, \gb)$. 
Moreover, 
when $\go$ is Gaussian $\alpha(\gb,h)= h/\gb$ and, in general,  we have
\begin{equation}
\label{eq:nGfma}
 \alpha(\gb,h)\stackrel{h\searrow 0}{\sim}\frac h{\gl'(\gb)}\, .
 \end{equation}

Now let $\tilde \bbP=\tilde \bbP_N$ be a new measure  on $\bbR^{\bbZ^d}$ (we are changing the law of the disorder keeping its independent  character) defined by  
\begin{equation}
\label{eq:RNderfm}
\frac{\dd \tilde \bbP}{\dd  \bbP}(\go):= \exp\left(\sum_{ x\in  \tilde \gL_N}
\left(-\alpha \go_x-\gl(-\alpha)\right)\right)\, ,
\end{equation}
and, by the  
 definition of $\alpha$, cf. \eqref{defalpha},   one has for $x \in \tilde \gL_N$
\begin{equation}
\label{alpha2}
\tilde \bbE\left[ e^{\gb \go_x-\gl(\gb)+h}\right]\,=\, 1\, .
\end{equation}
From the Cauchy-Schwartz inequality we obtain
\begin{equation}
\label{eq:CSfm}
\left(\bbE\left[\sqrt{Z^{\gb,\go}_{N,h}}\right]\right)^2\,\le\, \tilde \bbE\left[ Z^{\gb,\go}_{N,h} \right] \bbE\left[ \frac{\dd \bbP}{\dd \tilde \bbP} \right]\, ,
\end{equation}
and the first factor in the right hand side  is equal to one because of \eqref{alpha2}. For the second one we have instead 
\begin{equation}
\bbE\left[ \frac{\dd \bbP}{\dd \tilde  \bbP} \right]= \exp\left(N^d \left(\gl(\alpha)+\gl(-\alpha)\right)\right)\, .
\end{equation}
Hence one can deduce from it that 
\begin{equation}
 \limsup_{N\to \infty} \frac{2}{N^d}\log \bbE\left[ \sqrt{Z^{\gb,\go}_{N,h}} \right]\le  \gl(\alpha)+\gl(-\alpha)
\stackrel{\ga \searrow 0} \sim \ga ^2 \, , 
\end{equation}
and by \eqref{eq:dafree} and \eqref{eq:nGfma} the proof is complete. 
\end{proof}

\section{Elevated boundary conditions, stationary boundary conditions and finite volume criterion}\label{preparatory}

In this section we manage to get a comparision between $\tf(\gb,h)$ 
and the  free energy per unit site of a finite system.
To obtain this inequality, we need to change a bit the boundary conditions: instead of $\phi\equiv 0$ on the boundary of $\gL_N$,
we choose to take $\phi$ to be distributed as an infinite volume LGFF (this requires $d\ge 3$).
 We will also play on taking \emph{elevated}  boundary conditions,
in the sense that the infinite volume LGFF is centered at a non zero value $u$ that then will be chosen suitably large (and will depend on $h$). 
For ease of exposition we first show that replacing $0$ boundary conditions ($\bP_N=\bP_N^0$) with
$u$ boundary conditions ($\bP_N^u$) does not change the free energy. We then show that the boundary conditions $u$
can be replaced by a typical realization of   the infinite volume LGFF of mean $u$. 

In this section the only requirement on $\gb$ is $\gl(\gb)< \infty$.

\subsection{Elevated boundary conditions}

\begin{proposition}\label{elevated}
For any $u\in \bbR$
\begin{equation}
\label{eq:elevated}
\lim_{N\to \infty} \frac{1}{N^d} \bbE \left[ \log Z^{\gb,\go,u}_{N,h}\right]= \tf(\gb,h).
\end{equation}
\end{proposition}
\medskip 

\begin{proof}
We are going to  prove almost sure convergence to  $\tf(\gb,h)$ rather than convergence of the expectation: 
since $\vert N^{-d} \log Z^{\gb,\go,u}_{N,h}\vert$ is bounded by
$ N^{-d}\sum_{x\in  \tilde \gL_N} \vert \gb \go_x-\gl(\gb)+h\vert$ and the latter forms a uniformly integrable sequence,
almost sure convergence implies $L^1$ convergence. 

%

We now start the proof of the a.s. convergence by observing that
 for  all $u$  
\begin{multline}
\label{partition}
\log Z^{\gb,\go,u}_{N,h}= -\ind_{\{|u|>  1\}} \left(\sum_{x\in \tilde \gL_N\cap  \partial \gL_N }
(\gb \go_x-\gl(\gb)+h)\right)\\
+ \log \bE_N\left[\exp\left( \sumtwo{x\in \partial \gL_N, y\in \partial^{-}\gL_N}{x\sim y} \left(u\phi_y-\frac{u^2}{2}\right)\right)  
\exp\left( \sum_{x\in  \tilde \gL_N} (\gb \go_x-\gl(\gb)+h)\delta_x\right)  \right]\, ,
\end{multline} 
where we used  $\cZ^u_N =\cZ^0_N$  (recall \eqref{eq:defcalz}).
The first term in the right-hand side, we call it $-b_{N,u}(\go)$,  yields a contribution which is $o(N^d)$ and thus has no influence on the limit.
What one has to check is that the second term compares well with $\log Z^{\gb,\go,0}_{N,h}$.
For this 
we first remark  that  if we choose a $C> \bbE[\vert \gb \go_x -\gl(\gb)+h\vert]$
there exists  $N_0(\go)$, with $\bbP(N_0(\go)< \infty)=1$,  such that for all $N\ge N_0(\go)$   we have
\begin{equation}\label{trivi}
\sup_{ \phi \in \bbR^{\mathring \gL_N} } \left | \sum_{x\in  \tilde \gL_N} (\gb \go_x-\gl(\gb)+h)\delta_x \right | \,\le\, C  
N^d \, .
\end{equation}
Then one can check that under the probability law $\bP_N$ (recall the definition \eqref{boundary}),
\begin{equation}\label{tphi}
T(\phi):=\sumtwo{x\in \partial \gL_N, y\in \partial^{-}\gL_N}{x\sim y}\phi_y\, ,
\end{equation}
is a centered Gaussian.  
Its variance is equal to $2d(N-1)^{d-1}$ which is the number of edges linking $\partial \gL_N$ to $\partial^{-}\gL_N$   because
\begin{equation}
 1=\frac{\cZ^{u}_{N}}{\cZ^0_N}= 
 \bE_N\left[\exp\left( \sumtwo{x\in \partial \gL_N, y\in \partial^{-}\gL_N}{x\sim y} \left(u\phi_y-\frac{u^2}{2}\right)\right)\right].
\end{equation}
Hence there exists $c>0$ such that for $N$ sufficiently large 
\begin{equation}\label{contrit}\begin{split}
\bP_N\left(|T(\phi)|\ge N^{d-1/4}\right)&\le \exp(-c N^{d+1/2})\\
\bE_N\left[ e^{ u T(\phi)} \ind_{\{|T(\phi)|\ge N^{d-1/4}\}}\right]&\le \exp(-c N^{d+1/2})\, ,
\end{split}\end{equation}
where, for the second inequality, how large $N$ should be chosen may depend on $u$. 
We now set 
 $A_N:=\{|T(\phi)|\le  N^{d-1/4}\}$.
We observe (by \eqref{contrit} and 
 \eqref{trivi} for the first line, and by the law of large number for the second one) that 
 \begin{equation}
 \label{basineq} \begin{split}
&\lim_N N^{-d}\log Z^{\gb,\go,u}_{N,h}(A_N^\complement)\, =\, -\infty \\
&\liminf_{N \to \infty }N^{-d}\log Z^{\gb,\go,u}_{N,h}\, \ge\,  -\bbE\vert \gb \go -\gl(\gb)+h \vert.
\end{split}\end{equation}
One can also easily show also that the inferior limit in the second line is non-negative, but here this bound suffices 
and we use it, coupled with the first inequality in  
\eqref{basineq},  to establish the first  of the inequalities, that holds for $N$ sufficiently large,  in the next formula  
 \begin{equation}
 \frac 1 2 Z^{\gb,\go,u}_{N,h}\le 
Z^{\gb,\go,u}_{N,h}(A_N) =
Z^{\gb,\go,u}_{N,h}-  Z^{\gb,\go,u}_{N,h}\left(A_N^\complement\right)\, \le \, Z^{\gb,\go,u}_{N,h} , 
\end{equation}
 and hence that
 $Z^{\gb,\go,u}_{N,h}(A_N)$ 
and $Z^{\gb,\go,u}_{N,h}$ are equivalent 
for computing  the free energy.
Moreover (recall that $b_{N,u}(\go)$ is defined right after \eqref{partition})
\begin{multline}
e^{-uN^{d-1/4}}
Z^{\gb,\go,0}_{N,h} \left( A_N \right)
= e^{-uN^{d-1/4}} \bE_N\left[
e^{ \sum_{x\in  \tilde \gL_N} (\gb \go_x-\gl(\gb)+h)\delta_x}; \, A_N\right]\\
\le\bE_N\left[ e^{u T(\phi)} 
e^{ \sum_{x\in  \tilde \gL_N} (\gb \go_x-\gl(\gb)+h)\delta_x} ; \, A_N
\right]= e^{b_{N,u}(\go)}Z^{\gb,\go,u}_{N,h}\left( A_N \right)
 \le e^{u N^{d-1/4}} Z^{\gb,\go,0}_{N,h}\, ,
\end{multline}
which is enough to conclude since the result for $Z^{\gb,\go}_{N,h}=Z^{\gb,\go,0}_{N,h}$ is known (cf. Proposition~\ref{freen}).
\end{proof}

\subsection{Stationary boundary conditions}\label{statbound}

When $d\ge 3$,
(recall Section \ref{afffaff}) $\bP^u$ is the law of the infinite volume free field $(\hat \phi_x)_{x\in \bbZ^d}$ with mean $u$.
We have seen that we can approach the free energy by considering the size $N$ approximation of the free energy 
$\bbE[\log Z^{\gb,\go, u}_{N,h}]/N^d$
instead of the original one $\bbE[\log Z^{\gb,\go}_{N,h}]/N^d$. Now we want to make the further step of replacing  $u$ at the boundary
by a typical configuration of the LGFF with mean $u$.  
We do this to recover a sharp super-additive property that in turn guarantees that, for every $N$, the new size $N$  
approximation bounds is a lower bound for  the free energy. 
\medskip

\begin{proposition}\label{superadd}
For any value of $u$ one has 
\begin{equation}\label{limit} 
\lim_{N\to \infty} \frac{1}{N^d} \bbE\hat \bE^u\left[ \log Z^{\gb,\go,\hat \phi}_{N,h}  \right]\,  = \, \tf(\gb,h)\, .
\end{equation}
Moreover for any $u$ and $N$ one has 
\begin{equation}\label{superad}
\frac{1}{N^d} \bbE\hat \bE^u\left[ \log Z^{\gb,\go,\hat \phi}_{N,h}  \right] \, \le \, \tf(\gb,h)\, .
\end{equation}
\end{proposition}
\medskip

The result \eqref{limit} is easy to believe because replacing $u$ by a sequence of Gaussian variables, of mean $u$ and variance $\gs_d$, in the boundary conditions
does not look a very drastic change: we are in the same framework of ideas as of Proposition~\ref{elevated}. However,
because of the random nature of the boundary values makes  the proof more technical. 
The second result, that is \eqref{superad}, just follows from the the Markov property of the LGFF
and Jensen's inequality.
\medskip

\begin{proof}
As for Proposition~\ref{elevated},  \eqref{limit} follows if we can show that 
\begin{equation}
\lim_{N\to \infty} \frac{1}{N^d} \hat \bE^u\left[ \log Z^{\gb,\go,\hat \phi}_{N,h}  \right] = \tf(\gb,h),
\quad \bbP \text{ a.s}.
\end{equation}
On the other hand, precisely by the same bound used at the beginning of the proof of Proposition~\ref{elevated} we see
that also $N^{-d} \log Z^{\gb,\go,\hat \phi}_{N,h} $ forms a uniformly integrable sequence (this time the measure is 
$\bbP \otimes \hat \bP ^u$). Therefore it suffices to show that
\begin{equation}\label{limit2} 
\lim_{N\to \infty} \frac{1}{N^d} \log Z^{\gb,\go,\hat \phi}_{N,h}  \,= \, \tf(\gb,h), 
\quad \bbP\otimes  \hat \bP ^u \text{ a.s}.
\end{equation}
For this 
we first note that
\begin{multline}\label{partition2}
\log Z^{\gb,\go,\hat \phi }_{N,h}= 
-\left(\sum_{x\in \tilde \gL_N\cap  \partial \gL_N }(\gb \go_x-\gl(\gb)+h)\ind_{\hat \phi_x \notin [-1,1]}\right)
+\log \left(\cZ^0_N / \cZ^{\hat \phi}_N\right)
\\
+\log \bE_N\left[\exp\left( \sumtwo{x\in \partial \gL_N, \, y \in \partial^- \gL_N}{x\sim y} \left(\hat \phi_x\phi_y-\frac{\hat \phi_x ^2}{2}\right)\right)  
\exp\left( \sum_{x\in  \tilde\gL_N} (\gb \go_x-\gl(\gb)+h)\delta_x\right)  \right]\, ,
\end{multline} 
The  right-hand side is of the form $T_1+T_2+T_3$: for $T_1$ we observe that
\begin{equation}
\label{trivi2}
 \vert T_1\vert \, \le \,  \sum_{x\in \tilde \gL_N\cap  \partial \gL_N }\vert \gb \go_x-\gl(\gb)+h\vert 
\, = \, O(N^{d-1})\, ,
\end{equation}
$\bbP\otimes \hat \bP^u$-a.s.
and thus we can neglect $T_1$.

Let us now examine $T_3$: first of all
the term
$\frac 1{2} \sum_{\ldots} \hat \phi_x^2$ is a constant with respect to $\bP_N(\dd \phi)$ and drops out of the expectation 
and one can easily show that it yields $\hat P^u$-a.s. an additive contribution to $T_3$ of order
$O(N^{d-1} \log N)$
and hence plays no role in the limit.
Let us then  control the influence of the term in the exponential. Set
\begin{equation}
 T(\hat \phi,\phi):=\sumtwo{x\in \partial \gL_N, y\in \partial^-{\gL}_N}{x\sim y} \hat \phi_x\phi_y\, ,
\end{equation}
Let us call $M_N= M_N(\hat \phi)$ the maximal value of $ |\hat \phi_x|$ in $\partial \gL_N$
(Note that $M_N$ is $O(\sqrt{\log N})$ $\hat \bP^u$-a.s.).
Since the  correlations are positive, the variance of 
$ T(\hat \phi,\phi)$ under $\bP_N$  is smaller than that of $M_N(\hat \phi) T(\phi)$ (recall \eqref{tphi}).
Hence similarly to \eqref{contrit} one obtains that there exists $c>0$ and $N_0$ (not depending on $\hat \phi$) such that
for $N \ge N_0$ we have  
\begin{equation}\begin{split}
&             \bP_N\left(\left| T(\hat \phi,\phi)\right|\ge N^{d-1/4}M_N\right)
\le \exp(-c N^{d+1/2})\, ,
\\
&\bE_N\left[e^{ T(\hat \phi,\phi)}\ind_{\{\left| T(\hat \phi,\phi)\right|\ge N^{d-1/4}M_N\}}\right]
\le \exp(-c N^{d+1/2})\, .
  \end{split}
\end{equation}
This together with \eqref{trivi} guarantees that if we set 
\begin{equation}
A_N:=\left\{\left| T(\hat \phi,\phi)\right|\ge N^{d-1/4}M_N\right\}\, ,
\end{equation}
then, like for for Lemma \ref{elevated}, we readily see that $\bbP\otimes \hat \bP^u-a.s.$
\begin{equation}
 \begin{split}
 & \lim_{N \to \infty}
N^{-d} \log  
Z^{\gb,\go,\hat \phi }_{N,h}(A_N) \, =\, -\infty\, ,  \\
& \liminf_{N \to \infty}
N^{-d} \log  
Z^{\gb,\go,\hat \phi }_{N,h}(A_N) \, \ge\,  \bbE  \vert \gb \go -\gl(\gb)+h \vert \, ,
 \end{split}
\end{equation}
and therefore 
 we obtain that there exists a random variable $N_0$, with
 $\bbP \otimes \hat \bP^u(N_0 < \infty)=1$ such that for $N\ge N_0$
\begin{equation}
\frac 12 Z^{\gb,\go,\hat \phi }_{N,h} \le Z^{\gb,\go,\hat \phi }_{N,h} (A_N) \le Z^{\gb,\go,\hat \phi }_{N,h}\, ,
\end{equation}
and analogous statement for 
$Z^{\gb,\go,0 }_{N,h}$.
Then one concludes similarly to what we have done for Lemma \ref{elevated}: we have
\begin{multline}
Z^{\gb,\go,0}_{N,h} (A_N) e^{-M_N N^{d-1/4}}
\le \bE_N\left[e^{T(\hat \phi,\phi)}
e^{ \sum_{x\in  \tilde \gL_N} (\gb \go_x-\gl(\gb)+h)\delta_x }; \, A_N\right]
 \\
 \le e^{M_N N^{d-1/4}} Z^{\gb,\go,0}_{N,h}\, ,
\end{multline}
therefore $\lim_N N^{-d}T_3= \lim_N N^{-d} \log Z^{\gb,\go,0}_{N,h}$,  $\bbP \otimes \hat \bP^u$-a.s., and the latter is just $\tf(\gb, h)$.
Similarly (and even in a slightly easier way) one shows that 
\begin{equation}
 | \log \left(\cZ^0_N / \cZ^{\hat \phi}_N\right) | \le M_N N^{d-1/4}\, ,
\end{equation}
and therefore $T_2$ is negligible and  the proof of \eqref{limit2} (hence \eqref{limit}) is complete.

\medskip 

\noindent To prove \eqref{superad} it is sufficient to show that (see \eqref{eq:4.29})
\begin{equation}
\frac{1}{(2N)^d} \bbE\hat \bE^u\left[ \log Z^{\gb,\go,\hat \phi}_{2N,h}  \right] \ge \frac{1}{N^d} \bbE\hat \bE^u\left[ \log Z^{\gb,\go,\hat \phi}_{N,h}  \right].
\end{equation}
Let us divide the box $\gL_{2N}$ into $2^d$ boxes, $\gL^i_N$, $i=1,\dots,2^d$. Set 
\begin{equation}\begin{split}
\gL^i_N&:=\gL_N+(\alpha_1(i),\dots,\alpha_d(i))N \\
\tilde \gL^i_N&:=\tilde \gL_N+(\alpha_1(i),\dots,\alpha_d(i))N \\
\end{split}\end{equation}
where $\alpha_j(i)\in\{0,1\}$ is the $j$-th digit of the dyadic development of $i-1$.
Let $\bP^{\hat\phi,i}_N$ be the law of the free field on $\gL^i_N$ with boundary conditions $\hat \phi$ and set
\begin{equation}
Z^{\gb,\go,\hat \phi,i}_{N,h}:=\bE^{\hat\phi,i}_N\left[ \exp\left( \sum_{x\in  \tilde \gL^i_N} (\gb \go_x-\gl(\gb)+h)\delta_x\right)\right].
\end{equation}
We define 
\begin{equation}
\gG_{N}:=\left( \bigcup_{i=1}^{2^dd} \partial \gL^i_N \right) \setminus \partial\gL_{2N}.
\end{equation}
Now we notice by that if one conditions on the realization on $\phi$ in $\gG_{N}$, 
the partition functions of the system of size $2N$
factors into $2^d$ partition functions of systemes of size $N$, whose boundary conditions are determined 
by $\hat \phi$ and $\phi|_{\gG_N}$.
\begin{multline}
\bE^{\hat \phi}_{2N} \left[ \exp\left( \sum_{x\in  \tilde\gL_{2N}} (\gb \go_x-\gl(\gb)+h)\delta_x\right) \ \Bigg| \ \phi|_{\gG_N} \right]
\\ = \prod_{i=1}^{2^d}\bE^{\hat \phi}_{2N} 
\left[ \exp\left( \sum_{x\in  \tilde\gL^i_{N}} (\gb \go_x-\gl(\gb)+h)\delta_x\right) \ \Bigg| \ \phi|_{\gG_N} \right]
=: \prod_{i=1}^{2^d} \tilde Z^i(\hat \phi , \phi|_{\gG_N} , \go  ).
\end{multline}
Note that by the spatial Markov property for the infinite volume field, each $\tilde Z^i(\hat \phi , \phi|_{\gG_N} , \go )$ has the same distribution as 
$Z^{\gb,\go,\hat \phi}_{N,h}$ (if $\hat \phi$ and $\phi|_{\gG_N}$ have distribution
 $\hat \bE^u$ and $\bE^{\hat \phi}_{2N}$ respectively and the $\go$ are IID).
By Jensen's inequality for $ \bE^{\hat \phi}_{2N} \left[  \cdot \ | \ \phi|_{\gG_N} \right]$
we have
\begin{equation}\begin{split}
\bbE \hat \bE^u \left[ \log Z^{\gb,\go,\hat \phi}_{2N,h} \right] \ge 
 \sum_{i=1}^{2d}   \bbE \hat \bE^u  \bE^{\hat \phi}_{2N}\left[ \log \tilde Z^i(\hat \phi , \phi|_{\gG_N} \ ) \right]
 =2^d \bbE \hat \bE^u \left[ \log Z^{\gb,\go,\hat \phi}_{N,h}\right].
 \end{split}\end{equation} 
 Iterating this inequality we obtains that 
 \begin{equation}
 \label{eq:4.29}
\tf(\gb,h)=\lim_{k\to \infty} \frac{1}{2^{dk}} \frac{1}{N^d}
\bbE \hat \bE^u \left[ \log Z^{\gb,\go,\hat \phi}_{2^ k N,h}\right] \ge \frac{1}{N^d} \bbE \hat
\bE^u \left[ \log Z^{\gb,\go,\hat \phi}_{N,h}\right].
 \end{equation}

%

\end{proof}

\section{A lower bound on the free energy}
\label{lowbno}

In this section we prove the lower bound in part (i) of  Theorem~\ref{th:main}.
The statement is:

\medskip

\begin{proposition}\label{simplelowerbound}
For $d\ge 3$, for any $\gb\in (0 , \ubgb)$, there exists a constant $h_0>0$ (which depends on the dimension and on the law of $\go$)
such that for any $h\in(0,h_0)$
\begin{equation}
 \tf(\gb,h)\,  \ge\,  h^{66d}\, .
\end{equation}
\end{proposition}

\medskip

\begin{rem}
While the constant $66$ is quite arbitrary and is the consequence of some rough approximations made in the proof,
there is a more serious reason why our bound gets worse when the dimension increases. This is due to
boundary effects which are more important in high dimension 
(cf. the isoperimetric inequality). See Section \ref{nonoptimality} for more on this.
\end{rem}

\medskip

We assume in this section that $\gb$ is a fixed positive number and $h$ is close to zero.
Let us set (recall from Section \ref{afffaff} that $\sigma_d$ is the standard deviation of the infinite volume free field)
\begin{equation}
u:=8\sigma_d \sqrt{ d \log N} \quad \text{and} \quad N= h^{-2} \, ,
\end{equation} 
where, without true loss of generality, we are assuming $h^{-2}$ to be integer.
We define the event 
\begin{equation}
\cE_u\,:=\, \left\{\phi\in \bbR^{\bbZ^d}:\,    \phi_x > u/2 \text{ for }  x\in \partial \gL_N
\right\}\,.
\end{equation}
The set  $\cE_u$ plays the role of the set of good boundary conditions. What we are going to show is that 
$\cE_u^c$ has a very small probability and use it to bound its contribution to the partition function.

Proposition~\ref{simplelowerbound} follows from the following two lemmas.
The first takes care of the case of bad boundary conditions:
\medskip

\begin{lemma}\label{lemma1}
For every $\gb>0$ such that $\gl(\gb)< \infty$, there exists $h_0$ such that for every $h\in (0,h_0)$
\begin{equation}\label{badevent}
\bbE \hat \bE^u \left[\log \left(Z^{\gb,\go,\hat \phi}_{N,h}\right) \ind_{\cE_u^c}\right]\,\ge\,
  -C\gl(\gb) N^{d-1}  \bE^u\left[ \delta_0\right]\, ,
\end{equation}
where $C>0$ is a constant that depends only on the dimension.
\end{lemma}

\medskip

The second lemma gives  a lower bound on (a suitable expectation of) $\log Z^{\gb,\go,\hat \phi}_{N,h}$  for  good boundary conditions and it is obtained 
by  considering only 
the contribution of the realizations of $\phi$ which have at most one contact in the box $\tilde \gL_N$: 

\medskip

\begin{lemma}\label{lemma2}
For every $\gb \in (0, \ubgb)$, there exists $h_0$ such that for every $h\in (0,h_0)$
\begin{equation}\label{goodevent}
\bbE \hat \bE^u \left[\log \left(Z^{\gb,\go,\hat \phi}_{N,h}\right)\ind_{\cE_u}\right]\ge \frac{h}{2} N^{d}  \bE^u\left[ \delta_0\right].
\end{equation}
\end{lemma}

\medskip

\noindent
{\it Proof of Proposition~\ref{simplelowerbound}.}
From Lemma~\ref{lemma1}, Lemma~\ref{lemma2} and the choice  $N=h^{-2}$ we have for $h$ small
\begin{equation}\label{adlema}
 \frac{1}{N^d}\bbE \hat \bE^u \left[\log Z^{\gb,\go,\hat \phi}_{N,h}\right]\, \ge\, 
  \left(\frac{h}{2} -C\frac{\gl(\gb)}N\right) \bE^u\left[ \delta_0\right]\,\ge\,  
 \frac{h}{4}\bE^u\left[ \delta_0\right].
\end{equation}
Since for $u$ sufficiently large
\begin{multline}
 \bE^u\left[ \delta_0\right]\, =\, \frac1{\sqrt{2\pi \gs_d^2}}\int_{u-1}^{u+1} \exp\left( -\frac{z^2}{2\gs_d^2}\right)
 \dd z\, \ge\, \frac1{2\sqrt{2\pi \gs_d^2}}  \exp\left( -\frac{\left(u-\frac 12\right)^2}{2\gs_d^2}\right)
\\ \ge\,  \exp\left(-\frac{u^2}{2\sigma_d^2}\right)=N^{-32d}=h^{64d}
 \, ,
\end{multline}
we obtain as a consequence of Proposition \ref{superadd}, and \eqref{adlema}
\begin{equation}
\tf(\gb,h)\, \ge\,  \frac{h}{4}  \bE^u\left[ \delta_0\right]\, \ge\, h^{65d}/4,
\end{equation}
provided $h$ is small enough. \qed

\subsection{Proof of Lemma \ref{lemma1}}
 By Jensen's inequality one has
\begin{equation}
\bbE \hat \bE^u \left[\log \left( Z^{\gb,\go,\hat \phi}_{N,h}\right)\ind_{\cE^c_u}\right]
\, \ge \,
 (-\gl(\gb)+h) \bE^u\left[\sum_{x\in \tilde \gL_N} \delta_x \ind_{\cE^c_u} \right]\, ,
\end{equation}
and therefore  it suffices to show that
\begin{equation}\label{bbound}
\bE^u\left[\sum_{x\in \tilde \gL_N} \delta_x \ind_{\cE_u^c} \right]\le C N^{d-1}  \bE^u\left[ \delta_0\right]\,.
\end{equation}
For every constant $c\ge 1$ we have
\begin{equation}
\label{supersum}
\begin{split}
\bE^u\left[\sum_{x\in \tilde \gL_N}  \delta_x \ind_{\cE_u^c} \right]
\, &\le\,  
\sumtwo{x\in  \tilde \gL_N}{y\in \partial \gL_N}\bE^u\left[\delta_x\ind_{\{\phi_y\le u/2\}}\right] 
\\
& = \, 
\sumtwo{x\in  \tilde \gL_N, y\in \partial \gL_N}{\vert x -y \vert \le c}\bE^u\left[\delta_x\ind_{\{\phi_y\le u/2\}}\right]
+
\sumtwo{x\in  \tilde \gL_N, y\in \partial \gL_N}{\vert x -y \vert > c}\bE^u\left[\delta_x\ind_{\{\phi_y\le u/2\}}\right]
\\
&\le \, 
2dc N^{d-1}
\bE^u\left[ \gd_0\right]+ \sumtwo{x\in  \tilde \gL_N, y\in \partial \gL_N}{\vert x -y \vert > c}\bE^u\left[\delta_x\ind_{\{\phi_y\le u/2\}}\right]
\, ,
\end{split}
\end{equation}
where in the first step we have used the union bound and in the third
we have replaced, in the obvious way, the expectation in the first sum with an upper bound independent of $x$
and $y$ and we have then estimated the cardinality of the set over which the sum is performed.

Now we claim that for $c$ sufficiently large, how large depends only on the dimension $d$, we have 
\begin{equation}
\label{eq:lemok}
\bE^u\left[\delta_x\ind_{\{\phi_y\le u/2\}}\right]\, \le\,  N^{-2d}\bE^u\left[ \delta_0\right]\, ,
\end{equation}
 for every $x,y\in \bbZ^d$ such that
$|x-y|> c$. 
By putting \eqref{supersum} and \eqref{eq:lemok} together
we obtain 
\begin{equation}
\bE^u\left[\sum_{x\in \tilde \gL_N}  \delta_x \ind_{\cE_u^c} \right]\, \le\, 
 \left( 2d c N^{d-1}+ 2d N^{2d-1}N^{-2d} \right)\bE^u\left[ \delta_0\right]\, \le \, 4d c N^{d-1} \bE^u\left[ \delta_0\right]\, .
\end{equation}
Therefore  to complete the proof  Lemma~\ref{lemma1} it suffices to establish \eqref{eq:lemok} and this is what we do now.

We set $x=0$ for notational simplicity and we observe that
\begin{equation}
\bP^u\left(\phi_y\le u/2 \ | \ \delta_0 =1\right)\, \le\,  \max_{z\in [u+1,u-1]} \bP^0\left( \phi_y\ge u/2 \ | \ \phi_0=z \right).
\end{equation}
Under $\bP^0\left( \cdot \ | \ \phi_0=z \right)$, $\phi_y$ is a Gaussian random variable of mean 
equal to $zG(0,y)/\sigma_d^2$ and variance is equal to $G(0,0)-G(0,y)\le \sigma_d^2$.
If $c$ is chosen appropriately we have
\begin{equation}
   \frac{zG(0,y)}{\sigma_d^2}
   \le \, \frac u4\, \,  \  \ \ \  \text{ for every } y \text{ such that }  |y|\ge c\, .
\end{equation}
More explicitly, for $u\ge 2$, it suffices to have $G(0,y) \le \frac 16 \gs_d^2= \frac 16G(0,0)$. 
We can then apply standard Gaussian bounds to obtain
\begin{equation}
\bP^0\left( \phi_y\ge u/2 \ | \ \phi_0=z \right)\le P\left( \mathcal N \ge u/(4\sigma_d)\right)\le e^{-\frac{u^2}{32\sigma_d^2}}\,=\, N^{-2d}\, ,
\end{equation}
and the proof of \eqref{eq:lemok}, hence of Lemma~\ref{lemma1}, is complete.
\qed

\subsection{Proof of Lemma \ref{lemma2}}
As a first step, we are going to prove that 
\medskip

\begin{lemma}\label{smallproba}
For any $\hat \phi\in \cE_u$ and $x,y\in\mathring \gL_N$, $x\ne y$, we have  for every $N$ larger than a constant that depends only on $d$
\begin{equation}\label{smallproba1}
 \bE^{\hat \phi}_N [\delta_x ]\,\le N^{-2d} \ \ \ \  \text{ and } \  \ \ \ 
\bE^{\hat \phi}_N [\delta_y \ | \ \delta_x=1 ]\,\le \, N^{-2d}.
  \end{equation}
\end{lemma}

\medskip

\begin{proof}
From the maximum principle for the discrete harmonic equation \eqref{harmonic} we have
\begin{equation}
 \bE^{\hat \phi}_N(\phi_x)\, =\, H^{\hat\phi}_{\gL_N}(x)\,\ge \, u/2\, .
\end{equation}
The variance of $\phi_x$ is $G_N(x,x)\le \sigma_d^2$. Hence we have 
\begin{equation}
 \bE^{\hat \phi}_N [\delta_x ]\le  \bP^{\hat \phi}_N \left(\phi_x\le 1\right)\, \le\,
  P\left( \mathcal N \ge (u/2-1)/\sigma_d \right)\, \le\,  e^{-\frac{(u-2)^2}{8\sigma_d^2}}\, \le\,  N^{-2d}.
\end{equation}
The second inequality is proved in the same manner: 
conditioned to the value of $\phi_y$ (which we set equal to some arbitrary $z\in [-1,1]$) we want to estimate the variance and expectation of $\phi_y$. By monotonicity of the solution of the harmonic equation \eqref{harmonic}, we may as well restrict to the case 
$\hat \phi \equiv u/2$.
We notice here that, as the escape probability of the simple random walk in $\bbZ^d$ $d\ge 3$ is always larger than $3/5$ (see \cite[Section 5.9]{cf:Finch})
we have for any $N>0$ and any $y\ne x$ in $\mathring{\gL}_N$,
\begin{equation}
\frac{G_N(x,y)}{G_N(x,x)}\, = \, P^y\left(\tau_{\gL_N}<\tau_x\right) \,\le\,  \frac 25.
\end{equation}
As a consequence we can write and bound the mean of $\phi_y$ conditioned to $\phi_x=z$ by 
\begin{equation}
\bE^{u/2}_N \left[ \phi_y \big \vert \, \phi_x=z\right] \, =\, 
z \frac{G_N(x,y)}{G_N(x,x)}+ \frac u2 \left(1-  \frac{G_N(x,y)}{G_N(x,x)} \right) \, \ge\,  \frac u4+1\, .
\end{equation}
Therefore 
\begin{equation}
 \bE^{u/2}_N \left[\delta_y  \ | \ \phi_x=z\right]\le  \bE^{u/2}_N \left[\phi_y\le 1  \ | \ \phi_x=z\right]\ \le P\left[ \mathcal N \ge u/(4\sigma_d)\right]\le e^{-\frac{u^2}{32\sigma^2}}= N^{-2d}\, ,
\end{equation}
which is enough to conclude.
\end{proof}

\medskip

We now go back to the proof of Lemma \ref{lemma2}:
till the end of the proof we will assume $\hat \phi\in \cE^u$.
Set
 \begin{equation}\label{defxi}
 \xi(x):=\exp(\gb \go_x-\gl(\gb)+h)-1.
 \end{equation}
Let $\cA_0$ be  the event that the field $\phi$ has no contact with the defect band and $\cA_1(x)$ the event that it has only one contact at $x$, $\cA_1$ the event that there is a unique contact and 
$\cA_2$ the event that there are two contacts or more.
\begin{equation}
 \begin{split}
  \cA_0&:= \left\{\phi_x\notin [-1,1] \text{ for every } x \in \tilde \gL_N \right\}\, ,\\
  \cA_1(x)&:=  \{\phi_x\in [-1,1]\}\cap \left\{ \phi_y\notin [-1,1]
  \text{ for every } y \in \tilde \gL_N\setminus\{x\}
   \right\},\\
  \cA_1&:= \bigcup_{x\in \tilde \gL_N} \cA_1(x),\\
  \cA_2&:= \bbR^{\tilde \gL_N} \setminus \left(\cA_0\cup \cA_1 \right).
 \end{split}
\end{equation}
One has from Lemma \ref{smallproba} for any $x\in \mathring \gL_N$,
\begin{equation}
\label{eq:thisone}
\bP^{\hat \phi}_N\left(\cA_1(x)\right)\,\ge \, 
\bE^{\hat \phi}_N[\delta_x]-\sum_{y\in \mathring \gL_N \setminus \{x\} }  \bE^{\hat \phi}_N[\delta_x\delta_y]
\ge (1-N^{-d}) \bE^{\hat \phi}_N[\delta_x]\, ,
\end{equation}
where the first inequality is obtained by applying the union bound to 
 $\{\gd_x=1\}= \cA_1 \cup \bigcup_{y\not=x} \{\gd_x=\gd_y=1\}$.
From \eqref{eq:thisone} we directly have 
\begin{equation}\label{A1}
\bP^{\hat \phi}_N\left(\cA_1\right)\,\ge\,  (1-N^{-d})\sum_{x\in \mathring \gL_N }\bE^{\hat \phi}_N [\delta_x]\, .
\end{equation}
Using again the union bound, Lemma \ref{smallproba} and \eqref{A1}, one also has 
\begin{equation}\label{A2}
\bP^{\hat \phi}_N\left(\cA_2\right)\le 
\frac{1}{2}\sumtwo{(x,y)\in (\mathring \gL_N)^2 }{x\ne y }\bE^{\hat \phi}_N [\delta_x\delta_y ]\le
\frac{N^{-d}}{2}\sum_{x\in \mathring \gL_N } \bE^{\hat \phi}_N [\delta_x]
\le N^{-d}\bP^{\hat \phi}_N\left(\cA_1\right)\, .
\end{equation}
Taking only into account the contribution of $\cA_1$ and $\cA_0$ to the partition function one has
\begin{multline}
 Z^{\gb,\go,\hat \phi}_{N,h}\ge \bP_N^{\hat \phi} \left(\cA_0\cup \cA_1\right)\, + \sum_{x\in \mathring \gL_N} \xi(x) \bP_N^{\hat \phi}\left(\cA_1(x)\right)
 \\ 
 =1\, +  \sum_{x\in \mathring \gL_N}\xi(x)\bP_N^{\hat \phi}\left(\cA_1(x)\right)-\bP_N^{\hat \phi} \left(\cA_2\right)=:Z'\, .
\end{multline}
But by \eqref{smallproba1}
\begin{equation}
Z'\ge \bP_N^{\hat \phi} \left(\cA_0\right)\, \ge\, 1- \frac 1{N^d}\, \ge \,   \frac 12 \, ,
\end{equation}
and hence
\begin{equation}
\log  Z' \, \ge\,  (Z'-1)-(Z'-1)^2\, .
\end{equation}
Therefore from \eqref{A2} one has
\begin{equation}\label{esperance}
\bbE[Z'-1]=(e^{h}-1) \bP_N^{\hat \phi}\left(\cA_1\right)-\bP_N^{\hat \phi} \left(\cA_2\right)\ge (e^{h}-1-N^{-d}) \bP_N^{\hat \phi}\left(\cA_1\right) \, .
\end{equation}
We also have (using \eqref{smallproba1})
\begin{equation}\label{variance}
\Var_{\bbP}\left(Z'\right)\, = \, e^{2h}\left(e^{\gl(2\gb)-2\gl(\gb)}-1\right) \sum_{x\in \mathring \gL_N}\bP_N^{\hat \phi}\left(\cA_1(x)\right)^2 
\,\le \, C N^{-2d}\bP_N^{\hat \phi}\left(\cA_1\right)\, , 
\end{equation}
which is much smaller than $\bbE[Z']-1$ (recall $N=h^{-2}$).
Overall (combining \eqref{esperance} \eqref{variance} and \eqref{A1}) one has for all $\hat \phi$ in $\cE_u$
\begin{multline}
\bbE\left[ \log  Z^{\gb,\go,\hat \phi}_{N,h} \right]\ge 
\bbE\left[ \log  Z'\right] \ge 
\bE_N^{\hat \phi}\left[ Z'-1\right]-\left( \bE_N^{\hat \phi}\left[ Z'-1\right]\right) ^2- \Var_{\bbP}\left(Z'\right)
\\ 
\ge (e^{h}-1-N^{-d}) \bP_N^{\hat \phi}\left(\cA_1\right)-
\left(\left(e^{h}-1-N^{-d}\right) \bP_N^{\hat \phi}\left(\cA_1\right)\right)^2- C N^{-2d}\bP_N^{\hat \phi}\left(\cA_1\right)\, ,
\end{multline}
By using again
 $N=h^{-2}$, as well as \eqref{A1},  we  deduce that for $h$ sufficiently small
\begin{equation}
 \bbE\left[ \log  Z^{\gb,\go,\hat \phi}_{N,h} \right]\ge \frac{3h}{4}\left(\sum_{x\in \mathring \gL_N }\bE^{\hat \phi}_N [\delta_x]\right)\, .
 \end{equation}
Hence
\begin{multline}
\bbE\hat \bE^u\left[ \log  \left( Z^{\gb,\go,\hat \phi}_{N,h} \right)\ind_{\cE^u}\right]\ge \frac{3h}{4} \bE^{u}\left[  \sum_{x\in \mathring \gL_N} \delta_x \ind_{\cE^u}\right]\\
= \frac{3h}{4}\left( (N-1)^d  \bE^{u}\left[   \delta_0\right]- \bE^{u}\left[  \sum_{x\in \tilde \gL_N} \ \delta_x \ind_{\cE^c_u}\right]\right)
\ge \frac h2 N^d   \bE^{u}\left[ \delta_0\right],
\end{multline}
where in the last inequality we used $h$ small and \eqref{bbound}. 
\qed

\subsection{Why is this not optimal?}\label{nonoptimality}

The main idea of the proof above is to change the boundary conditions 
$\hat \phi$ so that there are only a few contacts 
(the main contribution to the partition function is given by $\cA_0\cup \cA_1$).
In that case the partition function (or at least $Z'$)
has a very small variance and for this reason, Jensen's inequality for $\log $ 
is essentially sharp. The strategy could in principle (and with a lot of efforts)
extend if one has typically a bounded number of contact, or possibly if one allows it to grow logarithmically ,
but the variance estimates would clearly blow up beyond this point.

 So for this strategy to work we need  that 
 $\bE^u\left[\delta_0\right]\le c N^{-d}$ for some positive constant $c$ (possibly large or even growing very slowly with $N$
 but this latter possibility does not add much to the discussion). 
Now let us notice that since $\delta_x$ on the boundary is completely determined by the boundary conditions
one can easily take away from the $\log$ partition function the contribution of the boundary
\begin{multline}\label{boundaryeffectseparated}
\bbE \hat \bE^u \left[\log Z^{\gb,\go,\hat \phi}_{N,h}\right]\, =\\
(h-\gl(\gb))\sum_{x\in \tilde \gL_N \cap \partial \gL_N}\bE^u[\delta_x]
+
\bbE \hat \bE^u \log \bE^{\hat \phi}_N\left[ \exp\left( \sum_{x\in  \mathring\gL_N} (\gb \go_x-\gl(\gb)+h)\delta_x\right)\right]\, .
\end{multline}
The first term which is the boundary effect is negative ($h$ is small!) and of order $N^{d-1}\bE^u[\delta_0]$.
Our best hope for the second term is to get something positive which is of order 
$h N^d \bE^u[\delta_0]$  (this is what we get with an annealed bound).
Hence for the second term in \eqref{boundaryeffectseparated} to be dominant, we need $h$ to be 
larger than $N^{-1}$. The best we can hope as a lower bound for the free energy density is then 
\begin{equation}
h \bE^u[\delta_0]\, \le \, c\, h N^{-d}\, =\, O\left( h^{d+1}\right)\, .
\end{equation}

 To reduce the influence of boundary effects, one has to work with larger boxes, 
 but in this case 
 the total number of contact in the box will be large and one has to try an find other means of 
 controlling $\bbE \left[\log Z\right]$ than only the variance. This is the aim of the coarse graining and  replica-coupling approach
 of the next section.

\section{The coarse graining procedure for the critical behavior (lower bound)} \label{lowbo}

For $\gs=\gs_d$, $a>0$ and $h>0$ small we set
\begin{equation}
\label{eq:u}
u\, =\, u(a,h)\,:=\, \gs \sqrt{2 \log(1/h)}+1- \frac{\gs}2 \frac{\log  \log(1/h) }{\sqrt{2 \log(1/h)}}- 
\gs  \frac{\log \left(2a \sqrt{\pi}\right)}{\sqrt{2 \log(1/h)}}\, .
\end{equation}
Such a choice has been made
to guarantee that the contact probability is (essentially) $a\, h$ for $h$
small. The choice \eqref{eq:u} is clearly connected to the following lemma on a standard Gaussian variable $\cN$: 

\medskip

\begin{lemma}
\label{th:Z}
If $v: (0, \infty) \to \bbR$ is such that $\lim_{h \searrow 0} v(h) \sqrt{\log(1/h)}=0$, then 
\begin{equation}
\label{eq:Z}
\bP\left[ u(a,h)+v(h)+\gs \cN \in [-1,1]\right]\stackrel{h \searrow 0} \sim  ah \, .
\end{equation}
 \end{lemma} 

\medskip

\noindent
{\it Proof.}
The result follows, via a lengthy computation, 
from the well known asymptotic ($x \nearrow \infty$) estimate 
\begin{equation}
\label{eq:asymptZ}
\bP(\cN>x) \,=\, \frac 1{x\sqrt{2\pi}} 
\exp\left(-\frac{x^2}{2}\right) \left( 1+ O\left( \frac1{x^2}\right) \right)\, .
\end{equation}
\qed
\medskip

\begin{rem}
\label{rem:cop3}
It is easy to see that the statement of Lemma~\ref{th:Z} holds also
if we replace $[-1,1]$ with $[c,1]$, any $c<1$. More interestingly for us, it holds also for 
$(-\infty,1]$ (and thus also for $(-\infty,0]$ provided that 
$u(a,h)$ is replaced by $u(a,h)-1$). This remark is important because it ultimately means 
that the strategy of the proof that we are going to present below works also if we replace
$\gd_x$ with $\gD_x$, that is if we pass from disordered pinning to the co-membrane model.
This is  true also for the rougher proof of Section~\ref{lowbno} and in a more evident way
since no estimate is sharp in that case.
\end{rem}

\medskip

Let us introduce
\begin{equation}
 \gr_h\, :=\, \exp \left(- \sqrt{\log (1/h) }\right)\  \text{ and } \ 
N_0\,:= \, \frac 1{ \gr_h} \, ,
\end{equation}
and without true loss of generality we will assume $N_0\in2 \bbN$. 
 For $h \searrow 0$ and arbitrary $b>0$
we have $$h^b \ll  \gr_h\ll \vert \log h\vert^{-1/b}.$$
We then choose
$N_1$ such that $N_1/N_0 \in 2\bbN$,  $N_1/N_0\ge 4$ and 
\begin{equation}
N_1\, \in  \, \left[\frac 12 h^{-3}, h^{-3}\right] \, .
\end{equation}

We aim at showing 
\medskip

\begin{proposition}
\label{th:fs}
Choose $\gb>0$.
There exist $a(\gb,d)>0$ and  $c(\gb,d)>0$  such that for $h>0$ sufficiently small 
we have 
\begin{equation}
\label{eq:fs}
\frac 1{N_1^d} \bbE \hat \bE ^{u(a,h)} \left[ \log Z_{N_1, \gb, h}^{\go} \right]  \, \ge \, c(\gb,d) \,  h^2 \, .
\end{equation}
Moreover one find a constant $C(d)$ such that forall $\gb\in(0, 1]$ we have 
\begin{equation}
  c(\gb,d)\ge C(d) \gb^2
\end{equation}

\end{proposition}

\medskip

\noindent
{\it Proof.}
The proof is done in several steps: a number of lemmas will be stated and proved after the main body of the proof.

\subsubsection{Step 1: Smoothing the roughness of $\hat \phi$}
We start off by selecting a subset of the $\hat \phi$ configuration (of $\hat \bP^u$ probability very close to one)
that guarantees that harmonic averages of the boundary value are extremely close to $u$, at least when we are not 
too close to the boundary. 

 We do this  by introducing  the event
\begin{equation}
\label{eq:Bset}
B_{u}\, :=\, \left\{
\hat \phi \in \bbR^{\bbZ ^d}:\, \left \vert H^{\hat \phi}_{\gL_{N_1}}(x) -u \right \vert \,\le 
\, \frac{ \gr_h^{\frac 18}}{2} \text{ for every } x \in \gL_{N_1} \text{ s.\,t. } d\left(x,\partial \gL_{N_1}\right) 
\ge \frac{N_0}2 \right\}\, ,
\end{equation}
where $H^{\phi}_{\gL}(x)$ is the solution to the harmonic equation \eqref{harmonic} in $\gL_{N}$,
  with $\phi$ boundary conditions.

We prove the following estimate at the end of the section.
\smallskip

\begin{lemma}
\label{th:Best}
For  $h$ sufficiently small
\begin{equation}
\hat \bP ^u \left( B_{u}^\complement\right) \, \le \, \exp\left( - \gr_h^{-\frac 1{5}}\right)\, .
\end{equation}
\end{lemma}

\medskip

In other terms, $\hat \bP ^u \left( B_{u}^\complement\right)$ is smaller than any power of $h$.
As it stands, Lemma~\ref{th:Best} is stated and will be used for the value of $u$ given in \eqref{eq:u}, but 
 it is
easy to realize that $\hat \bP ^u \left( B_{u}\right)$ does not depend on $u$ and therefore   Lemma~\ref{th:Best}
holds  uniformly in $u$.

\medskip
Since by Jensen's inequality for $h \ge 0$
\begin{equation}
 \bbE  \left[ \log Z_{N_1, \gb, h}^{\go, \hat \phi}\right] \, \ge \,  \sum_{x \in \tilde \gL_{N_1} } 
 \left(h - \frac{\gb ^2} 2 \right )  \bE_{N_1}^{\hat \phi}[\gd_x]  \, \ge \, - \frac{\gb^2}2  N_1 ^d\, , 
 \end{equation}
we readily see that
\begin{equation}
\label{eq:fs-1}
\frac 1{N_1^d} \bbE \hat \bE ^u \left[ \log Z_{N_1, \gb, h}^{\go} ; \, B_{u}^\complement \right]  \, \ge  \,
- \frac{\gb^2}2 \hat \bP ^u \left( B_{u}^\complement\right)  \, .
\end{equation}
Therefore, in view of Lemma~\ref{th:Best} and of the result 
\eqref{eq:fs} we are after, it suffices to show that
\begin{equation}
\label{eq:fs2}
\frac 1{N_1^d} \bbE \hat \bE ^u \left[ \log Z_{N_1, \gb, h}^{\go} ; \, B_{u}\right]  \, \ge \, 2c\,  h^2 \, .
\end{equation} 

\medskip 

\subsubsection{Step 2: Neglecting the  energy contribution near $\partial \gL_{N_1}$}
We show now that we can neglect the energy contribution coming from
the sites on which we do not have a control on the harmonic average of the boundary. 
This is done  by introducing 
\begin{equation}
\gL^{-}_{N_1,N_0}\,:= \, 
\left\{N_0+1,N_0+2, \ldots , N_1-N_0\right\}^d\, ,
\end{equation}
by restricting the sum in the energy term to sites in $\gL^{-}_{N_1,N_0}$: as we are going to show right away is that this introduces an error in the free energy computation that is $o(h^2)$, hence irrelevant.
This space between the boundary of the box $\tilde \gL_{N_1}$ and the sites that contribute to the energy has been introduced
to allow some averaging of the boundary conditions $\hat \phi$. In fact $\hat \phi$ has fluctuations  of order one and therefore the field
$\phi$ close to the boundary has a mean that inherits this incertitude, while  sufficiently far away -- a distance $N_0$ suffices -- the mean will be $u$
up to an error which is $O(h^b)$, any $b>0$.
The estimate for the error introduced by restricting the energy contribution   to sites in $\gL^{-}_{N_1,N_0}$ goes as follows:
start by observing that 
\begin{multline}
\sum_{x \in \tilde \gL_{N_1}} \left( \gb \go_x -\frac {\gb^2}2 +h \right) \gd_x \, \ge \,
-
\sum_{x \in  \tilde \gL_{N_1}\setminus \gL^-_{N_1,N_0}} \left \vert \gb \go_x -\frac {\gb^2}2 +h \right\vert\\
+\, 
\sum_{x \in  \gL^-_{N_1,N_0}} \left( \gb \go_x -\frac {\gb^2}2 +h \right) \gd_x  \, ,
\end{multline}
so that,
\begin{multline}
\label{eq:fs3}
\frac 1{N_1^d} \bbE \hat \bE ^u \left[ \log Z_{N_1, \gb, h}^{\go, \hat \phi} ; \, B_{u}\right]  \, \ge \, -
\frac {\left \vert \tilde \gL_{N_1}\setminus \gL^-_{N_1,N_0}\right\vert} {N_1^d} \bbE 
\left[\left \vert \gb \go_0 -\frac {\gb^2}2 +h \right\vert\right]
\\ +\, 
\frac 1{N_1^d} \bbE \hat \bE ^u \left[  \log \bE_{N_1}^{\hat \phi} 
\left[ \exp\left(\sum_{x \in  \gL^-_{N_1,N_0}} \left( \gb \go_x -\frac {\gb^2}2 +h \right) \gd_x\right)
\right]
; \, B_{u}\right]
 \, .
\end{multline} 
Therefore the first term on the right-hand side is $O(N_0/N_1)\ll h^2$ (this has determined our  choice of $N_1$), so that
we can effectively neglect the energy contribution of the   sites outside  $\gL^{-}_{N_1,N_0}$ and
Proposition~\ref{th:fs} reduces to showing
\begin{equation}
\label{eq:fs4}
\frac 1{N_1^d} \bbE \hat \bE ^u \left[  \log \bE_{N_1}^{\hat \phi} 
\left[ \exp\left(\sum_{x \in  \gL^-_{N_1,N_0}} \left( \gb \go_x -\frac {\gb^2}2 +h \right) \gd_x\right)
\right]
; \, B_{u, }\right]
 \, \ge \, 3c \, h^2\, .
\end{equation}
This estimate will be obtained by restricting the $\bP_{N_1}^{\hat \phi}$-expectation to an event
$A_{ \kappa}$, $\kappa$ a positive integer (given explicitly just below) that depends only on $d$:
$\kappa$ is a constraint that we are going to introduce on the number of contacts in \emph{intermediate} scale boxes.

\medskip

\subsubsection{Step 3: the coarse graining grid and the event $A_{ \kappa}$}
To define $A_{ \kappa}$
we first introduce
a decomposition of $\gL^-_{N_1,N_0}$ (Figure~\ref{fig:N1N2} may be of help in following the construction).
For $w \in \{0,1\}^d$ we set
\begin{equation}
\gL_{N_1, N_0}^w\, :=\, \left \{x \in \gL^-_{N_1,N_0}:\, 
\left \lceil \frac{x_i}{N_0} \right\rceil \stackrel{\text{mod}\, 2}= w_i \text{ for }
i=1,2, \ldots, d
\right\}\, .
\end{equation}
\begin{figure}[h]
\begin{center}
\leavevmode
\epsfxsize =9 cm
\psfragscanon
\psfrag{0}[c][l]{\small $0$}
\psfrag{N1}[c][l]{ \small $N_1$}
\psfrag{tN0}[c][l]{ \small $2N_0$}
\psfrag{N0}[c][l]{ \small $N_0$}
\psfrag{C25}[c][l]{\small $\bbC_{(2,5)}$}
\psfrag{C67}[c][l]{\small $\bbC_{(6,7)}$}
\psfrag{B27}[c][l]{\small $\bbB_{(6,3)}$}
\epsfbox{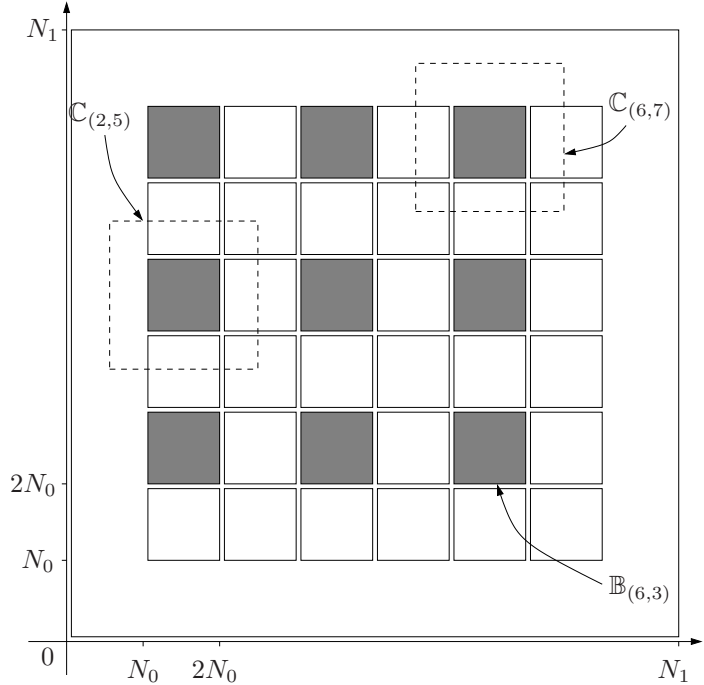}
\end{center}
\caption{\label{fig:N1N2} 
The set $\tilde \gL _{N_1}$ is drawn for $d=2$ (we need it only for $d\ge 3$, but for illustration purposes 
$d=2$ is enough) and $N_1=8N_0$. The set $\gL^{-}_{N_1,N_2}$ is the (disjoint union) of the 
$B_j$'s boxes, $j\in\{2,3,4,5,6,7\}^2= \cup_{w\in\{0,1\}^2} \cJ_w$. We have singled out $\gL_{N_1,N_0}^{(0,1)}$
by making it darker: observe that $\gL_{N_1,N_0}^{-}$ is the disjoint union of $\gL_{N_1,N_0}^{w}$,$w\in \{0,1\}^2$.
We have also drawn, with dashed boundaries, some of the $\bbC_j$ boxes.
}
\end{figure}
Actually $\gL_{N_1, N_0}^w$ can be seen as a disjoint union of $\left(\frac{N_1}{2 N_0} -1\right)^d$
(hyper)cubes of edge length $N_0$:
\begin{equation}
\gL_{N_1, N_0}^w \, =\, \cup_{ j\in \cJ_w} \bbB_j \, ,
\end{equation}
where
\begin{equation}
 \bbB_j:=\tilde \gL_{N_0}+N_0 j
\end{equation}
Even if it is probably not necessary, we can make $\cJ_w$ explicit:
\begin{equation}
\cJ_w\, =\, 
\left \{ 
\left( \left \lceil \frac{x_1}{N_0} \right\rceil , \ldots, \left\lceil \frac{x_d}{N_0} \right\rceil \right)
:\, x \in \gL^-_{N_1,N_0} \text{ and }\left \lceil \frac{x_i}{N_0} \right\rceil\stackrel{\text{mod}\, 2}= w_i \text{ for }
i=1,2, \ldots, d
\right\}\,,
\end{equation}
Note that 
\begin{equation}
\bigcup_{w \in \{0,1\}^d} \gL_{N_1, N_0}^w= \gL_{N_1, N_0}
\end{equation} 
and therefore 
\begin{equation} \gL_{N_1, N_0}= \cup_{j\in \cJ} \bbB_j 
\end{equation} 
where $\cJ= \cup_w \cJ_w=\left\{2, \dots , N_1/N_0-1\right\}^d$.
Last, for $j \in \cJ$ we set
\begin{equation}
\bbC_j:= \{x:\, \mathrm{dist}(x, \bbB_j)\le N_0/2\}.
\end{equation}
We are now ready to introduce $A_{ \kappa}:= A^{(1)} \cap A^{(2)}_{\kappa}$, where
\begin{equation}
\label{eq:A1}
A^{(1)}\, :=\, \left \{ \phi:\, \max_{j \in \cJ}\max_{x \in \bbB_j} \left \vert H^\phi_{\bbC_j} (x) -u \right \vert \, 
\le \,  \gr_h^{-\frac 18} \right\} \, , 
\end{equation}
and 
\begin{equation}
\label{eq:A2}
A^{(2)}_{\kappa} \, :=\, \left \{ \phi:\,
\max_{j \in \cJ} \sum_{x \in \bbB_j} \gd_x \, \le\,  \kappa
\right\}\, .
\end{equation}
\medskip

Recall that $u$ is chosen in \eqref{eq:u} and $B_{u}$ is given in \eqref{eq:Bset}.
The argument that we are going to present works for $\kappa$ sufficiently large and how large just depends on the dimension.
For definiteness we make the choice of $\kappa$ explicit:   
$$\kappa:= \lceil d^3 3^{12} 2^{12(d+5)} \mathsf{c}_d^{12}\rceil,$$ 
where $\mathsf{c}_d\ge 1$ is the constant in 
\eqref{eq:Green-bound}. Such a choice 
 stems out of several arbitrary and lazy choices in the chain of rough bounds that constitutes the proof. 
 \medskip
 
We introduce now an important technical estimate whose proof is postponed to the end of the section.

\medskip

\begin{lemma}
\label{th:Aest}
 For $h$ sufficiently small
\begin{equation}
\label{eq:Aest}
\sup_{\hat \phi \in B_{u}}\bP_{N_1}^{\hat \phi} \left( A_{ \kappa}^\complement \right) \, \le \, 
h^{k^{1/4}}\, .
\end{equation}
\end{lemma}

\medskip

\subsubsection{Step 4: Replica coupling}
Going back to \eqref{eq:fs4}, it is clear that it suffices to show that for 
$h$ sufficiently small
\begin{equation}
\label{eq:fs5}
\frac 1{N_1^d} \bbE \hat \bE ^u \left[  \log \bE_{N_1}^{\hat \phi} 
\left[ \exp\left(\sum_{x \in  \gL_{N_1,N_0}} \left( \gb \go_x -\frac {\gb^2}2 +h \right) \gd_x\right); \, A_{ \kappa}
\right]
; \, B_{u}\right]
 \, \ge \, 3c \, h^2\, .
\end{equation}

For this we will exploit a replica coupling argument bound. The following result, which we prove in the Appendix,
is inspired by and   very close to the computations made in \cite{cf:Trep} for renewal pinning. 
The proof exploits interpolation technics similar to those 
found in spin-glass literature \cite{cf:GuTo}.

\medskip

\begin{lemma}
\label{th:replica}
\begin{equation}
\label{eq:replica}
\frac 1{N_1^d} \bbE \left[  \log \bE_{N_1}^{\hat \phi} 
\left[ \exp\left(\sum_{x \in  \gL_{N_1,N_0}} \left( \gb \go_x -\frac {\gb^2}2 +h \right) \gd_x\right); \, A_{ \kappa}
\right]
\right]
 \, \ge  \, T_1 -T_2 \, ,
 \end{equation}
 with
 \begin{equation}
T_1 \, :=\,  \frac 1{N_1^d}
 \log \bE_{N_1}^{\hat \phi} 
\left[ \exp\left(h \sum_{x \in  \gL_{N_1,N_0}}  \gd_x\right); \, A_{ \kappa}
\right] \, ,
\end{equation}
and
\begin{equation}
T_2\, :=\, 
  \frac 1{N_1^d} \log \left \langle \exp \left( 2 \gb^2 \sum_{x \in  \gL_{N_1,N_0}} \gd^{(1)}_x
\gd^{(2)}_x \right); \, A_{ \kappa}^2 \right \rangle^{\otimes 2}_{N_1,h, \hat \phi }\, ,
\end{equation}
with
\begin{equation}
\left \langle \, \cdot \, 
\right \rangle_{N_1,h, \hat \phi }\, :=\, 
\frac{ 
\bE_{N_1}^{\hat \phi} 
\left[ \, \cdot \,  \exp\left(h \sum_{x \in  \gL_{N_1,N_0}}  \gd_x\right)
\right] 
}
{\bE_{N_1}^{\hat \phi} 
\left[ \exp\left(h \sum_{x \in  \gL_{N_1,N_0}}  \gd_x\right)
\right] 
}\, .
\end{equation}
\end{lemma}

\medskip

\begin{rem}
It is obvious from the proof that the above Lemma remains valid without restriction to the event $A_{ \kappa}$
(or with a restriction to another event). 
However it is not too difficult to check that without this restriction, the quantity $T_2$ would be of order $\gb^2$
and hence the result completely useless (the  right hand side of \eqref{eq:replica} would be negative).
We have designed the event $A_{\kappa}$ to be of small probability 
so that $T_1$ is almost equal to to the value it would have with no conditioning, but such that $T_2$ becomes much smaller with the 
conditioning.
\end{rem}

\medskip

We now need a lower bound on $\bE ^u [T_1; \, B_{u}]$ and an upper bound on $\bE ^u [T_2; \, B_{u}]$.

\medskip

\subsubsection{Step 5: lower bound on $\bE ^u [T_1; \, B_{u}]$}
We apply Jensen's inequality after a rearrangement 
\begin{multline}
T_1 \, =\, 
\frac 1{N_1^d}
 \log \bE_{N_1}^{\hat \phi} 
\left[ \exp\left(h \sum_{x \in  \gL_{N_1,N_0}}  \gd_x\right) \Bigg \vert  \, A_{ \kappa}
\right] + \frac 1{N_1^d}
 \log \bP_{N_1}^{\hat \phi} \left( A_{ \kappa}\right) \\
 \ge \, 
 \frac h{N_1^d}
\bE_{N_1}^{\hat \phi} 
\left[  \sum_{x \in  \gL_{N_1,N_0}}  \gd_x \, \bigg \vert  \, A_{ \kappa}
\right] + \frac 1{N_1^d}
 \log \bP_{N_1}^{\hat \phi} \left( A_{ \kappa}\right) \, =:\, T_{1,1}+T_{1,2}.
\end{multline}
We have 
\begin{equation} 
\begin{split}
\hat \bE ^u [T_{1,1}; \, B_{u}] \, &\ge \, \frac h{N_1^d}
\hat \bE ^u \left[\bE_{N_1}^{\hat \phi} 
\left[  \sum_{x \in  \gL_{N_1,N_0}}  \gd_x \, ;   \, A_{ \kappa}
\right]; \, B_{u}\right]\\ 
&= \,
\frac h{N_1^d}
 \bE ^u
\left[  \sum_{x \in  \gL_{N_1,N_0}}  \gd_x \, ;   \, A_{ \kappa}
 \cap B_{u}\right]
 \\
 &\ge  \,  
 \frac h{N_1^d}
 \bE ^u
\left[  \sum_{x \in  \gL_{N_1,N_0}}  \gd_x \, \right] - h \left( \bP^u \left( (A_{ \kappa}\cap B_{u})^\complement \right) 
\right) \, ,
\end{split}
\end{equation}
so that by using Lemma~\ref{th:Z}, Lemma~\ref{th:Best} and Lemma~\ref{th:Aest} (recall the choice of $\kappa$)
we readily see that 
\begin{equation} 
\hat\bE ^u [T_{1,1}; \, B_{u}] \, \ge \frac {2a}3 h^2
\, ,
\end{equation}
for $h$ sufficiently small.

On the other hand, by Lemma~\ref{th:Aest} we see that 
$\bP^{\hat \phi} \left( A_{ \kappa}^\complement \right)\ge 1/2$ for $h$ small, uniformly in $\hat \phi \in B_u$ and this entails 
 $\hat\bE ^u [T_{1,2}; \, B_{u}] \ge -8(\log 2) h^{3d}$ so 
\begin{equation} 
\label{eq:T1B}
\hat\bE ^u [T_1; \, B_{u}] \, \ge \frac {a}2 h^2
\, .
\end{equation}

\medskip

\subsubsection{Step 6: upper bound on $\bE ^u [T_2; \, B_{u}]$}
We start with the preliminary observation that $\phi\in A^{(2)}_{\kappa}$
means that there are at most $\kappa$ contacts in $\bbB_j$ for every $j \in \cJ$ and this implies 
that there are at most $3^d \kappa$ contacts in $\bbC_j \cap \gL_{N_1, N_0}$ (and a fortiori in $\mathring{\bbC}_j \cap \gL_{N_1, N_0}$)
because if  $\bbC_j \cap \gL_{N_1, N_0}=\bbC_j$ then $\bbC_j$ is covered by $3^d$ $\bbB_j$'s (this is the typical case: in the bulk).
When $\bbC_j\cap \gL_{N_1, N_0}$ is (strictly) contained in $\bbC_j$ (the boundary case), fewer $\bbB_j$ boxes suffice. 
Therefore for every $j\in \cJ$ we introduce the event
\begin{equation}
A^{(3)}_\kappa (j)\, :=\, \left\{\sum_{x\in \mathring{\bbC_j} \cap \gL_{N_1, N_0} } \gd_x \le  3^d \kappa
\ \text{ and } \  \sum_{x\in \bbB_j  } \gd_x \le  \kappa.
\right\}\, ,
\end{equation}
One can check that
\begin{equation}
A_{ \kappa}:= A^{(1)} \cap A^{(2)}_{\kappa}
\,=\, A^{(1)} \cap \left(\cap_{j \in \cJ}A^{(3)}_{\kappa}(j)\right)\, ,
\end{equation} 
Now we apply the H\"older inequality $\vert \bE  \prod_{i=1}^{k} X_i \vert \le \prod_{i=1}^{k} (\bE \vert X_i \vert^{k})^{1/k}$
\begin{equation}
\label{eq:Holder}
\begin{split}
T_2\, &=\, 
  \frac 1{N_1^d} \log \left \langle \prod_{w \in\{0,1\}^d}
\exp \left( 2 \gb^2 \sum_{x \in  \gL^w_{N_1,N_0}} \gd^{(1)}_x
\gd^{(2)}_x \right); \, A_{ \kappa}^2 \right \rangle^{\otimes 2}_{N_1,h, \hat \phi }
\\
& \le \, 
 \frac 1{(2N_1)^d} \sum_{w \in\{0,1\}^d} \log
 \left \langle 
\exp \left( 2^{1+d} \gb^2 \sum_{x \in  \gL^w_{N_1,N_0}} \gd^{(1)}_x
\gd^{(2)}_x \right); \, A_{ \kappa}^2 \right \rangle^{\otimes 2}_{N_1,h, \hat \phi }\, .
\end{split}
\end{equation}
Let us focus on  the argument of the logarithm and  condition the measure 
$\langle \cdot \rangle _{N_1,h, \hat \phi }$ to the $\gs$-algebra generated by $\{\phi_x\}_{x\in \cup_{j \in \cJ_w} \partial \bbC_j}$.
By the spatial Markov property we obtain
\begin{multline}
\label{eq:HolderMarkov}
\left \langle
\exp \left( 2^{1+d} \gb^2 \sum_{x \in  \gL^w_{N_1,N_0}} \gd^{(1)}_x
\gd^{(2)}_x \right); \, \left( \phi^{(1)},
\phi^{(2)}\right) \in 
A_{ \kappa}^2 \right \rangle^{\otimes 2}_{N_1,h, \hat \phi }\, \le \\
\bigg \langle
\prod_{j\in \cJ_w}
 \bE^{\otimes 2}_{\bbC_j,h,(\phi^{(1)},\phi^{(2)})}
 \left[ \exp \left( 2^{1+d} \gb^2 \sum_{x \in  \bbB_j} \gd^{(1)}_x
\gd^{(2)}_x \right); 
 \left(A_\kappa^{(3)}(j)\right)^2
 \right]  
 ;
 \\
 \, \left( \phi^{(1)},
\phi^{(2)}\right) \in \left(A_\kappa^{(1)}\right)^2
 \bigg \rangle^{\otimes 2}_{N_1,h, \hat \phi }\, ,
\end{multline}
where
\begin{equation}
 \bE^{\otimes 2}_{\bbC_j,h,(\phi^{(1)},\phi^{(2)})} \, =\,
 \bE_{\bbC_j,h,\phi^{(1)}}\bE_{\bbC_j,h,\phi^{(2)}}\, ,
\end{equation}
and
\begin{equation}
\bE_{\bbC_j,h,\phi}\left[\, \cdot\,\right]:=
 \left. \bE_{\bbC_j,\phi}\left[\, \cdot\,\exp\left(h \sum_{x \in \mathring{\bbC_j}} \gd_x\right)\right] \middle/
  \bE_{\bbC_j,\phi}\left[\exp\left(h \sum_{x \in \mathring{\bbC_j}} \gd_x\right)\right]\right.
  \end{equation}
and $\bP_{\bbC_j,\phi}$ is the free field on the set $\bbC_j$ with boundary conditions $\phi$ (remark that this notation is a bit improper 
since $\phi$ is used for the boundary conditions and in the definition of $\delta$, but we believed this is more readable that 
introducing 
$\tilde \delta$ and should not generate confusion).
Of course if $F: \bbR^{\mathring{\bbC_j}} \mapsto [0,\infty)$ is measurable then
 $ \bE_{\bbC_j,h,\phi}[F]$ is measurable with respect to $\{\phi_x\}_{x \in \partial \bbC_j}$.

We now recall that we need to bound from above $\hat\bE ^u [T_2; \, B_{u}]$
and we will do this by taking the supremum over $\phi \in A^{(1)}$ (and  applying $\hat\bE ^u [\, \cdot \, ; \, B_{u}]$
will be irrelevant), that is by 
\eqref{eq:Holder}
and \eqref{eq:HolderMarkov} we have
\begin{multline}
\label{eq:almid}
\hat\bE ^u [T_2; \, B_{u}]\, \le \\
 \frac 1{N_1^d} \sum_{j \in \cJ}
  \sup_{\phi^{(1)}, \phi^{(2)}  \in A^{(1)}}
 \log 
 \bE^{\otimes 2}_{\bbC_j,h,(\phi^{(1)},\phi^{(2)})}\left[ \exp \left( 2^{1+d} \gb^2 \sum_{x \in  \bbB_j} \gd^{(1)}_x
\gd^{(2)}_x \right);
 \left(A_\kappa^{(3)}(j)\right)^2
 \right]  \, ,
\end{multline}
and we are left with estimating the terms in the sum in the right-hand side.
These terms are actually identical, except for the boundary cases (but they can be bounded in the very same way).
Let us record the first part of the argument as a lemma, that we will also use in the next sections.

\begin{lemma}\label{lemmeduX}
 Let $X$ be a positive random variable such that $X\le \gamma$ with probability $1$.
 Then
 \begin{equation}
   \log \bbE \left[e^X\right]\le \left(e^\gamma-1\right)\bbE[X].
 \end{equation}
\end{lemma}
\begin{proof}
 We simply use convexity to show that 
 $$e^X\le 1+(e^\gamma-1)X,$$
 and the inequality $\log (1+u)\le u.$
\end{proof}
Now applying the Lemma to $X:= 2^{1+d} \gb^2\sum_{x \in  \bbB_j} \gd^{(1)}_x
\gd^{(2)}_x$, and remarking that on $A_\kappa^{(3)}(j)$, 
we have $X\le 2^{1+d} \gb^2\kappa$, we have
\begin{multline}
\bE^{\otimes 2}_{\bbC_j,h,(\phi^{(1)},\phi^{(2)})}
\left[ \exp \left( 2^{1+d} \gb^2 \sum_{x \in  \bbB_j} \gd^{(1)}_x
\gd^{(2)}_x \right);
 \left(A_\kappa^{(3)}(j)\right)^2
 \right] \,
 \\ \le
 1+ c(\gb, \kappa,d) \sum_{x \in  \bbB_j} \sup_{\phi\in A^{(1)}} \bE_{\bbC_j,h,\phi}\left[\gd_x; A_\kappa^{(3)}(j)\right]^2\, ,
\end{multline}
with
\begin{equation} 
c(\gb, \kappa,d)\,:=\, \frac{\exp \left( 2^{1+d} \gb^2\kappa \right) -1}{\kappa}\, . 
\end{equation}
In the case $\beta\in(0,1]$ notice that we have 
\begin{equation}\label{beta1}
c(\gb, \kappa,d)\,\le \, c(\kappa,d)\gb^2\, .
\end{equation}
But by using the definition of $ \bE_{\bbC_j,h,\phi}[\cdot]$, since $\bE_{\bbC_j,0,\phi}[\exp(h \sum_{x \in \mathring{\bbC_j}} \gd_x)]\ge 1$ and since $\sum_{x \in \mathring{\bbC_j}} \gd_x$ is bounded by $3^d \kappa$ on $A_\kappa^{(3)}(j)$ we have
\begin{equation}
\label{eq:in1l}
\bE_{\bbC_j,h,\phi}\left[\gd_x; A_\kappa^{(3)}(j)\right] \, \le \, \exp\left( 3^d \kappa h\right)
\bE_{\bbC_j,\phi}\left[\gd_x \right] \, \le \, \frac 3 2
\bE_{\bbC_j,\phi}\left[\gd_x \right] \, \le \, 2ah\, ,
\end{equation}
where we have chosen $h$ sufficiently small. The last inequality is due to Lemma~\ref{th:Z} applied with 
$v(h,x):= H^{\phi}_{\bbC_j}(x)- u$ (since $\phi \in A^{(1)}_{\kappa}$, $v(h,x)$ satisfies the assumption of the Lemma uniformly in $x$). Therefore
\begin{equation}
\label{eq:T2B}
\hat \bE ^u [T_2; \, B_{u}]\, \le \, 4a^2 c(\gb,\kappa,d) h^2\,.
\end{equation}

\medskip

\subsubsection{Step 7: conclusion}
By putting \eqref{eq:T1B} and  \eqref{eq:T2B} together
and recalling Lemma~\ref{th:replica} we find that $((a/2)-4a^2 c(\kappa,d)\gb^2) h^2$
is a lower bound for the expression in the left-hand side of  \eqref{eq:fs5} and 
once $a$ is chosen smaller $1/(8 c(\gb,\kappa,d))$ the proof of Proposition~\ref{th:fs} is complete, the case $\gb \in [0,1]$ 
follow from \eqref{beta1}.
\qed

\subsection{Proof of Lemmas \ref{th:Best} and \ref{th:Aest}}
\label{sec:lem-proofs}

We start off with an elementary (rough) estimate on the variance of the harmonic extension.
Throughout this section we use the short-cut notation for $x\in \gL$ and $y\in \partial \gL$,
\begin{equation}
\label{eq:RWrepr-sh}
p_\gL(x,y)\, :=\, P^x\left( X_{\tau_{\partial \gL}}=y\right) \, ,
\end{equation}
with which  \eqref{poisson} becomes
\begin{equation}
H^\phi_\gL (x)= \sum_{y\in \partial \gL} p_\gL(x,y) \phi_y\, .
\end{equation}

\medskip

\begin{lemma}
\label{th:forH}
Let $N$ and $M$ be two integer numbers such that $N>2M>0$. As usual $\gL_N=\{0, 1, \ldots, N\}^d$
and we introduce $\gL_{N,M}:= \{ M, \ldots, N-M\}^d$. Let $\{\phi_x\}_{x \in \gL_N}$, with law $\bP_N$,
 be a centered Gaussian field with covariance $G_{\gL_N}(x,y) \le G(x,y)$ for every $x,y\in \partial \gL_N$
 (we recall that $\hat \bP^0$ is the law of the infinite volume Gaussian lattice free field). Then there exists $C_d>0$ (depending only on the dimension) such  that for every  $M$ we have
\begin{equation}\label{maxvarianz}
\sup_{N: \, N>2M}
\max_{x\in \gL_{N,M} }\mathrm{var}_{\bP_N}\left( H^\phi_{\gL_N} (x) \right) \, \le \, C_dM^{-\frac{d-2}2}\, .
\end{equation} 
 \end{lemma}
 \medskip

 \begin{rem}
  More advanced computations could show that the l.h.s of \eqref{maxvarianz} is truly of order 
  $M^{2-d}$, but the bound presented above is much easier to obtain and sufficient for our purposes.
 \end{rem}

 \medskip
 
 \noindent
 {\it Proof.}
 By recalling \eqref{eq:Green-bound}, we observe that
 $G_N(x,y)$
 is bounded above by $\mathsf{c}_d/(1+\vert x- y\vert)^{d-2}$ for every $x$ and $y$. Hence, we obtain
 that 
 for $x \in \gL_{N,M}$
 \begin{multline}
 \mathrm{var}_{\bP_N}\left( H^\phi_{\gL_N} (x) \right) \,=\, \sum_{y, y' \in \partial \gL_{N}}
 p_{\gL_N} (x,y) p_{\gL_N}(x,y')G_{\gL_N}(y,y')\\
 \le \, \frac{\mathsf{c}_d}{M^{\frac{d-2}2}}\sumtwo{y, y' \in \partial \gL_{N}: }{\vert y -y' \vert \ge  M^{\frac 12}} 
  p_{\gL_N} (x,y) p_{\gL_N}(x,y') + \mathsf{c}_d \sumtwo{y, y' \in \partial \gL_{N}: }{\vert y -y' \vert <  M^{\frac 12}} p_{\gL_N} (x,y) p_{\gL_N}(x,y')
  \\
  \le \, \frac{\mathsf{c}_d}{M^{\frac{d-2}2}} + \mathsf{c}_d \left(\maxtwo{y\in \partial \gL_{N}}{x\in \gL_{N,M}}  p_{\gL_N} (x,y)\right)
  \sum_{y \in \partial \gL_{N} } p_{\gL_N} (x,y) \left \vert\left \{ y' \in \partial\gL_N: \, \vert y'-y \vert <M^{\frac 12}\right \} \right\vert
  \\
  \le \, \frac{\mathsf{c}_d}{M^{\frac{d-2}2}} + d \, 2^{d-1} \mathsf{c}_d\, M^{\frac{d-1}2} \left(\maxtwo{y\in \partial\gL_{N}}{x\in \gL_{N,M}}  p_{\gL_N} (x,y)\right) \, .
  \end{multline}
We are now going to bound the term in the parentheses in the last line by $M^{-d+1}$ times a constant that depends only on the dimension: and once this is done the proof of Lemma~\ref{th:forH} is complete. This can be achieved by using the 
explicit expression for $p_{\gL_N} (x,y)$ -- the exit probability from a cube -- that one finds in \cite[Prop.~8.1.3]{cf:LL}, but this expression is rather involved and 
we prefer to perform some steps to bound $p_{\gL_N} (x,y)$ with an exit probability from a half-space. 
For this we  observe that without loss of generality
we can assume that $y$ belongs to the hyperplane $\bbH_d:=\{z \in \bbZ^d:\, z_d=0\}$ 
and, by elementary considerations, 
$$p_{\gL_N}(x,y)\le p_{\bbH^+_d} (x,y),$$
where 
$$ \bbH^+_d:=\{z :\, z_d\ge 0\}. $$
In order to simplify further the expression let us point out that 
 we are  left with estimating  $\max_y p_{\bbH_d^+} (x,y)$ 
and we can therefore simply consider the case of $x=(0, \ldots, 0,L)$, with $L\ge M$.
We have
\begin{equation}
\begin{split}
p_{\bbH^+_d} \big((0, \ldots, 0,L),(y_1, \ldots, y_{d-1},0)\big)\, &=\, 
p_{\bbH^+_d} \big((-y_1, \ldots, -y_{d-1},L),(0, \ldots ,0)\big) 
\\
&= \frac{2L}{\Sigma_d \vert z \vert ^{d}} \left( 1+O\left(\frac L{\vert z \vert ^2}\right) \right) 
+ O\left(\frac 1{\vert z \vert ^{d+1}}\right)\, ,
\end{split}
\end{equation}
where  $z=(-y_1, \ldots, -y_{d-1},L)$
and   in the second step we have used \cite[Th.~8.1.2]{cf:LL} ($\Sigma_d$ is the measure of the $d-1$ surface 
of the unit ball in $\bbR^d$). This last estimate suffices to conclude the proof of Lemma~\ref{th:forH}.
 \qed

\medskip

\noindent
{\it Proof of Lemma~\ref{th:Best}.}
As remarked right after the statement, we can assume $u=0$.
We use 
Lemma~\ref{th:forH} with $N=N_1$ and $M=N_0/2$. Therefore (using $d\ge 3$)
\begin{equation}
\maxtwo{x \in \gL_{N_1}: }{\dd \left(x, \gL_{N_1}^\complement\right) \ge N_0/2 }
\mathrm{var}_{\hat \bP^0}\left(
H^{\hat \phi}_{\gL_{N_1}} (x)
\right)\, \le \, C_d 
N_0^{-\frac{d-2}2}
\, \le \, C_d \gr_h^{\frac 12}
\, .
\end{equation}
Thus, by exponential Chebychev bounds,  for the $x$'s we are dealing with we have 
\begin{equation}
\hat \bP^0 \left( \left \vert H^{\hat \phi}_{\gL_{N_1}} (x)\right \vert \, > \, \frac{\gr_h^{\frac 18}}2\right)\, \le \, 
2 \, \exp\left(
- \frac{\gr_h^{\frac 14 -\frac 12}}{8 C_d}
\right)\, =\, 
2 \, \exp\left(
- \frac{\gr_h^{-\frac 1{4}}}{8 C_d}
\right)
\, .
\end{equation}
By  performing a union bound we see
that 
\begin{equation}
\hat \bP ^u \left( B_{u}^\complement\right) \, \le \,
2 h^{-3d}\, \exp\left(
- \frac{\gr_h^{-\frac 1{4}}}{8 C_d}
\right) \,  ,
\end{equation}
and the proof is complete.
\qed

\bigskip

In the proof of Lemma~\ref{th:Aest} we make use of the following estimate whose proof is postponed.

\begin{lemma}
\label{th:green3}
For $d \ge 3$ and for every $\kappa=1,2, \ldots$ we have
\begin{equation}
\label{eq:green3}
\sup_{B\subset \bbZ^d:\, \vert B \vert = \kappa}
\sum_{(x,y)\in B^2} G(x,y) \, \le \, c(d) \kappa^{1+\frac 2d}\, ,
\end{equation}
where $c(d)= 2^{d+4}\mathsf{c}_d$ ($\mathsf{c}_d$ is the constant appearing in \eqref{eq:Green-bound}).
\end{lemma}

\noindent
{\it Proof of Lemma~\ref{th:Aest}.}
Recall \eqref{eq:A1} and \eqref{eq:A2}. We use
\begin{equation}
\label{eq:sep12}
\bP_{N_1}^{\hat \phi} \left( A_{ \kappa}^\complement \right)\, =\, 
\bP_{N_1}^{\hat \phi} \left({A^{(1)}}^\complement \right)+ 
\bP_{N_1}^{\hat \phi} \left({A_{\kappa}^{(2)}}^\complement \cap {A^{(1)}} \right)\, ,
\end{equation}
and we estimate the two terms on the right-hand side, uniformly over $\hat \phi \in B_{u}$.

\medskip

For the first term we start by observing that since $\hat \phi \in B_{u}$
\begin{equation}
\bP_{N_1}^{\hat \phi} \left( \left \vert H_{\bbC_j}^\phi(x)-u \right\vert \, >\, \gr_h^{\frac 18} \right)\, \le \, 
\bP_{N_1}^{\hat \phi} \left( \left \vert H_{\bbC_j}^\phi(x)-  \bE_{N_1}^{\hat \phi}\left[ H_{\bbC_j}^\phi(x)\right]
 \right\vert \, >\, \frac{\gr_h^{\frac 18} }2 \right)\, ,
\end{equation}
and 
we apply Lemma~\ref{th:forH} with $N= 2N_0$ and $M=N_0/2$  to obtain 
that for every $x \in \bbC_j$ 
\begin{equation}
\mathrm{var}_{\bP_{N_1}^{\hat \phi} }\left(  H_{\bbC_j}^\phi(x)\right) \, \le \, 
C_d \left(\frac{N_0}2\right)^{-\frac{d-2}2} \, \le \, 2^{\frac{d-2}2} C_d \gr_h^{\frac12}  \, =:\, C_d' \gr_h^{\frac12}
\, ,
\end{equation}
and therefore for such $x$'s
\begin{equation}
\bP_{N_1}^{\hat \phi} \left( \left \vert H_{\bbC_j}^\phi(x)-u \right\vert \, >\, \gr_h^{\frac 18} \right)\, \le \, 
2 \, \exp\left(
- \frac{\gr_h^{-\frac 1{4}}}{8 C'_d}
\right)\, .
\end{equation}
Therefore, by a union bound, we obtain
\begin{equation}
\sup_{\hat \phi \in B_{u}}
\bP_{N_1}^{\hat \phi} \left({A^{(1)}}^\complement \right)\, \le \, N_1^d
2 \, \exp\left(
- \frac{\gr_h^{-\frac 1{4}}}{8 C'_d}
\right)\, ,
\end{equation}
where  $N_1^d=h^{-3d}$ and therefore 
this term is $O(h^\alpha)$ for any $\alpha>0$. So, in view of \eqref{eq:Aest}, we can safely focus on the second
term in the right-hand side of \eqref{eq:sep12}.

\medskip 

For such a term we first observe that
\begin{equation}
\begin{split}
\bP_{N_1}^{\hat \phi} \left({A_{\kappa}^{(2)}}^\complement \cap {A^{(1)}} \right)
\, &= \, 
\bP_{N_1}^{\hat \phi} \left(
\cup_{j \in \cJ}\left \{
 \sum_{x \in \bbB_j} \gd_x \, >\,  \kappa
\right\}
\cap {A^{(1)}} \right)
\\
&\le \, 
\sum_{j \in \cJ }
\bP_{N_1}^{\hat \phi} \left(
\left \{
 \sum_{x \in \bbB_j} \gd_x \, >\,  \kappa
\right\}
\cap
 \left\{ \max_{x \in \bbB_j} \left \vert H^\phi_{\bbC_j} (x) -u \right \vert \, 
\le  \, \gr_h^{\frac 18}\right\} \right)\, .
\end{split}
\end{equation}
We now 
condition on the $\gs$-algebra generated by $(\phi_x)_{x \in \cup_{j \in \cJ_w}\partial \bbC_j}$
by the Markov property we obtain that
\begin{equation}
\label{eq:union-kappa}
\bP_{N_1}^{\hat \phi} \left({A_{\kappa}^{(2)}}^\complement \cap {A^{(1)}} \right)
\, \le  \, \sum_{j \in \cJ } {\sup}' \, 
\bP_{\bbC_j}^{\hat \phi}\left( 
\sum_{x \in \bbB_j} \gd_x \, >\,  \kappa
\right)\, ,
\end{equation} 
where ${\sup}'$ stands for the supremum over $\hat\phi$ such that 
$\max_{x \in \bbB_j}  \vert H^{\hat \phi}_{\bbC_j} (x) -u \vert  
\le   \gr_h^{\frac 18}$. To bound the $\sup'$ term, we proceed as follows:
\begin{equation}
\label{eq:withsum}
\bP_{\bbC_j}^{\hat\phi}\left( 
\sum_{x \in \bbB_j} \gd_x \, >\,  \kappa
\right)\, \le  \, \sumtwo{x_1, \ldots , x_\kappa \in \bbB_j: }{x_i \neq x_{i'} \text{ for } i \neq {i'}}
\bP_{\bbC_j}^{\hat\phi}\left( 
\frac 1{\kappa}\sum_{i=1}^\kappa \vert \phi_{x_i} \vert  \, \le 1
\right)\, ,
\end{equation}
and 
\begin{multline}
\label{eq:withoutsum}
 \bP_{\bbC_j}^{\hat\phi}\left( 
\frac 1{\kappa}\sum_{i=1}^\kappa \vert \phi_{x_i} \vert  \, \le 1
\right) \; \le \, 
 \bP_{\bbC_j}^{\hat\phi}\left( 
\frac 1{\kappa}\sum_{i=1}^\kappa  \phi_{x_i}   \, \le 1
\right)
\, 
 \le 
\,   \bP_{\bbC_j}^{0}\left( 
\frac 1{\kappa}\sum_{i=1}^\kappa  \phi_{x_i}   \, \le \,  1 - 
\frac 1{\kappa}\sum_{i=1}^\kappa H^{\hat \phi}_{\bbC_j} (x)
\right)
\\
\le \, 
 \bP_{\bbC_j}^{0}\left( 
\frac 1{\kappa}\sum_{i=1}^\kappa  \phi_{x_i}   \, \le \,  1 - u+ \gr_h^{\frac 18}
\right) \,
\le \, 
 \bP_{\bbC_j}^{0}\left( 
\frac 1{\kappa}\sum_{i=1}^\kappa  \phi_{x_i}   \, \ge \, u-2
\right),
\end{multline}
where in the last step we used the symmetry and the choice of subtracting $2$ is arbitrary (any number larger than $1$ would do, and we have to choose $h$ sufficiently small).
It is now a matter of evaluating the variance of 
$\frac 1{\kappa}\sum_{i=1}^\kappa  \phi_{x_i}$ and we perform this  estimate uniformly over the location of $x_1, \ldots, x_\kappa$. In fact, we apply 
Lemma~\ref{th:green3} to obtain that 
\begin{equation}
\mathrm{var}_{\bP_{\bbC_j}^{0}}\left( 
\frac 1{\kappa}\sum_{i=1}^\kappa  \phi_{x_i}  
\right) \, \le \, c(d) \kappa^{-1 + \frac 2 d}\, \le \, c(d) \kappa^{-1/3}\,,
\end{equation}
and therefore \eqref{eq:asymptZ} yields
\begin{equation}
 \bP_{\bbC_j}^{0}\left( 
\frac 1{\kappa}\sum_{i=1}^\kappa  \phi_{x_i}   \, \ge \, u-2
\right)\, = \, \bP\left( \cN \ge \kappa^{1/6}c(d)^{-1/2} (u-2) \right)\, \le \, h^{\frac{\kappa^{1/3}}{2c(d)}}\,,
\end{equation}
where $\cN$ is a standard Gaussian random variable.
With this estimate we now go
 back to
\eqref{eq:withsum}
and we obtain
\begin{equation}
\bP_{\bbC_j}^{\hat\phi}\left( 
\sum_{x \in \bbB_j} \gd_x \, >\,  \kappa
\right)\, \le  \, N_0^\kappa h^{\frac{\kappa^{1/3}}{2c(d)}} \, \le  \, h^{\frac{\kappa^{1/3}}{3c(d)}}\, ,
\end{equation}
for $h$ sufficiently small. In turn, this estimate yields the control (see 
\eqref{eq:union-kappa})
\begin{equation}
\label{eq:union-kappa2}
\bP_{N_1}^{\hat \phi} \left({A_{\kappa}^{(2)}}^\complement \cap A^{(1)} \right)
\, \le  \, h^{-3d} h^{\frac{\kappa^{1/3}}{3c(d)}}\, \le \, h^{\kappa^{1/4}}\, ,
\end{equation}
for $\kappa$ such that $9d\, c(d)\le \kappa^{\frac 13} - 3c(d)\kappa^{\frac 14}$.
Let us choose $\kappa^{\frac 13} \ge 6c(d) \, \kappa^{\frac 14}$, that is $\kappa \ge (6c(d))^{12}$,
and, under this assumption, $9d\, c(d)\le \kappa^{\frac 13} - 3c(d)\kappa^{\frac 14}$ is satisfied if $\kappa^{\frac 13} \ge 18 d\, c(d)$. 
Taking into account these two lower bounds on $\kappa$ we see
that is suffices to choose  $\kappa\ge d^3 6^{12} c(d)^{12}$. 
The proof of Proposition~\ref{th:fs} is therefore complete.
\qed



\medskip

\noindent
{\it Proof of Lemma \ref{th:green3}.}
First of all we introduce $d_\square (x,y):= \max_{i=1, \ldots, d} \vert x_i- y_i\vert \le \vert x-y\vert$ and we start from the
direct consequence of \eqref{eq:Green-bound}:
for every $x,y \in \bbZ^d$
\begin{equation}
\label{eq:C_G}
G(x,y) \, \le \, \frac{\mathsf{c}_d}{\left( 1+ d_\square(x,y) \right)^{d-2}}\, =: g_\square \left( d_\square(x,y) \right)\, .
\end{equation}

Thanks to \eqref{eq:C_G}, it suffices to prove the statement for 
$\sum_{(x,y)\in B^2} g_\square(x-y) $. We then observe that 
for any $B\subset \bbZ^d$ with $\vert B \vert =j$ and every $z\notin B$ we have  
\begin{equation}
\sum_{(x,y)\in (B\cup \{z\})^2} g_\square(x-y)-
\sum_{(x,y)\in B^2} g_\square(x-y)\, \le \, 2 \sum_{i=1}^j g_\square(x_i) + g_\square(0)\, ,
\end{equation}
where $x_1,x_2, \ldots$  yields a {\sl fully} packed configuration of points around the origin.
By this we mean that if $B(n)= \{-n, \ldots, n\}^d$ and $A(n)= B(n) \setminus B(n-1)$, then 
$x_1, \ldots, x_{\vert B(1)\vert}$ is an arbitrary numbering of the points in $A(1)$,   
$x_{\vert B(1)\vert+1}, \ldots, x_{\vert B(2)\vert}$ is an arbitrary numbering of the points in $A(2)$, and so on.
Of course $z$ disappears because it has been translated to the origin. We have 
\begin{multline}
2 \sum_{i=1}^j g_\square(x_i)\, \le \, 2  c_d \sumtwo{m=1, 2 , \ldots}{m\le \frac{j^{1/d}}2 + \frac 12} \frac{\vert A(m)\vert}
{(1+m)^{d-2}}\, \le \, 2^{d+1}  c_d \sumtwo{m=1, 2 , \ldots}{m\le j^{1/d}} (m+1)\, \le \\ 2^{d}  c_d
\left(  j^{1/d} +2 \right)^2 \, \le \, 2^{d+1}  c_d
\left(  j^{2/d} +4 \right) \, \le \, 5 \cdot 2^{d+1}  c_d j^{2/d} \, ,
\end{multline}
where in the first step we have simply made $g_\square(\cdot)$ explicit, used the fact that is constant on annuli and 
 that with $j$ points we cannot go beyond filling $\frac{j^{1/d}}2 + \frac 12 (\le j^{1/d})$ annuli. In the second step instead we 
used $A(n) \le 2d (2m+1)^{d-1} \le d 2^d (m+1)^{d-1}$.
Therefore for any $B$ with $\vert B \vert =j$ and every $z\notin B$ we have  
\begin{equation}
\sum_{(x,y)\in (B\cup \{z\})^2} g_\square(x-y)-
\sum_{(x,y)\in B^2} g_\square(x-y)\, \le \, 5 \cdot 2^{d+1}  c_d j^{2/d}+  c_d\, ,
\end{equation} 
so that for any $B$ with $\vert B \vert =\kappa$
\begin{multline}
\sum_{(x,y)\in B^2} g_\square(x-y)\, \le \, \kappa g_\square (0)+
5 \cdot 2^{d+1}  c_d \sum_{j=1}^{\kappa-1} j^{2/d}\\ \le \, \kappa  c_d+5 \cdot 2^{d+1}  c_d \kappa^{1+ \frac 2d}\, 
\le \, 2^{d+4}  c_d \kappa^{1+ \frac 2d}\, ,
\end{multline} 
and the proof is complete.
\qed

\section{The two dimensional case}\label{twodim}

This section is dedicated to the proof of Theorem \ref{th:d=2}.
The first step of the proof in Section \ref{finitevoll} is to establish a finite volume criterion similar to \eqref{superad} .
Then in Sections \ref{sec:rough} and \ref{sec:lessrough} we will use replica coupling arguments in the spirit of Lemma \ref{th:replica} to have a bound on the free energy of a 
system with finite volume.

\subsection{A finite volume criterion and replica coupling in dimension two }\label{finitevoll}

We want to have a criterion similar to \eqref{superad} in two dimension but the problem we have to face is that
there is no infinite volume limit for the two dimensional lattice free field.
A way out is to consider the massive free field \eqref{massivedensity} and then find a way to compare the free energy with the original one.

\medskip

We let $\bP^{u,m}$ be the law the infinite volume limit of the massive free field with mean $u$ and mass $m$ (we write $\hat\bP^{u,m}$ when the 
variable is denoted by $\hat \phi$).
We define also the measure in finite volume with boundary conditions:
\begin{equation}\label{densitym}
\bP_N^{\hat \phi,u,m}(\dd \phi)=\frac{1}{\mathcal Z_N^{\hat \phi,u,m}} \left( \prodtwo{x,y\in \gL_N}{x\sim y} \exp\left(-\frac{ (\phi_x-\phi_y)^2 }{4} \right)\right) \left(\prod_{x\in \gL_N\setminus \partial \gL_N} e^{-\frac{m^2}{2}(\phi_x-u)^2} \dd \phi_x \right).
\end{equation} 
where 
\begin{equation}
\mathcal Z^{\hat \phi}_N:= \int_{\bbR^{ \gL_N\setminus \partial \gL_N}} \left( \prodtwo{x,y\in \gL_N}{x\sim y} \exp\left(-\frac{ (\phi_x-\phi_y)^2 }{4} \right)\right) \left(\prod_{x\gL_N\setminus \partial \gL_N}  e^{-\frac{m^2}{2}(\phi_x-u)^2} \dd \phi_x \right).
\end{equation}

The particular case where $\hat \phi \equiv u$ is denoted by $\bP_N^{u,m}$.
We set
\begin{equation}
Z^{\gb,\go,\hat \phi, u, m}_{N,h}:=\bE^{\hat \phi, u, m}_N\left[ \exp\left( \sum_{x\in  \tilde \gL_N} (\gb \go_x-\gl(\gb)+h)\delta_x\right)\right],
\end{equation}
and denote by $Z^{\gb,\go,u,m}_{N,h}$ the partition function corresponding to constant boundary conditions $u$.
Similarly to Proposition \ref{superadd}, we can prove:

\medskip

\begin{proposition}\label{lenergie}
For any $m$ and $u$ 
\begin{equation}
\lim_{N\to \infty} \frac{1}{N^2}\bbE \left[\log Z^{\gb,\go, u, m}_{N,h}\right]=
\lim_{N\to \infty} \frac{1}{N^2}\bbE \hat \bE^{u,m}\left[\log Z^{\gb,\go,\hat \phi, u, m}_{N,h}\right]=\tf(\gb,h,m,u).
\end{equation}
and furthermore for any value of $N$
\begin{equation}
\frac{1}{N^d}\bbE \hat \bE^{u,m}\left[\log Z^{\gb,\go,\hat \phi, u, m}_{N,h}\right]\le \tf(\gb,h,m,u).
\end{equation}
\end{proposition}
\medskip

Note that, unlike the massless ($m=0$) case, there is now a dependence on $u$ (and on $m$).
Now, for this criterion to be useful, we need to be able to compare 
$\tf(\gb,h,m,u)$ with $\tf(\gb,h)$.
The idea for this is the following: the derivative of the massive free field measure with respect to the non massive one has the following expression

\begin{equation}\label{densitic}
\frac{\dd \bP^{u,m}_{N}}{\dd \bP^{u}_N}=\frac{1}{W^m_N}\exp\left(\sum_{x\in \gL_N\setminus \partial \gL_N}-\frac{m^2}{2}\left(\phi_x-u\right)^2 \right).
\end{equation}
where 
\begin{equation}
W^m_N=  \bE^{u}_N\left[-\frac{m^2}{2}\sum_{x\in \gL_N\setminus \partial \gL_N}(\phi_x-u)^2\right] =\bE_N\left[\exp\left(-\frac{m^2}{2}\sum_{x\in \gL_N\setminus \partial \gL_N}\phi_x^2\right)\right]
\end{equation}

\begin{lemma}\label{radonnic}
We have 
\begin{equation}\label{massicost}
\lim_{N\to \infty} \frac{1}{N^2}\log W^m_N=
-\frac{1}{2}\int_{[0,1]^2} \log\left(1+\frac{m^2}{4[\sin^2(\pi x/2  )+\sin^2(\pi y/2 )]}\right)\dd x \dd y=-f(m).
\end{equation}
Around zero we have the following equivalence 
\begin{equation}
f(m)\stackrel{m\searrow 0}\sim c_W m^2\vert \log m\vert\, , \ \ \ \ \text{ with } c_W\, :=\, \frac{1}{4\pi}.
\end{equation}
Furthermore we have 
 \begin{equation}\label{comparons}
\tf(\gb,h,m,u)\le \tf(\gb,h)+f(m).
\end{equation}
\end{lemma}

\begin{proof}
Let us start to prove \eqref{comparons} is deduced from \eqref{massicost}.
Because the exponential in \eqref{densitic} is always smaller than $1$ we have 
\begin{equation}
 Z^{\gb,\go, u, m}_{N,h}\le \frac{1}{W^m_N} Z^{\gb,\go, u}_{N,h}.
\end{equation}
The result is obtained by taking $\log$, dividing by $N^d$ and passing to the limit.
The functions $u_i$, $i=1,\dots, N-1$ defined 
\begin{equation}
u_i(k):=\sqrt{\frac{2}{N}}\sin(ik\pi/N)
\end{equation}
forms an orthogonal base of eigenfunctions of the one dimensional Laplacian with Dirichlet boundary conditions on $[0,N]\cap \bbZ$.
We let Let 
$0>-\lambda_1>-\lambda_2>\dots>-\lambda_{(N-1)}$ denote the associated eigenvalues where 
\begin{equation}
\lambda_{i}=2(1-\cos(i\pi/N))\, .
\end{equation}
We set 
\begin{equation} \label{defvij} 
v_{i,j}(x_1,x_2):= u_i(x_1)u_j(x_2),
\end{equation}
and
\begin{equation}
\alpha_{i,j}(\phi):=\sum_{x \in \gL_N \setminus \partial \gL_N}v_{i,j}(x)\phi_x.
\end{equation}
From Parseval's formula we have
\begin{equation}
\sum_{x\in \gL_N\setminus \partial \gL_N}\phi_x^2=\sum_{i,j=1}^{N-1} \alpha_{i,j}(\phi)^2.
\end{equation}
Now note that when $\phi$ has law $\bP_N$, the $ \alpha_{i,j}(\phi)$ are independent Gaussian variables.
Their variance is equal to $(\gl_i+\gl_j)^{-1}$.
Hence 
\begin{equation}
 W^m_N:=\prod_{i,j=1}^{N-1} \bbE\left[ \exp\left(-\frac{m^2}{2(\gl_i+\gl_j)}\mathcal N ^2\right)\right]
 =\left( \prod_{i,j=1}^{N-1}\sqrt{1+\frac{m^2}{(\gl_i+\gl_j)}}\right)^{-1}. 
\end{equation}
where $\mathcal N$ is a standard Gaussian. It is then standard (it is a Riemann sum) to check that 
\begin{equation}
\frac{1}{N^2} \log W^m_N:=-\frac{1}{2N^2}\sum_{i,j=1}^{N-1}\log \left( 1+\frac{m^2}{(\gl_i+\gl_j)}\right),
\end{equation}
converges to 
\begin{equation}
-\frac{1}{2}\int_{[0,1]^2} \log\left(1+\frac{m^2}{4[\sin^2(x\pi /2)+\sin^2(y\pi /2)]}\right)\dd x \dd y\, .
\end{equation}
For the leading order asymptotic behavior one can restrict the domain of integration to positive $x$ and $y$ such that
$x^2 +y^2\le \gep^2$, $\gep$ arbitrarily small. Passing to polar coordinates the estimate becomes rather straightforward.  
\end{proof}
\medskip

As a consequence of Lemma \ref{radonnic} and Proposition \ref{lenergie} we obtain the following finite volume criterion:
\medskip

\begin{cor}\label{massic}
For every $m$, $N$ and $u$ one has 
\begin{equation}
\label{eq:massic}
\tf(\gb,h) \ge \frac{1}{N^2}\bbE \hat \bE^{u,m}\left[\log Z^{\gb,\go,\hat \phi, u, m}_{N,h}\right]-f(m).
\end{equation}
\end{cor}

\medskip

To estimate the quantity $\bbE \hat \bE^{u,m}\left[\log Z^{\gb,\go,\hat \phi, u, m}_{N,h}\right]$
we will use (as for Lemma~\ref{th:replica}) the following bound derived from replica-coupling. 
For the proof see Appendix~\ref{sec:replica} where the slightly more involved proof of Lemma~\ref{th:replica}
is given in  detail. 

\medskip

\begin{lemma}\label{th:replica2}
We have for all $u\in \bbR$ for all $m>0$ and all boundary conditions $\hat \phi$
\begin{equation}
\log Z^{\gb,\go,\hat \phi, u, m}_{N,h}\ge
 \log \bE_{N}^{\hat \phi,u,m} 
\left[ e^{h \sum_{x \in  \tilde\gL_{N}}  \gd_x}
\right] 
-\log \left \langle \exp \left( 2 \gb^2 \sum_{x \in  \tilde\gL_{N}} \gd^{(1)}_x
\gd^{(2)}_x \right) \right\rangle^{\otimes 2}\, ,
\end{equation}
where 
\begin{equation}
\label{eq:rnder}
\left \langle \, \cdot \, 
\right \rangle=\left \langle \, \cdot \, 
\right \rangle_{N,h, \hat \phi,u,m}\, :=\, 
\frac{ 
\bE_{N}^{\hat \phi,u,m} 
\left[ \, \cdot \,  \exp\left(h \sum_{x \in  \tilde \gL_{N} } \gd_x\right)
\right] 
}
{\bE_{N}^{\hat \phi,u,m} 
\left[ \exp\left(h \sum_{x \in  \tilde \gL_{N}}  \gd_x\right)
\right] 
}\, .
\end{equation}
\end{lemma}

\subsection{A  first, rough bound on the critical point (warm up argument)}
\label{sec:rough}

In order to use \eqref{eq:massic} in the most efficient way, for a given $\beta$ we must tune the value of $m$ and $u$ and $N$ in 
to obtain the best possible bound.

\medskip

We start by stating and proving a weaker version of  Theorem~\ref{th:d=2} obtained by choosing $u=0$.
The proof uses some of the steps that will be used for  Theorem~\ref{th:d=2},
but not all and it is considerably simpler. So it can be viewed as a warm up.

\begin{proposition}
\label{th:d=2-0}
 When $d=2$ and $\go$ is Gaussian, for every $\gb_0>0$ there exists a constant $c>0$ such that  for all $\gb<\gb_0$
 \begin{equation}
 h_c(\gb)\, \le\,   \frac{c\gb^{2}}{\sqrt{\vert \log \gb \vert }}\, .
 \end{equation} 
\end{proposition}

\medskip

\noindent
{\it Proof.} It suffices to prove the result for $\gb$ small. So let us choose
 $\gb>0$ and let us set $N=1/ \gb $: with slight abuse we will assume $N \in \bbN$ without taking integer part. 
We introduce a mass $m$ and exploit Corollary~\ref{massic}, but with $u=0$. We aim at showing
that there exists $c>0$ such that if  $h>c \frac{\gb^{2}}{\sqrt{\vert \log \gb \vert }}$
there exists $m_0$ such that for $m< m_0$
\begin{equation}
 \frac{1}{N^2}\bbE \hat \bE^{0,m}\left[\log Z^{\gb,\go,\hat \phi, 0, m}_{N,h}\right]-f(m) \, >\, 0\, .
\end{equation}
When applying Lemma \ref{th:replica2}
 we have then to deal with two terms:

\medskip

\begin{itemize}
\item For the first one we obtain a lower bound just by computing and  using Jensen's inequality:
\begin{equation}
\begin{split}
\label{eq:reprep-1}
\frac{1}{N^2} \hat \bE^{0,m}\left[\log \bbE Z^{\gb,\go,\hat \phi, 0, m}_{N,h}\right] \, = \, 
\frac{1}{N^2} \hat \bE^{0,m} \log  Z^{\hat \phi, 0, m}_{N,h}\, &\ge \, 
\frac{1}{N^2} \hat \bE^{0,m}\bE^{\hat \phi, m,0}\left[ h \sum_{x\in \tilde \gL_N} \gd_x \right]
\\
&=\, h  \bE^{0,m} \left[ \gd_0 \right]\,.
\end{split}
\end{equation}
\item For the second term we need an upper bound and we observe that
\begin{multline}
\label{eq:reprep-2}
 \frac{1}{N^2}\hat \bE^{0,m} 
  \log
\left \langle
\exp\left(2\gb^2 \sum_{x\in  \tilde \gL_N} \delta^{(1)}_x \delta^{(2)}_x
\right) 
 \right \rangle ^{\otimes 2}_{N,0,m, h; \hat \phi} \\
  \le \, \frac{ \gb^2 (e^2-1)}{N^2} \hat \bE^{0,m} 
 \left \langle \sum_{x\in  \tilde \gL_N} \delta^{(1)}_x \delta^{(2)}_x
 \right \rangle ^{\otimes 2}_{N,0,m, h; \hat \phi}
 = \, \frac{ \gb^2 (e^2-1)}{N^2} \hat \bE^{0,m} 
  \sum_{x\in  \tilde \gL_N} \left \langle\delta_x 
 \right \rangle_{N,0,m, h; \hat \phi}^2
 \\
 \le \,
 \frac{7 \gb^2}{N^2} \hat \bE^{0,m} 
  \sum_{x\in  \tilde \gL_N} \left \langle\delta_x 
 \right \rangle_{N,0,m, 0; \hat \phi}^2\,, 
 \end{multline}
where in the first inequality we have used the fact that 
$\gb^2 \sum_{x\in  \tilde \gL_N} \delta^{(1)}_x \delta^{(2)}_x \le 1$ and Lemma \ref{lemmeduX}
and in the second we have used the fact that $h=o(\gb^2)$ therefore the Radon-Nykodym density \eqref{eq:rnder}
can be made arbitrarily close to $1$ when $\gb$ becomes small.
\end{itemize}

\medskip
By separating the contribution of the boundary in the second term (i.e. the rightmost term in \eqref{eq:reprep-2}), we realize that it suffices to show that
\begin{equation}
\label{eq:tbs65}
 h  \bE^{0,m} \left[ \gd_0 \right] -  \frac{7 \gb^2}{N^2} \hat \bE^{0,m} 
  \sum_{x\in  \mathring{\gL}_N} \left \langle\delta_x 
 \right \rangle_{N,0,m, 0; \hat \phi}^2 -  \frac{7d \gb^2}{N} \bE^{0,m} \left[ \gd_0 \right]
 - 2 c_W m^2 \vert \log m \vert \, >\, 0\, ,
\end{equation}
where the boundary term is irrelevant because $\frac{7d \gb^2}{N}= O(\gb^3)$, hence it is dominated by $h$ for $\beta $ small.
We are playing on choosing $\gb$ small, hence $N$ large, but one should think that we have chosen $\gb$, possibly 
small, and then we choose $m$ as small as we wish or need. 

In particular
\begin{equation}
\left \langle\delta_x 
 \right \rangle_{N,0,m, 0; \hat \phi}\, \le \, \left \langle\delta_x  \right \rangle_{N,0,m, 0;0}\, =\, P\left(
 \vert\cN\vert \le \frac 1{\gamma_x}\right) \, \le \, \sqrt{\frac 2 \pi}  \frac 1{\gamma_x}\, , 
\end{equation}
where $\gamma_x^2$ is the variance of $\phi_x$ under the massive field with Dirichlet boundary conditions. And hence
\begin{equation}
 \hat \bE^{0,m} 
 \left \langle\delta_x  \right \rangle_{N,0,m, 0; \hat \phi}^2
 \, \le\,  \sqrt{\frac 2 \pi}  \frac 1{\gamma_x}
  \hat \bE^{0,m} 
 \left \langle\delta_x  \right \rangle_{N,0,m, 0; \hat \phi}\, =\, \sqrt{\frac 2 \pi}  \frac 1{\gamma_x}\bE^{0,m}[\delta_0]
\end{equation}

By choosing $m$ sufficiently small we can say that $\gamma_x^2$ is bounded below by  the variance in the 
$m=0$ case times a positive number smaller than one, that is by \eqref{eq:G2est}
\begin{equation}
\gamma_x^2 \, \ge \frac1{5\pi} \log d_N(x)\, ,
\end{equation}
at least if $d_N(x):=\mathrm{dist}(x, \partial \gL_{N})$ is larger than a fixed positive number $d_0$. It is at this stage more practical to 
lose track of the constants and choose a $c'>0$ such that
$
\gamma_x^2 \ge c'\log (d_N(x)+1)$
for every $x\in \mathring{\gL}_N$. 
Therefore 
 \begin{multline}
 \sqrt{\frac 2 \pi}  
 \frac{ 7\gb^2}{N^2}  \bE^{0,m} \left[ \gd_0 \right]
  \sum_{x\in \mathring{\gL}_N}  \frac 1{\gamma_x}\, \le \, 7
 \sqrt{\frac 2 {c' \pi}}  
 \left(\frac{1}{N^2} \sum_{x\in \mathring{\gL}_N}  \frac 1{\sqrt{\log (d_N(x)+1)}}\right) 
  \gb^2 \bE^{0,m} \left[ \gd_0 \right]
  \\
  \, \le \,c'' \frac{\gb^2}{\sqrt{\vert \log \gb \vert}}  \bE^{0,m} \left[ \gd_0 \right]\,,
 \end{multline}
 where $c''$ is a positive constant that one can easily express in terms of $c'$.
 
 Going back to \eqref{eq:tbs65} we therefore see that it suffices to prove that
\begin{equation}
\label{eq:tbs65-1}
\left(c-c''- 7d \gb \sqrt{\vert \log \gb \vert}\right) \frac{\gb^2}{\sqrt{\vert \log \gb \vert}}  \bE^{0,m} \left[ \gd_0 \right]
  - 2 c_W m^2 \vert \log m \vert \, >\, 0\, ,
\end{equation}
and it is clearly necessary to choose $c>c''$, which we do, hence $c-c''- 7d \gb \sqrt{\vert \log \gb \vert}>0$
for $\gb $ suitably small. It is now clear that if 
$ \bE^{0,m} \left[ \gd_0 \right]\gg m^2 \vert \log m \vert$ we are done.
But 
\begin{equation}
 \bE^{0,m} \left[ \gd_0 \right] \, =\, P \left( \vert \cN \vert \le \frac 1{\gs_m}\right)\, 
 \end{equation}
 where
 $\sigma^2_m$ is the variance of the infinite volume massive field, which   satisfies 
\begin{equation}
\label{eq:gsm2}
\sigma^2_m \, =\, \int_0^\infty e^{-m^2 t} P^0\left(X(t)=0\right)\dd t 
\stackrel{m\searrow 0}\sim \frac{1}{2\pi}| \log m |\,,
\end{equation} 
where the asymptotic equivalence is a direct consequence of the Local Central Limit Theorem
\cite[Th.~2.5.6]{cf:LL}: $P^0 (X(t)=0)\sim (4\pi t)^{-1}$ as $t$ tends to infinity (recall the speed factor $4$ with respect to 
\cite{cf:LL}).  
We therefore see  $ \bE^{0,m} \left[ \gd_0 \right] $ is bounded below by $1/\sqrt{| \log m |}$ 
(times a positive constant) and the proof is complete.
\qed

\subsection{Proof of Theorem~\ref{th:d=2}}
\label{sec:lessrough}
For the proof we choose an arbitrary (small)
 $\gb>0$ and $N=1/ \gb $: again, with slight abuse we will assume $N \in \bbN$ without taking integer part. 
   We will then make estimates by introducing a massive field: the mass $m$ will be taken to go to zero and the height of the field $u$ will be a function of $m$, see \eqref{eq:u_m} below,
that tends to infinity as $m\searrow 0$. Our estimates correspond to taking $m\searrow 0$ first and then $\gb \searrow 0$.
Let us focus first on choosing the mean height $u=u_m$ of the massive field that we intend to exploit. 

\subsubsection{Setting the parameters of the massive field (mass and height)}
Let us start by recalling the behavior of the 
the variance $\gs^2= \sigma^2_m$ of the infinite volume massive field   \eqref{eq:gsm2}.
We now assume that $u_m$ is such that $\lim_{m\to 0} u_m/\gs^2_m =\sqrt{8\pi}=: C$.
A precise choice of $u_m$ is made in \eqref{eq:u_m} below but for now we  need neither
this expression nor the precise value of the positive constant $C$. To make formulas lighter
we write $\gs=\gs_m$ and $u=u_m$. We then compute
\begin{equation}
\label{eq:Egd0.1}
\begin{split}
\bE ^{u,m}[\gd_0] \, 
&=\, 
\int_{-2}^0 \frac 1{\sqrt{2\pi \gs^2}} \exp\left( -\frac{( x-u+1)^2}{2 \gs^2} \right)  \dd x
\\
&=\, 
\frac 1{\sqrt{2\pi \gs^2}} \exp\left( -\frac{( u-1)^2}{2 \gs^2} \right) \int_{-2}^0 \exp
\left( 
\frac{x(u-1)}{\gs^2} -\frac{x^2}{2\gs^2}
\right)\dd x \, ,
\end{split}
\end{equation}
and since $\lim_m (u-1)/ \gs^2 =C$ we readily see that the integral in the last line
converges to $\int_{-2}^0 \exp(xC) \dd x=(1-e^{-2C})/C$. On the other hand
we  see that 
\begin{equation}
\exp\left( -\frac{( u-1)^2}{2 \gs^2} \right)\stackrel{m\searrow 0}\sim \exp\left( -\frac{u^2}{2 \gs^2} \right)
\exp(C)\, ,
\end{equation}
 so that we obtain
 \begin{equation}
\label{eq:Egd0.2}
\bE ^{u,m}[\gd_0] \stackrel{m\searrow 0}\sim \exp\left( -\frac{u^2}{2 \gs^2} \right)
\frac{\sinh(C)}{\sqrt{\frac \pi 2}C} \frac1 \gs \,\sim\, 
 \exp\left( -\frac{u^2}{2 \gs^2} \right)
\frac{\sinh(C)}{C} \frac {2} {\sqrt{\vert \log m\vert}}\,.
\end{equation}
We now choose  $u=u_m$ such that
\begin{equation}
\label{eq:trf45}
\frac 1 {\sqrt{\vert \log m\vert}}
\exp\left( -\frac{u^2}{2 \gs^2} \right)
  \, =\,C' m^2 \vert \log m \vert ^2\, , \ \ \ \ \text{ for }
 C'= \frac{C}{2 \sinh(C)}\, .
\end{equation} 
Let us point out that, because of \eqref{eq:gsm2}, the choices
\eqref{eq:trf45} and $u/\gs^2 \sim C$ force $C=\sqrt{8\pi}$.

\medskip

\begin{rem}
The choice \eqref{eq:trf45} is linked to the fact that in the replica argument this term correspond
to the energy gain and it needs to beat the loss due to the presence of the term $f(m)\sim c_W m^2 \vert \log m \vert$
in the free energy lower bound \eqref{eq:massic}.
\end{rem}
\medskip

Here is a slightly more explicit expression for $u$
\begin{equation}
\label{eq:u_m}
\frac {u^2}{2\gs^2}\, =\, 2 \vert \log m \vert -\frac 52 \log \vert \log m \vert - \log C'\, .
\end{equation}
For the ease of the reader we collect 
 here the asymptotic behaviors ($m\searrow 0$)
\begin{equation}
\label{eq:asympt}
u \, \sim \, \frac{\sqrt{2}} {\sqrt{\pi}} \vert \log m \vert \, , \ \ \ \ \ \ \ \ \
\gs \sim \frac 1{\sqrt{2\pi}} \sqrt{\vert \log m \vert }
\, , \ \ \ \ \ \ \ \ \
 \frac{u}{\gs^2} \, \sim \,  \sqrt{8\pi}\, .
\end{equation} 

\subsubsection{Replica coupling estimates}
Recall that $N=1/\gb$. We aim at showing that if $h= \gb^b$, any $b<3$, there exists
a $\gb_0$ such that for $\gb< \gb_0$ there exists $m_0$ such that for $m< m_0$
\begin{equation}
 \frac{1}{N^2}\bbE \hat \bE^{u,m}\left[\log Z^{\gb,\go,\hat \phi, u, m}_{N,h}\right]-f(m) \, >\, 0\, .
\end{equation}
Using Lemma \ref{th:replica2} we see that it suffices to show that 
\begin{equation}
\label{eq:reprep}
\frac{1}{N^2} \hat \bE^{u,m}\left[\log \bbE Z^{\gb,\go,\hat \phi, u, m}_{N,h}\right]-  \frac{1}{N^2}\hat \bE^{u,m} 
  \log
\left \langle
\exp\left(2\gb^2 \sum_{x\in  \tilde \gL_N} \delta^{(1)}_x \delta^{(2)}_x
\right) 
 \right \rangle ^{\otimes 2} -f(m) > 0\,. 
\end{equation}
From  the choice of $u$ (cf. \eqref{eq:Egd0.2} and \eqref{eq:trf45}), 
once the choice of $\gb$ (hence of $N$ and $h$) is made,
for $m$ 
sufficiently small we have 
\begin{equation}\label{eq:scroto}
\frac{1}{N^2} \hat \bE^{u,m}\left[\log \bbE Z^{\gb,\go,\hat \phi, u, m}_{N,h}\right]-f(m)
\, \ge \, h  \, \bE^{u,m} \left[ \gd_0\right] -f(m)
\, \ge\,  \frac{h}{2} m^2 |\log m|^2\, ,
\end{equation}
where the first inequality is Jensen's.
Therefore we are left  with estimating 
\begin{equation}
\label{eq:repc1}
\frac{1}{N^2}\hat \bE^{u,m} 
  \log
\left \langle
\exp\left(2\gb^2 \sum_{x\in  \tilde \gL_N} \delta^{(1)}_x \delta^{(2)}_x
\right) 
 \right \rangle ^{\otimes 2}_{N,u,m, h; \hat \phi}\, ,
 \end{equation}
 from above. The first point to remark is that, thanks to the choice of $N$,
 the argument of the exponential is bounded by $2$.  This way at the cost of loosing a multiplicative constant, 
 the exponential and $\log$ essentially cancels. More precisely, by Lemma \ref{lemmeduX}, the expression in \eqref{eq:repc1} is bounded by $e^2-1$ times
 \begin{equation} 
 \frac{2\gb^2}{N^2}\hat \bE^{u,m} 
\left \langle
\sum_{x\in  \tilde \gL_N} \delta^{(1)}_x \delta^{(2)}_x
\right \rangle ^{\otimes 2}_{N,u,m, h; \hat \phi}\, =\, \frac{2\gb^2}{N^2} \sum_{x\in  \tilde \gL_N} \hat \bE^{u,m} 
\left(  \left \langle
 \delta_x
\right \rangle _{N,u,m, h; \hat \phi}^2\right).
 \end{equation}
 With our choice of $h= o(\gb^2)$ we see that
 that the Radon Nikodyn derivative in \eqref{eq:rnder} is bounded above and below uniformly in $\beta$ small and we can even 
 replace  $\left \langle
 \cdot
\right \rangle _{N,u,m, h; \hat \phi}$ with the original measure at the expense of a multiplicative constant that can be chosen arbitrarily close to $1$.
Hence
\begin{equation}
\label{eq:thingtoestimate}
\frac{1}{N^2}\hat \bE^{u,m} 
  \log
\left \langle
\exp\left(2\gb^2 \sum_{x\in  \tilde \gL_N} \delta^{(1)}_x \delta^{(2)}_x
\right) 
 \right \rangle ^{\otimes 2}_{N,u,m, h; \hat \phi}\, \ge \,  \frac{2(e^2-1)\gb^2}{N^2} \sum_{x\in  \tilde \gL_N} \hat \bE^{u,m} 
\left( \bE_N^{u,m,\hat \phi}  
 \delta_x \right)^2\,,
\end{equation}
where of course $ {\gb^2}/{N^2}=\gb^4$.

\medskip

\begin{lemma}
\label{th:lemmar}
For any $\gep>0$, there exists a constant $C_\gep>0$ such that  for all $N$, for all $m\ge m_0(N)$  and all $x\in  \tilde \gL_N$
\begin{equation}
\label{eq:lemmar}
\hat \bE^{u_m,m} 
\left( \bE_N^{u_m,m,\hat \phi}  
 \delta_x \right)^2\, \le\,  C_\gep m^2 | \log m |^2  \left( d_N(x)+1 \right)^{-1+\gep}.
\end{equation}
where here $d_N(x)$ is the distance of $x$ to $\partial\gL_N$.
\end{lemma}
\medskip

In view of \eqref{eq:scroto} and \eqref{eq:thingtoestimate}, 
Lemma~\ref{th:lemmar} directly implies the result we are after. In fact
for $b=3-2\gep$ we have
\begin{multline}
\label{eq:beforeadd-forrev}
\frac{1}{N^2} \hat \bE^{u,m}\left[\log \bbE Z^{\gb,\go,\hat \phi, u, m}_{N,h}\right]-  \frac{1}{N^2}\hat \bE^{u,m} 
  \log
\left \langle
\exp\left(2\gb^2 \sum_{x\in  \tilde \gL_N} \delta^{(1)}_x \delta^{(2)}_x
\right) 
 \right \rangle ^{\otimes 2}_{N,u,m, h; \hat \phi} -f(m) \\
 \ge \frac{h}{2} m^2 |\log m|^2- 2(e^2-1) \frac{\gb^2}{N^2} \sum_{x\in  \tilde \gL_N} \hat \bE^{u,m} 
\left( \bE_N^{u,m,\hat \phi}  
 \delta_x \right)^2 >\left(\frac 12 \gb^{3-2\gep}- C_{\gep, d} \gb^{3-\gep}\right) m^2 |\log m|^2\, ,
\end{multline}
where $C_{\gep, d}$ is a positive constant, where in the last step we have used the bound
\begin{equation}
\label{eq:added-forrev}
\sum_{x\in  \tilde \gL_N}\frac 1 {\left( d_N(x)+1 \right)^{1-\gep}}\, \le \, 4N \sum_{n=1}^N
\frac 1{n^{1-\gep}} \, \le \, \frac{4}\gep N^{1+\gep} \,= \, \frac{4}\gep \gb^{-1-\gep}  .
\end{equation}
 Since
for $\gb$ sufficiently small the right-hand side  in \eqref{eq:beforeadd-forrev} is positive, we are done.
\qed

\medskip

\noindent
{\it Proof of Lemma~\ref{th:lemmar}.}
First of all remark that the result is trivial for $d_N(x)=O(1)$ because
by \eqref{eq:Egd0.2} and \eqref{eq:trf45}
\begin{equation}
\hat \bE^{u_m,m} 
\left( \bE_N^{u_m,m,\hat \phi}  
 \delta_x \right)^2\le \hat \bE^{u_m,m} 
\left( \bE_N^{u_m,m,\hat \phi}  
 \delta_x \right)=\bE^{u_m,m} \delta_x =O(m^2 | \log m |^2).
\end{equation}
Hence, it suffices to prove \eqref{eq:lemmar} for $x$ such that $d_N(x)$ is larger than a constant that may depend
on $\gep$. 

At this point we
note that the variable $\phi_x$ can be written  the sum of two independent Gaussian variables
\begin{equation}
\label{eq:go12}
\gp_x\, :=\, \bE_N^{u,m,\hat \phi} \phi_x\ \ \ \ \text{ and } \ \ \ \
\psi_x\, :=\, \phi_x -\gp_x\, .
\end{equation}
The variance of $\psi_x$ tends, as $m\searrow 0$, to $G_{\gL_N}(x,x)$  
for which \eqref{eq:G2est} holds, that is 
$G_{\gL_N}(x,x)\, \ge \, (1-(\gep/2)) \frac 1 {2\pi} \log d_N(x)$ for $x$ away from the boundary (how far depends
just on $\gep$ and not on $N$). Hence 
$$\eta^2:=\text{Var}(\psi_x)\ge (1-\gep) \frac 1 {2\pi} \log (x)$$ 
and 
since we are performing the estimates by sending first $m$ to zero
we are effectively performing our estimates in the regime
\begin{equation}
\sigma^2- \eta^2\,\gg\,  \eta^2\, \gg 1\, .
\end{equation}

Recall that $g_{\gs^2}(\cdot)$ is the density of a centered Gaussian variable with variance $\gs^2$.
One has
\begin{equation}
\hat \bE^{u,m} 
\left( \bE_N^{u,m,\hat \phi}  
 \delta_x \right)^2=
 \int_{-\infty}^\infty g_{\gs^2-\eta^2}(s) \left( 
\int_{u-1-s}^{u+1-s} g_{\eta^2} (t) \dd t
\right)^2\dd s \, ,
\end{equation}
and to this expression we can directly apply the next lemma.

\medskip

\begin{lemma}
\label{th:gs1gs2}
Recall that $\gs=\gs_m$. 
For every $\gep\in (0,1)$ there exists $\eta_0>0$ such that for every
$\eta> \eta_0$ there exists $m_0>0$ such that for every $m \in (0, m_0)$ 
we have $\gs>\eta$ and
\begin{equation}
\label{eq:gs1gs2}
\int_{-\infty}^\infty g_{\gs^2-\eta^2}(s) \left( 
\int_{u-1-s}^{u+1-s} g_{\eta^2} (t) \dd t
\right)^2\dd s \, \le \,  2C' m^2 \vert \log m \vert ^2 \exp\left( -2\pi (1-\gep) \eta^2 \right)\, .
\end{equation}
\end{lemma}

\medskip

It is now just a matter of observing that
\begin{equation}
 \exp\left( -2\pi (1-\gep) \eta^2 \right)\,\le \, d_N(x)^{-(1-\gep)^2}\, ,
\end{equation}
and we are done.
\qed

\subsection{Proof of Lemma \ref{th:gs1gs2}}
   In \eqref{eq:gs1gs2} we consider separately the case of $s$ larger or smaller
than $u-\eta$: note that $u-\eta \sim u$ in the limit that suffices to consider to establish the result, i.e.
$\gs \gg \eta \gg 1$, even if, at this stage, we cannot replace $u$ with $u-\eta$. The choice of $u-\eta$ is arbitrary
in the sense that $u-c\eta$, $c\ge 1/\sqrt{2}$, would do.
\medskip

We start with considering $s\ge u-\eta$ and we have
\begin{equation}
\label{eq:gs1gs2.1}
\int_{u-\eta}^\infty g_{\gs^2-\eta^2}(s) \left( 
\int_{u-1-s}^{u+1-s} g_{\eta^2} (t) \dd t
\right)^2\dd s \, \le \, 
\int_{u-\eta}^\infty g_{\gs^2-\eta^2}(s) \dd s
\,=\, P\left( \cN \ge \frac{u-\eta}{\sqrt{\gs^2-\eta^2}}
\right) .
\end{equation}
Since $P(Z\ge t)\le \frac 1{\sqrt{2\pi} t} \exp(-t^2/2)$ for every $t>0$ we can continue 
the chain of inequality in \eqref{eq:gs1gs2.1} by 
\begin{equation}
P\left( \cN \ge \frac{u-\eta}{\sqrt{\gs^2-\eta^2}}
\right)\, \le \, \sqrt{\frac 2\pi } \frac \gs u \exp\left(-\frac{(u-\eta)^2}{2(\gs^2-\eta^2)}\right)\, ,
\end{equation}
for $m$ such that $u\ge 2 \eta$.
By recalling that $\lim_m u/\gs^2$ is a positive constant  we see that
\begin{equation}
\label{eq:int6-1}
\frac{(u-\eta)^2}{2(\gs^2-\eta^2)}\, =\, \frac{u^2}{2 \gs^2} - \frac{\eta u}{2 \gs ^2}+
\frac{u^2\eta^2}{2 \gs^4}+ O\left(\frac {\eta^4}{\gs^4}\right) \, .
\end{equation}
More precisely, since
$\lim_m u/\gs^2=\sqrt{8\pi}$, we see that for 
every $q<1$, but we  choose $q\in (1/2,1)$, we have that
\begin{equation}
- \frac{\eta u}{2 \gs ^2}+
\frac{u^2\eta^2}{2 \gs^4}\, \ge \, 4\pi q \eta^2\, ,
\end{equation}
for $\eta$ sufficiently large and $m$ sufficiently small. Therefore, by choosing if needed  $m$ even smaller so that
$\gs/u \sim 1/(2 \sqrt{\vert \log m\vert })$ is smaller than $1/\sqrt{\vert \log m\vert }$ and that
the $O[(\eta/\gs)^4]$ term in \eqref{eq:int6-1} can be absorbed by replacing $q$ with a smaller value still larger than $\frac 12$, we have 
\begin{equation}
\label{eq:gs1gs2.1-1}
\begin{split}
\int_{u-\eta}^\infty g_{\gs^2-\eta^2}(s) \left( 
\int_{u-1-s}^{u+1-s} g_{\eta^2} (t) \dd t
\right)^2\dd s \, &\le \, 
\frac  1{\sqrt{\vert \log m \vert}} \exp\left(-\frac{u^2}{2\gs^2}\right) \exp\left(-4\pi q \eta^2\right)
\\
&= \,  C' m^2 \vert \log m \vert ^2  \exp\left(-4\pi q \eta^2\right)\, ,
\end{split}
\end{equation} 
where in the last step we have used \eqref{eq:trf45}. 
In view of what we need to establish, that is \eqref{eq:gs1gs2}, we can move to look for an  upper bound for 
the case $s\le u-\eta$.

\medskip

For $s\le u-\eta$ we use instead
\begin{multline}
\label{eq:ntmtcp}
\int_{u-1-s}^{u+1-s} g_{\eta^2} (t) \dd t \,  =\,
P \left( \cN \in \frac {u-s-1}\eta +\left[ 0, \frac 2 \eta\right] \right)
\, \le \,  P \left( \cN \ge \frac {u-s-1}\eta\right)
\\
\le \frac 1{\sqrt{2\pi}}\frac \eta{u-s-1} \exp\left( -\frac 12 \left( \frac {u-s-1}\eta\right)^2 \right)\, 
\le \frac {\eta}{u-s} \exp\left( -\frac {1}2 \left( \frac {u-s}{\eta_\gep}\right)^2 \right)\, ,
\end{multline}
where $\eta_\gep:= \eta/(1-\gep)$, $\gep\in (0, 1/20)$, and  we have used the bounds on the distribution of 
$\cN$ recalled just after \eqref{eq:gs1gs2.1} (we are choosing $m$ and $1/\eta$ sufficiently small).
Hence we have
\begin{equation}
\label{eq:gs1gs2.2}
\int_{-\infty}^{u-\eta} g_{\gs^2-\eta^2}(s) \left( 
\int_{u-1-s}^{u+1-s} g_{\eta^2} (t) \dd t
\right)^2\dd s \, \le \,\frac{1}{(2\pi)^{3/2}} \frac{\eta^2}{\gs_1}
\int_{-\infty}^{u-\eta}
\frac{\exp\left( -\frac{s^2}{2\gs_1^2}- \frac{(u-s)^2}{\eta_\gep^2}\right)}{(u-s)^2}\dd s\, ,
\end{equation}
where we have introduced $$\gs_1:=\sqrt{\gs^2-\eta^2}.$$ 
Recall that   we look for a result in a regime in which $\gs_1\sim \gs$.
We now reconstruct the square in the term in the exponential and we obtain that the right-hand side of 
\eqref{eq:gs1gs2.2} equals 
\begin{equation}
\label{eq:gs1gs22}
\frac{1}{(2\pi)^{3/2}} \frac{\eta^2}{\gs_1}
\exp\left( -\frac{u^2}{\eta_\gep^2+2\gs_1^2}\right)
\int_{-\infty}^{u-\eta}
\frac 1
{(u-s)^2}
\exp\left( - \frac{\eta_\gep^2+2 \gs_1^2}{2 \eta_\gep^2\gs_1^2} \left( s- \frac{2u \gs_1^2}{\eta_\gep^2 +2 \gs_1^2}\right)^2\right)\dd s\, ,
\end{equation}
Let us introduce
\begin{equation}
\label{eq:7.64}
a_{m, \eta_\gep}\,:=\, u- \frac{2u  \gs_1^2}{\eta_\gep^2+2 \gs_1^2} \stackrel{m\searrow 0} \sim  \sqrt{2\pi} \eta_\gep^2\, ,
\end{equation}
so that the integral term in 
\eqref{eq:gs1gs22} can be rewritten as
\begin{equation}
\label{eq:gs1gs23}
\int_{-\infty}^{a_{m,\eta_\gep}-\eta} \frac 1{(a_{m,\eta_\gep}-s)^2} \exp\left( - \frac{\eta_\gep^2+2 \gs_1^2}{2 \eta_\gep^2\gs_1^2}s^2 \right) \dd s\, ,
\end{equation}
and since $\frac{\eta_\gep^2+2 \gs_1^2}{2 \eta_\gep^2\gs_1^2} \ge  \frac 1{\eta_\gep^2}$ 
we see that 
we can bound 
the expression in \eqref{eq:gs1gs23} by 
\begin{equation}
\label{eq:gs1245}
\frac 1 {\eta_\gep}
\int_{-\infty}^{
\frac{a_{m,\eta_\gep}}{\eta_\gep}
-1+ \gep} \frac 1{
\left(
\frac{a_{m,\eta_\gep}}{\eta_\gep}-s
\right)^2
} 
\exp\left( - s^2 \right) \dd s\, \sim \,
{\sqrt{\pi}} \frac{\eta_\gep}{(a_{m,\eta_\gep})^2}\, \sim\, \, \frac 1 {2\sqrt{\pi}\, \eta_\gep^3} \, , 
\end{equation}
where the asymptotic limit is as $m\searrow 0$ and then $\eta\to \infty$.
Therefore, by \eqref{eq:7.64}--\eqref{eq:gs1245}  and the choice of $\gep<1/20$, we see that 
the integral term in 
\eqref{eq:gs1gs22} is smaller than $1/(4\eta^3)$ for suitably small $m$ and $1/\eta$.
Therefore, going back to \eqref{eq:gs1gs2.2} and \eqref{eq:gs1gs22}
we see that
\begin{equation}
\label{eq:gs1gs2.5}
\int_{-\infty}^{u-\eta} g_{\gs^2-\eta^2}(s) \left( 
\int_{u-1-s}^{u+1-s} g_{\eta^2} (t) \dd t
\right)^2\dd s \, \le \,
\frac{1}{4(2\pi)^{3/2}} \frac{1}{\eta\gs_1}
\exp\left( -\frac{u^2}{\eta_\gep^2+2\gs_1^2}\right)
\, ,
\end{equation}
and in turn, with $c_\gep:=2-(1-\gep^{-2}\ge 1- 3\gep $ ($\gep<1/20$), we have
\begin{multline}
\exp\left( -\frac{u^2}{\eta_\gep^2+2\gs_1^2}\right) \, =\, \exp\left( -\frac{u^2}{2 \gs^2}\right)
\exp\left(- \frac{c_\gep\eta^2  u^2}{2 \gs^2\left(2\gs^2-c_\gep\eta^2 \right))}\right)
\\
=\, C' m^2\vert \log m \vert ^{5/2} \exp\left(- \frac{c_\gep \eta^2 u^2}{2 \gs^2(2\gs^2-c_\gep\eta^2)}\right)
\, \le C' m^2\vert \log m \vert ^{5/2} \exp\left(- (1-3\gep)\eta ^2\left(\frac{ u}{2 \gs^2}\right)^2\right)\, ,
\end{multline}
where we have used \eqref{eq:trf45}. Finally, since $(\frac{ u}{2 \gs^2})^2 > 2\pi (1-\gep)$ for $m$ small
and recalling also \eqref{eq:gsm2}, going back to
\eqref{eq:gs1gs2.5} we obtain
\begin{equation}
\label{eq:ad15}
\int_{-\infty}^{u-\eta} g_{\gs^2-\eta^2}(s) \left( 
\int_{u-1-s}^{u+1-s} g_{\eta^2} (t) \dd t
\right)^2\dd s \, \le \, C' m^2 \vert \log m\vert^2
 \frac{\exp(- 2\pi (1-4\gep) \eta^2)}{\eta} 
\, ,
\end{equation}
which, together with
\eqref{eq:gs1gs2.1-1},
yields \eqref{eq:gs1gs2} and the proof of Lemma~\ref{th:gs1gs2} is complete.
\qed

\medskip

\begin{rem}\label{rem:cop2}
The warm up argument of Section~\ref{sec:rough} does not yield interesting information in the case of the co-membrane, simply because  
the probability of visiting the lower half plane is $1/2$ for a centered field and the quadratic term in the replica computation is too large. 
But the arguments of Section~\ref{sec:lessrough} have a chance to be generalized because we introduce a shift in the field that makes
the probability of visiting the lower half plane small. And they 
do generalize, giving the analog of Theorem~\ref{th:d=2} for the co-membrane model:
 let us quickly see why. The estimate \eqref{eq:Egd0.1} becomes
\begin{equation}
\bE ^{u,m}[\gD_0]\,=\, \int_{-\infty}^0  \frac 1{\sqrt{2\pi \gs^2}} \exp\left( -\frac{( x-u)^2}{2 \gs^2} \right)  \dd x
\, =\, \bP\left( \gs \cN >u \right) \sim \frac 1{\sqrt{2\pi}} \frac{\gs} u \exp\left( -\frac{u^2}{2 \gs ^2} \right)\, ,
\end{equation}  
and $u/\gs \sim 2 \sqrt{\vert \log m \vert }$ and, apart for the value of $C'$ (in \eqref{eq:trf45}), 
we have the analog of \eqref{eq:Egd0.2}.  
In the remainder of the proof in reality we estimate the probabilities either by replacing $[-1,1]$ with $\bbR$, this is the case in \eqref{eq:gs1gs2.1}, or with $(-\infty,1]$, see \eqref{eq:ntmtcp}. Therefore the proof can be adapted to the 
co-membrane set-up.
\end{rem}

\medskip

\begin{rem}
\label{rem:whendis2}
To complete the discussion in Section~\ref{rem1} we observe that a lower bound on the non disordered free energy $\tf(h)$ for $d=2$
can be obtained by first localizing the $\phi$ field,  by introducing a mass, so that we can apply the approach in Section~\ref{rem1}
and then by optimizing the choice of $m$ as a function of $h$.
More precisely, by Corollary~\ref{massic} (for $\gb=0$) and by applying the same argument as for the lower bound in Section~\ref{rem1}
we obtain
\begin{equation}
\tf(h)\, \ge\, h \sqrt{\frac 2 \pi} \frac 1{\gs_m} -f(m)\, .
\end{equation}
The two terms in the right-hand side are then estimated for $m$ small by  \eqref{eq:gsm2} 
and Lemma~\ref{radonnic} to obtain that for every positive $c <1$ (that can be chosen arbitrarily close to $1$)
we have
\begin{equation}
\tf(h)\, \ge\, \frac 2 c \frac{h}{\sqrt{\vert \log m \vert }} -2\pi c\, m^2 \vert \log m \vert\, ,
\end{equation}
for $m$ smaller than a constant that depends on the choice of $c$.
It is now sufficient to choose $m$ equal to $h$ to a power larger than $1/2$, for example $m=h^{3/4}$
to obtain that there exists $C_2>0$ such that
\begin{equation}
\tf(h)\, \ge\, C_2 \frac h {\sqrt{\vert \log h \vert }}\,,
\end{equation}
which should be compared to \eqref{eq:fact2.4}.
\end{rem}

\appendix 

\section{Replica coupling: proof of Lemma~\ref{th:replica}}
\label{sec:replica}
The argument follows closely the main argument in \cite{cf:Trep}.
We do not detail the proof of Lemma~\ref{th:replica2} which is extremely similar (and simpler).
Let us fix $\hat \phi$.
Given an event $A\subset \bbR^{\bbZ^d}$ (in the specific application $A$ is measurable with respect to $\{\phi_x: \, x \in \mathring{\gL}_N\}$, but at this stage we just require $\bP_N^{\hat \phi}(A)>0$),  
we write
\begin{multline}
\label{eq:app1}
F_{N} (\gb, h; A)\, :=\, 
\frac 1{N^d} \bbE \left[ \log \bE_N^{\hat \phi} \left[
\exp\left( \sum_{x\in  \gL} \left(\gb \go_x-\frac {\gb^2} 2+h\right)\delta_x\right) ; \, A 
\right] \right]
\\
=\, 
\frac 1{N^d} 
 \log \bE_N^{\hat \phi} \left[
\exp\left( h \sum_{x\in  \gL} \delta_x\right) 
\right]
+ R_{N,h}(\gb; A)\, ,
\end{multline}
where $\gL \subset \gL_N$ (for this proof it suffices to consider $\gL=\gL_{N_1,N_0}$, but this specific choice is irrelevant at this stage) and 
\begin{equation}
R_{N,h}(\gb; A)\,:=\,
\frac 1{N^d} \bbE  \left[ \log \left \langle
\exp\left( \sum_{x\in  \gL}  \left(\gb \go_x-\frac {\gb^2} 2\right)\delta_x\right) ; \, A 
\right\rangle_{N,h; \hat \phi} \right]\, .
\end{equation}
Of course 
\begin{equation}
\left \langle \cdot \right \rangle _{N,h; \hat \phi} \, :=\, 
\frac{
\bE_N^{\hat \phi} \left[ \, \cdot \,  \exp\left(h \sum_{x\in  \gL} \gd_x \right) \right] }
{\bE_N^{\hat \phi} \left[   \exp\left(h \sum_{x\in  \gL} \gd_x \right) \right]} \, .
\end{equation}
By (Gaussian) integration by parts -- the basic formula being $E[\cN F(\cN)]=E[F'(\cN)]$ which holds for $F\in C^1$ with a suitable growth condition at $\pm \infty$ -- we obtain
\begin{multline}
\label{eq:ibp1}
\frac{\dd }{\dd t}\left(-R_{N,h}(\sqrt{t}\gb; A)\right)\,=\, 
\\
\frac{\gb ^2}{2N^d}
\sum_{x \in \gL} \bbE 
\left[
\left(
\frac{
\left \langle
\gd_x \exp\left( \sum_{x\in  \gL}  \left(\sqrt{t}\gb \go_x- t\frac {\gb^2} 2\right)\delta_x\right) ; \, A \right \rangle _{N,h; \hat \phi}
}{ 
\left \langle
\exp\left( \sum_{x\in  \gL}  \left(\sqrt{t}\gb \go_x-t \frac {\gb^2} 2\right)\delta_x\right) ;\, _A \right \rangle _{N,h; \hat \phi}
}
\right)^2
\right]
\, .
\end{multline}
At this point we introduce
\begin{equation}
\psi_{N, h}(t, \gl, \gb; A)\,:=\,
\frac1{2 N^d}  \bbE  
  \log
\left \langle
\exp\left( H_N
\right) 
; \, {A^2} \right \rangle ^{\otimes 2}_{N,h; \hat \phi}\, ,
\end{equation}
where 
\begin{equation}
H_N\, :=\, \sum_{x\in  \gL}  \left(\sqrt{t}\gb \go_x- t\frac {\gb^2} 2\right)\left(\delta^{(1)}_x+ \delta^{(2)}_x\right)
+ \gl \gb^2 \sum_{x\in  \gL} \delta^{(1)}_x \delta^{(2)}_x\, .
\end{equation}
In particular
\begin{equation}
\label{eq:psi2c}
\psi_{N, h}(0, \gl, \gb; A)\,=\, 
\frac1{2 N^d}   \bbE
  \log
\left \langle
\exp\left( \gl \gb^2 \sum_{x\in  \gL} \delta^{(1)}_x \delta^{(2)}_x
\right) 
; \, {A^2} \right \rangle ^{\otimes 2}_{N,h; \hat \phi}\, ,
\end{equation}
and 
\begin{equation}
\label{eq:psiR}
\psi_{N, h}(t, 0, \gb; A)\,=\,
R_{N,h}(\sqrt{t}\gb; A)\, .
\end{equation}
Again by integration by parts we obtain
\begin{multline}
\frac{\dd}{\dd t} \psi_{N, h,u}(t, \gl, \gb; A)\,\le\,
\frac {\gb^2}{2N^d}  \bbE  \hat \bE ^u 
\left[\frac{ \sum_{x\in  \gL} \delta^{(1)}_x \delta^{(2)}_x
\left \langle
\exp\left( H_N
\right) 
; \, {A^2} \right \rangle ^{\otimes 2}_{N,h; \hat \phi }
}
{
\left \langle
\exp\left( H_N
\right) 
; \, {A^2} \right \rangle ^{\otimes 2}_{N,h; \hat \phi }
}
\right]
\\
=\, \frac{\dd}{\dd \gl} \psi_{N, h}(t, \gl, \gb;A)\,,
\end{multline}
where the inequality comes from neglecting the (negative) term coming from the derivative of the denominator.
We therefore see that $(\dd / \dd s) \psi_{N, h}(t-s, \gl+s, \gb;A) \ge 0$ for $s \in [0,t]$ and therefore
\begin{equation}
\label{eq:transp}
\psi_{N, h}(t, \gl, \gb;A) \, \le \, \psi_{N, h}(0, \gl+t, \gb;A)\, .
\end{equation}
Now we go back to \eqref{eq:ibp1} and we remark that 
\begin{equation}
\label{eq:di1}
\frac{\dd }{\dd t}\left(- R_{N,h}(\sqrt{t}\gb; A)\right)\,=\,
\frac{\dd}{\dd \gl} \psi_{N, h}(t, \gl, \gb;A)\Big \vert_{\gl=0}\,,
\end{equation}
and  for $t \in [0,1]$
\begin{equation}
\label{eq:di2}
\begin{split}
\frac{\dd}{\dd \gl} \psi_{N, h}(t, \gl, \gb;A)\Big \vert_{\gl=0}\,
&\le \, \frac{\psi_{N, h}(t, 2-t, \gb;A)- R_{N,h}(\sqrt{t}\gb; A)}{2-t}
\\
& \le \, 
\psi_{N, h}(0, 2, \gb;A)- R_{N,h}(\sqrt{t}\gb; A)\, ,
\end{split}
\end{equation}
where  the first bound  follows by convexity of $\psi_{N, h}(t, \cdot, \gb;A)$ and \eqref{eq:psiR},
while for the second we use $2-t \ge 1$, non negativity of the numerator  and \eqref{eq:transp}.
By \eqref{eq:ibp1} and by integrating the differential inequality obtained by combining  \eqref{eq:di1} and \eqref{eq:di2} 
one obtains 
\begin{equation}
\label{eq:app2}
\frac 1{N^d} 
\log \left \langle \ind_A \right\rangle_{N, h, \hat \phi} \ge 
R_{N,h}(\gb; A)\ge  \frac 1{N^d}  
\log \left \langle \ind_ A \right\rangle_{N, h, \hat \phi} - \left(e-1 \right) \psi_{N, h}(0, 2, \gb;A)\, .
\end{equation}
Therefore, since
\begin{equation}
 \left \langle \ind_ A \right\rangle_{N, h, \hat \phi}\, =\,
\frac{\bE_N^{\hat \phi} \left[
\exp\left( h \sum_{x\in  \gL} \delta_x\right); \, A 
\right]}{\bE_N^{\hat \phi} \left[
\exp\left( h \sum_{x\in  \gL} \delta_x\right) 
\right]}\, ,
\end{equation}
by putting \eqref{eq:app1} and \eqref{eq:app2} together we obtain
\begin{multline}
\label{eq:app3}
F_{N} (\gb, h; A)\, \ge\, 
\frac 1{N^d} 
 \log \bE_N^{\hat \phi} \left[
\exp\left( h \sum_{x\in  \gL} \delta_x\right) ; \, A
\right]
 - \left(e-1 \right) \psi_{N, h}(0, 2, \gb;A)\, .
\end{multline}
The expressions in the statement of Lemma~\ref{th:replica} are retrieved from
\eqref{eq:app3} by setting $\gL=\gL_{N_1,N_0}$, $A=A_\kappa$ and by replacing  $e-1$ by 
the larger value $2$. 
\qed

\section{Proof of Proposition~\ref{th:conjection}}
\label{sec:app2}
We give the proof in four  steps.

\subsubsection{Step 1: upper bound on the contact density}
The fractional moment method yields also a quantitative  upper bound on 
the contact fraction: if for $c>0$ we introduce the event $B_{N,c}=\{\sum_{x \in \tilde \gL_N} \gd _x 
\ge c N^d\}$, as in  \eqref{eq:RNderfm}   and \eqref{eq:CSfm}, but this time with $\alpha=2h/ \gb$ (compare with \eqref{eq:nGfma}
in the Gaussian case)
\begin{multline}
\left(\bbE \left[ \sqrt{Z_{N, h}^{\gb, \go}\left( B_{N, Ch N^d}\right)}\right]\right)^2\, \le \, 
\tilde \bbE \left[ Z_{N, h}^{\gb, \go}\left( B_{N, Ch N^d}\right)\right]
\exp\left(4\frac{h^2}{\gb^2}N^d \right)
\\
 \le \, Z_{N, -h}\left( B_{N, Ch N^d}\right) \exp\left(4\frac{h^2}{\gb^2}N^d \right)\, \le \, 
 \exp\left(- \left( C -\frac 4{\gb^2}\right) h^2 N^d \right)\,,
\end{multline}
so that if we choose $C=6/\gb^2$, by Markov inequality, we have 
\begin{equation}
\label{eq:sjdg2}
\bbP \left( Z_{N, h}^{\gb, \go}\left( B_{N, Ch N^d}\right) \ge \exp\left(-2 \frac{h^2}{\gb^2} N^d\right) \right)\, \le \, 
\exp\left(-\frac {h^2}{\gb^2} N^d \right)\, .
\end{equation}
We now focus on   $\bP_{N, h}^{\gb, \go}(B_{N, Ch N^d})=Z_{N, h}^{\gb, \go}\left( B_{N, Ch N^d}\right)/Z_{N, h}^{\gb, \go}$
which we are going to bound from above simply by one for $\go$ in the event whose probability is estimated 
in \eqref{eq:sjdg2} and otherwise we use, as in 
 Section~\ref{sec:puremodel} the {\sl entropic repulsion} estimate \cite{cf:LM}
$\inf_\go Z_{N, h}^{\gb, \go}\ge \exp(-r(N))$ with $r(N)=o(N^d)$. This yields
\begin{equation}
\bbE \bP_{N, h}^{\gb, \go}\left( B_{N, 6h N^d/\gb^2}\right) \, \le \, 
\exp\left(-\frac{h^2}{\gb^2} N^d \right) + \exp\left(r(N) - 2\frac{ h^2}{\gb^2} N^d \right) \, .
\end{equation}
The punchline of step 1 is that for every $h>0$ 
\begin{equation}
\label{eq:Bstep1}
\lim_{N \to \infty}
 \bbE \bP_{N, h}^{\gb, \go}\left( B_{N, 6h N^d/\gb^2}\right) \, = \, 0\,.
\end{equation}


\subsubsection{Step 2: neighbor averages below a threshold for too many sites implies high contact fraction}
Set $\overline{\phi}_x=
(1/2d) 
\sum_{y:\, \vert y-x \vert =1} \phi_y$
 and consider the event that on the even sites  there is  
 a density of at least $\gep/2$  of the $\overline{\phi}$ variables that in absolute value are smaller than $\sqrt{1/(4d)\log(1/h)}$:
\begin{equation}
\label{eq:BF_k}
F_{N, \gep}\, =\, \left\{ \phi:\, \sum_{  x \in \mathring{\gL}_N:\, 
x \text{ even }} \ind_{\left(- \sqrt{\frac 1 {4d} \log (1/h)},  \sqrt{\frac 1 {4d}\log (1/h)}\right)}\left(\overline{\phi}_x\right) \, \ge \, \frac \gep 4 N^d
\right\}\, .
\end{equation}
  We aim at showing that there exists $h_0$ such that for $h \in (0, h_0)$
\begin{equation}
\label{eq:Bstep2}
\lim_N \bbE
\bP_{N,h}^{\gb , \go} \left(F_{N,  \gep}\right)
\, =\, 0\, .
\end{equation}
By \eqref{eq:Bstep1}, \eqref{eq:Bstep2} is implied by
\begin{equation}
\label{eq:Bstep2-1}
\lim_N \bbE
\bP_{N,h}^{\gb , \go} \left(B_{N, 6h N^d/\gb^2}^\complement \cap F_{N,  \gep}\right)
\, =\, 0\, ,
\end{equation}
and, by writing once again the probability as ratio of partition functions, by using the lower bound on the denominator given 
by the entropic repulsion estimate and by performing the $\bbP$ expectation, we see that
\begin{equation}
\bbE
\bP_{N,h}^{\gb , \go} \left(B_{N, 6h N^d/\gb^2}^\complement \cap F_{N, \gep}\right)
\,\le \, \exp(-r(n)+h N^d) \bP_{N} \left(B_{N, 6h N^d/\gb^2}^\complement \cap F_{N, \gep}\right)\, ,
\end{equation}
so that we are done if we  show that for a $c\in (0,1)$ and an $h_0>0$ we have that for $h \in (0, h_0)$
\begin{equation}
\label{eq:tbs87}
\bP_{N} \left(B_{N, 6h N^d/\gb^2}^\complement \cap F_{N,  \gep}\right)\, \le \, \exp\left(-h^c N^d\right)\, .
\end{equation}
 For this use the fact that the event $B_{N, 6h N^d/\gb^2}$ contains the event that 
 the $6h N^d/\gb^2$ (or more) contacts are all on the even sites. By conditioning  on the odd sites and by using the Markov property -- remark that 
 $F_{N, \gep}$ is measurable with respect to the $\gs$-algebra of the odd variables --
 one realizes that the random variables $\gd_x$, $x$ even, are independent Bernoulli variables of parameter
 \begin{equation}
 p_x\, :=\, P\left( \frac{1}{\sqrt{2d}}\cN + \overline \phi _x \in [-1, 1] \right) \, .
 \end{equation}
 If $ \vert \overline \phi _x \vert \le  \sqrt{1/(4d) \log(1/h)}$ by the standard Gaussian tail estimate -- one can use   \eqref{eq:asymptZ} even if what we claim here is substantially rougher -- 
 we see that
for $h$ sufficiently small
 \begin{equation}
 p_x \, \ge \, \exp \left( -\frac 1 {2}   \log \left (\frac 1 h\right) \right)\, =\, h^{\frac 1 {2}}\, =:\,p\, =p(h) \, .
 \end{equation}
 So, once $\phi \in F_{N,  \gep}$ is chosen and hence the set (of at least $\gep N^d /4$ sites $x$) on which $p_x\ge p$
 is determined, we are simply left with a Large Deviation upper bound on a binomial random variable $B(n, p)$: 
 it is a well known fact, direct consequence of the exponential form of the Markov inequality, 
 that the probability that for $\gD \in (0, 1)$
 \begin{equation}
  P\left( B(n, p)\le p \gD n \right)\le  \exp(-n f(p, \gD))
 \end{equation}
where 
  $$f(p, \gD):= \gD p \log \gD +(1-\gD p) \log ((1-\gD p)/(1-p)).$$
 If $\gD=1/2$
 $$f(p, 1/2)\stackrel{p\to 0}{\sim}\frac 1 2(1-\log 2)p,$$ so that $f(p, 1/2)\ge p/10$ for $p$ sufficiently small.
Therefore for $p(h)=h^{1/2}$  this implies in a rather direct way that 
\begin{equation}
\label{eq:fortbs87}
\bP_{N} \left(B_{N,\frac{\gep}{8}h^{\frac 1 2} N^d}^\complement \cap F_{N,  \gep}\right)\, \le \, \exp\left(-h^{\frac 12} \frac{\gep}{40} N^d\right)\, , 
\end{equation}
 and this implies \eqref{eq:tbs87} for any $c\in (1/2, 1)$ and $h$ sufficiently small. 
 
\subsubsection{Step 3: the Gaussian Hamiltonian cannot be too large under the pinning model}
If we set
\begin{equation}
H_N(\phi)\, :=\, \sumtwo{(x,y)\in (\gL)^2 \setminus (\partial \gL)^2 }{x\sim y}\frac{ (\phi_x-\phi_y)^2 }{2}\, =:\,
\left( \phi, A_N \phi\right)_N ,
\end{equation}
where $A_N$ is a positive definite symmetric $(N-1)^d\times(N-1)^d$ and $(\cdot, \cdot)_N$ is the scalar product 
on $\bbR^{N-1}$. In fact $A_N$ is just discrete Laplacian operator with zero Dirichlet boundary condition which is the generator of the simple random walk killed upon hitting 
$\partial  \gL_N$.
We have 
\begin{equation}
\label{eq:JFT6}
\begin{split}
\bbE \bP_{N, h}^{\gb , \go}\left( H_N(\phi) > C N^d \right)\, &\le \, 
e^{r_N} \bbE
 \bE_N \left[ \exp \left(\sum_{x\in \tilde \gL _N} \left(\gb \go_x - \frac{\gb^2} 2 +h \right) \gd_x\right); \, 
H_N(\phi) > C N^d \right]
\\
& \le \, e^{r_N + hN^d} \bP_{N}\left( H_N(\phi) > C N^d \right)\, ,
\end{split}
\end{equation}
where in the first step we have applied once again  the entropic repulsion bound  ($r_N=o(N^d)$) and 
in the second step we have performed the expectation with respect to the disorder and then we have bounded in the obvious way  the pinning part. We are left with estimating the remaining probability, which is a Gaussian computation:
since for $\gl < 1/2$
\begin{multline}
\bE_N \left[ \exp( \gl H_N(\phi))\right] \, =\,
\int_{\bbR ^{\mathring{\gL}_N}} \sqrt{
\frac{
\text{det}\left(A_N\right)}{(2\pi)^{(N-1)^d}} }
\exp \left( -\frac {\left(1-2\gl \right)}2 \left( \phi, A_N \phi\right)_N\right) \prod_{x \in \mathring{\gL}_N} \dd \phi_x
\\ =\, 
 (1-2\gl)^{-\frac 12(N-1)^d}\, ,
\end{multline}
where we have used det$(c A_N)=c^{(N-1)^d}$det$(A_N)$, $c>0$.
By applying the Markov inequality with $\gl=3/8$
we obtain
$$ \bP_{N}\left( H_N(\phi) > C N^d \right)\le 2^{(N-1)^d} \exp( -(3/8) C N^d.$$ 
Therefore, by recalling \eqref{eq:JFT6}, we have 
\begin{equation}
\label{eq:Bstep3}
\bbE \bP_{N, h}^{\gb , \go}\left( H_N(\phi) > C N^d \right)\, \le \, \exp\left(-\frac C4 N^d\right)\, ,
\end{equation}
for any $C>8 \log 2$ (for example $C=6$), $h$ small and $N$ sufficiently large.

\subsubsection{Step 4: conclusion}
It is now a matter of putting together \eqref{eq:Bstep2} and \eqref{eq:Bstep3}. 
Let us first observe that
\begin{equation}
\label{eq:Bs4.1}
\bbE \bP_{N, h}^{\gb , \go} 
\left( F_{N, \frac 12, \gep}^\complement \cap \left\{ \phi: \, 
\sumtwo{x\in \mathring{\gL}_N}{x \text{ even}}
\ind_{\left\{\vert \phi_y\vert \le  \sqrt{ \frac 1 {8d}\log (1/h)} \} \textrm{ for at least a } y \textrm{ s.t. } 
\vert y-x \vert=1 
\right \}} \ge \frac {\gep N^d}2 
\right\} 
\right)\, ,
\end{equation}
tends to zero as $N \to \infty$ for $h$ sufficiently small. 
This is because on the event whose probability is computed in \eqref{eq:Bs4.1} there will be at least $\gep N^d /4$ even sites $x$ on which 
$\vert \overline{\phi}_x\vert >
\sqrt{\frac 1 {4d}\log(1/h)} $ and at least for one of the neighboring sites $y$ we have  
$\vert \phi_y\vert \le \sqrt{ \frac 1 {8d}  \log (1/h)}$, while instead  there is another neighbor $y'$ of $x$ for which 
$\vert \phi_{y'} \vert > \sqrt{\frac 1 {4d} \log(1/h)}$.
Therefore it is not difficult to see that this implies $(\phi_y-\phi_x)^2+ (\phi_{y'}-\phi_x)^2 \ge \frac 1 {100 d} \log (1/h)$, and in turn 
$H_N(\phi)\ge  \log (1/h) \gep N^d /(400d)$.
By choosing $h$ sufficiently small we see that the event under analysis becomes a subset of   $\{H_N(\phi)> 6 N^d\}$
and, by \eqref{eq:Bstep3}, we see that \eqref{eq:Bs4.1} tends to zero.
It remains then to repeat the argument (that is even  simpler) to show that also 
\begin{equation}
\label{eq:Bs4.2}
\lim_{N \to \infty}
\bbE \bP_{N, h}^{\gb , \go} 
\left( F_{N, \frac 12, \gep}^\complement \cap \left\{ \phi: \, 
\sumtwo{x\in \mathring{\gL}_N}{x \text{ even}}
\ind_{\left\{\vert \phi_x\vert \le  \sqrt{ \frac{1}{8d}\log (1/h)} \} 
\right \}} \ge \frac {\gep N^d}2 
\right\} 
\right)\, =\, 0\, .
\end{equation}
 If we combine \eqref{eq:Bs4.1} and \eqref{eq:Bs4.2}, by recalling \eqref{eq:Bstep2} we conclude.
\qed

\end{document}